\documentclass[11pt, twoside, leqno]{article}

\usepackage{amsmath}
\usepackage{amssymb}
\usepackage{titletoc}
\usepackage{mathrsfs}
\usepackage{amsthm}
\usepackage{indentfirst}
\usepackage{color}
\usepackage{txfonts}
\usepackage{enumerate}

\allowdisplaybreaks

\newtheorem{theorem}{Theorem}[section]
\newtheorem{lemma}[theorem]{Lemma}
\newtheorem{corollary}[theorem]{Corollary}
\newtheorem{proposition}[theorem]{Proposition}

\theoremstyle{definition}
\newtheorem{remark}[theorem]{Remark}
\newtheorem{definition}[theorem]{Definition}

\numberwithin{equation}{section}

\begin{document}

\title{\vspace{-1.68cm}\bf\Large Matrix-Weighted Besov-Type and Triebel--Lizorkin-Type Spaces III:
Characterizations of Molecules and Wavelets, Trace Theorems,
and Boundedness of Pseudo-Differential Operators and Calder\'on--Zygmund Operators
\footnotetext{\hspace{-0.35cm} 2020 {\it Mathematics Subject Classification}.
Primary 46E35; Secondary 47A56, 42B25, 42B20, 42C40, 35S05, 42B35.\endgraf
{\it Key words and phrases.}
matrix weight,
Besov-type space,
Triebel--Lizorkin-type space,
$A_p$-dimension,
molecule,
wavelet,
trace,
pseudo-differential operator,
Calder\'on--Zygmund operator.\endgraf
This project is supported by
the National Key Research and Development Program of China
(Grant No.\ 2020YFA0712900),
the National Natural Science Foundation of China
(Grant Nos.\ 12371093, 12071197, and 12122102),
the Fundamental Research Funds
for the Central Universities (Grant No. 2233300008),
and the Academy of Finland (Grant Nos.\ 314829 and 346314).}}
\date{}
\author{Fan Bu, Tuomas Hyt\"onen, Dachun Yang\footnote{Corresponding author, E-mail:
\texttt{dcyang@bnu.edu.cn}/{\color{red} December 27, 2023}/Final version.}\ \ and Wen Yuan}

\maketitle

\vspace{-0.8cm}

\begin{center}
\begin{minipage}{13cm}
{\small {\bf Abstract}\quad
This is the last one of three successive
articles by the authors on matrix-weighted Besov-type and Triebel--Lizorkin-type spaces
$\dot B^{s,\tau}_{p,q}(W)$ and $\dot F^{s,\tau}_{p,q}(W)$. In this article,
the authors establish the molecular and the wavelet characterizations of these spaces.
Furthermore, as applications, the authors obtain the optimal boundedness of
trace operators, pseudo-differential operators, and Calder\'on--Zygmund operators
on these spaces. Due to the sharp boundedness of almost diagonal operators
on their related sequence spaces obtained in the second article of this series,
all results presented in this article
improve their counterparts on matrix-weighted Besov and Triebel--Lizorkin spaces
$\dot B^{s}_{p,q}(W)$ and $\dot F^{s}_{p,q}(W)$.
In particular, even when reverting to the boundedness of Calder\'on--Zygmund operator
on unweighted Triebel--Lizorkin spaces $\dot F^{s}_{p,q}$, these results are still better.
}
\end{minipage}
\end{center}


\vspace{0.1cm}

\section{Introduction}

This is the last one of our three successive articles on
matrix-weighted Besov-type and Triebel--Lizorkin-type spaces of
$\mathbb C^m$-valued distributions on $\mathbb R^n$.
We consistently denote by $m$ the dimension of the target space of our distributions and
hence our matrix-weights take values in the space of $m\times m$ complex matrices.

Note that both matrix weights and Besov-type and Triebel--Lizorkin-type spaces
have broad applications, which motivates us to develop a complete real-variable theory of
matrix-weighted Besov-type and Triebel--Lizorkin-type spaces.
Indeed, in our first article \cite{bhyy1} of this series, we introduced
matrix-weighted Besov-type and Triebel--Lizorkin-type spaces
$$\dot A^{s,\tau}_{p,q}(W)\in\left\{\dot B^{s,\tau}_{p,q}(W),\dot F^{s,\tau}_{p,q}(W)\right\}$$
and established their $\varphi$-transform characterizations
which establish a bridge connecting $\dot A^{s,\tau}_{p,q}(W)$ and
their corresponding sequence spaces
$\dot a^{s,\tau}_{p,q}(W)\in\{\dot b^{s,\tau}_{p,q}(W),\dot f^{s,\tau}_{p,q}(W)\}$.
The $\varphi$-transform characterization was originally
proposed by Frazier and Jawerth in \cite{fj90},
and it has been widely used in the study of various real-variable
characterizations and the boundedness of operators on
Besov-type and Triebel--Lizorkin-type spaces.
We also introduced the $A_p$-dimension for matrix weights in \cite{bhyy1},
which plays an irreplaceable role in both establishing
the aforementioned $\varphi$-transform characterization in \cite{bhyy1}
and obtaining the sharp boundedness of almost diagonal operators
on $\dot a^{s,\tau}_{p,q}(W)$ in the second article \cite{bhyy2}.
We refer to \cite{bhyy1,bhyy2} for the details and, especially,
for their histories of these subjects and more related references.

Recall that the classical Besov  and Triebel--Lizorkin spaces and many of their variants
have wide applications in analysis and such applications usually
rely on various real-variable characterizations of these spaces, especially on their decomposition
characterizations, respectively, in terms of atoms, molecules,
or wavelets. Thus, the study on decomposing and reconstructing spaces
via certain building blocks as well as their applications
is one key topic of various Besov and Triebel--Lizorkin type spaces,
and there exist a large amount of literatures working on this.
Here we refer, for instance, to \cite{fj85,fj90,fjw91,gt99,Sa18,t83,t10}
for classical Besov and Triebel--Lizorkin spaces,
to \cite{HP08,is09,is12} for weighted Besov and Triebel--Lizorkin spaces,
to \cite{b05,b07,b08,bh06,lbyy12} for anisotropic Besov and Triebel--Lizorkin spaces,
to \cite{cgn17,cgn19,cgn19-2,gjn17,gn16} for mixed-norm Besov and Triebel--Lizorkin spaces,
to \cite{gkp21} for product Besov and Triebel--Lizorkin spaces,
to \cite{bbd20,bd15,bd17,bd21c,bd23,gkkp19} for Besov and Triebel--Lizorkin spaces associated with operators,
to \cite{lsuyy,S08,S10b,ST07,yy10,ysy10-2,ysy10} for Besov-type and Triebel--Lizorkin-type spaces,
and to \cite{HL23,HLMS23,HLMS23-2,HMS16,HMS20,HMS22,HT23,lyysu}
for some generalized Besov-type and Triebel--Lizorkin-type spaces.
In particular, the decomposition characterizations of Besov and Triebel--Lizorkin type spaces
are known to be useful in obtaining the trace theorem
(see, for instance, \cite{fj90,fr08,syy10,ysy10-2,ysy10})
and the boundedness of pseudo-differential operators and Calder\'on--Zygmund operators
on them (see, for instance, \cite{ftw88,gt99,syy10,tor,ysy10-2,ysy10}).

In the matrix-weighted setting, Roudenko \cite{ro03} and Frazier and Roudenko \cite{fr04} established
 the molecular decomposition   of
matrix-weighted Besov spaces $\dot B^s_{p,q}(W)$. Based on this molecular characterization,
Frazier and Roudenko \cite{fr08} established the trace theorems of
matrix-weighted Besov spaces,
which is of crucial interest for boundary value problems of elliptic differential operators.
Additionally, Frazier and Roudenko \cite{fr04,fr21,ro03}
also used the molecular characterizations of
matrix-weighted Besov and Triebel--Lizorkin spaces
$\dot A^s_{p,q}(W)\in\{\dot B^s_{p,q}(W),\dot F^s_{p,q}(W)\}$
to obtain the boundedness of Calder\'on--Zygmund operators on these spaces.

In this article, using the discrete Calder\'on reproducing formula,
the $\varphi$-transform characterization of
$\dot A^{s,\tau}_{p,q}(W)$ in \cite{bhyy1},
and the boundedness of almost diagonal operators
on $\dot a^{s,\tau}_{p,q}(W)$ in \cite{bhyy2},
we first establish the molecular characterization of
$\dot A^{s,\tau}_{p,q}(W)$.
To simplify some conditions, our molecules differ slightly from classical ones.
As an application of both the $\varphi$-transform and the molecular
characterizations, we obtain the optimal boundedness of pseudo-differential operators
on $\dot A^{s,\tau}_{p,q}(W)$. Recall that the significance of pseudo-differential operators
is partly attributed to their role in the para-differential calculus of Bony \cite{b81}.
Additionally, using the $\varphi$-transform
and the molecular characterizations and the boundedness of almost diagonal operators,
we establish the Meyer and the Daubechies wavelet characterizations of $\dot A^{s,\tau}_{p,q}(W)$.
Applying them, we obtain the atomic characterization of $\dot A^{s,\tau}_{p,q}(W)$,
which further induces out the trace theorems of $\dot A^{s,\tau}_{p,q}(W)$.
Also, by the $\varphi$-transform characterization and
the boundedness of almost diagonal operators,
we establish the optimal boundedness of
Calder\'on--Zygmund operators on $\dot a^{s,\tau}_{p,q}(W)$.
Due to the sharp boundedness of almost diagonal operators
obtained in \cite{bhyy2}, all results presented in this article
improve their counterparts on
matrix-weighted Besov and Triebel--Lizorkin spaces $\dot A^{s}_{p,q}(W)$.
In particular, even comparing with the boundedness of Calder\'on--Zygmund operator
on unweighted Triebel--Lizorkin spaces $\dot F^{s}_{p,q}$, our results are still better.

To be precise, the organization of the remainder of this article is as follows.

In Section \ref{preliminaries}, we recall the definitions of $A_p$-matrix weights
and matrix-weighted Besov-type and Triebel--Lizorkin-type (sequence) spaces.

In Section \ref{molecular characterization}, using the discrete
Calder\'on reproducing formula, the $\varphi$-transform
characterization of $\dot A^{s,\tau}_{p,q}(W)$ and the boundedness
of almost diagonal operators on $\dot a^{s,\tau}_{p,q}(W)$,
we establish the molecular characterization of $\dot A^{s,\tau}_{p,q}(W)$.
To make some conditions take a simpler form,
we choose to slightly depart from the classic definition by first introducing molecules in
their own smoothness, cancellation, and decay,
and only link them to function spaces after proving some results that
motivate the choice of the particular parameters in making this link.
When $\tau=0$, $\dot A^{s,\tau}_{p,q}(W)$ reduces to $\dot A^s_{p,q}(W)$ and, in this case,
to obtain the molecular characterization of $\dot A^s_{p,q}(W)$,
we only need weaker conditions about molecules than the corresponding ones
in \cite[Theorem 5.2]{ro03}, \cite[Theorem 3.1]{fr04}, and \cite[Theorem 2.9]{fr21}
because the ranges of indices of a part of the bounded almost diagonal operators
used in its proof and obtained in \cite{bhyy2} are sharp.
When the dimension of the target space is $m=1$ and we have the trivial weight $W\equiv 1$,
our result in this case coincides with the molecular characterization of
Besov-type and Triebel--Lizorkin-type spaces
$\dot A^{s,\tau}_{p,q}\in\{\dot B^{s,\tau}_{p,q},\dot F^{s,\tau}_{p,q}\}$
in \cite[Theorem 4.2]{yy10}. In the last of this section,
we prove the boundedness of pseudo-differential operators
on $\dot A^{s,\tau}_{p,q}(W)$ by using both the $\varphi$-transform
and the molecular characterizations of $\dot A^{s,\tau}_{p,q}(W)$.
When $\tau=0$, we obtain the boundedness of pseudo-differential operators
on $\dot A^s_{p,q}(W)$, which is also new.
When $m=1$ and $W\equiv 1$, the space $\dot A^{s,\tau}_{p,q}(W)$
reduces to $\dot A^{s,\tau}_{p,q}$ and, in this case,
our result contains \cite[Theorem 1.5]{syy10} in which $\tau$
has an extra upper bound.

In Section \ref{wavelet characterization},
we first recall both Meyer and Daubechies wavelets
and then obtain the corresponding wavelet characterizations of $\dot A^{s,\tau}_{p,q}(W)$
by using the boundedness of almost diagonal operators on $\dot a^{s,\tau}_{p,q}(W)$
and both the $\varphi$-transform and the molecular characterizations of $\dot A^{s,\tau}_{p,q}(W)$.
As a simple application of both the wavelet and the molecular characterizations,
we establish the atomic characterization of $\dot A^{s,\tau}_{p,q}(W)$.
When $\tau=0$, $\dot A^{s,\tau}_{p,q}(W)$ reduces to $\dot A^s_{p,q}(W)$.
To obtain the corresponding wavelet characterization of $\dot A^{s,\tau}_{p,q}(W)$,
due to the sharpness of the ranges of indices of a part of the bounded almost diagonal operators used in its proof and obtained in \cite{bhyy2},
we only need a weaker smoothness assumption on Daubechies wavelets than the corresponding one
in \cite[Theorem 10.2 and Corollary 10.3]{ro03}, \cite[Theorem 4.4]{fr04}, and \cite[Theorem 1.2]{fr04}).
When $m=1$ and $W\equiv 1$,
our wavelet characterization of $\dot A^{s,\tau}_{p,q}$
in this case contains \cite[Theorem 8.3]{ysy10} in which $s\in(0,\infty)$,
but it is covered by \cite[Theorem 6.3]{lsuyy}
in which Liang et al. considered the biorthogonal wavelet system
and established the wavelet characterization of $\dot A^{s,\tau}_{p,q}$.

In Section \ref{trace theorems},
we establish the trace theorem of $\dot A^{s,\tau}_{p,q}(W)$
by using the boundedness of almost diagonal operators on $\dot a^{s,\tau}_{p,q}(W)$
and both the wavelet and the molecular characterizations of $\dot A^{s,\tau}_{p,q}(W)$.
When $\tau=0$, we obtain the trace theorem of $\dot A^s_{p,q}(W)$,
which is even new for $\dot F^s_{p,q}(W)$ and improves the known one for $\dot B^s_{p,q}(W)$
because the ranges of indices of a part of the bounded almost diagonal operators
used in its proof and obtained in \cite{bhyy2} are sharp.
When $m=1$ and $W\equiv 1$, $\dot A^{s,\tau}_{p,q}(W)$
reduces to $\dot A^{s,\tau}_{p,q}$ and, in this case,
our result contains \cite[Theorem 1.4]{syy10}
in which $\tau\in[0,\frac{1}{p}+\frac{s+n-J}{n})$.

In the last Section \ref{C-Z operators},
as an application, we obtain the boundedness of Calder\'on--Zygmund operators on
$\dot A^{s,\tau}_{p,q}(W)$.
Recall that, in \cite{ftw88,tor}, Frazier et al.
established the boundedness of Calder\'on--Zygmund operators on $\dot F^s_{p,q}$.
Compared with their result, our result only needs a weaker assumption and hence is better,
that is, there exist some Calder\'on--Zygmund operators that satisfy our assumptions but not theirs,
and Calder\'on--Zygmund operators that satisfy their assumptions always satisfy ours.
Besides being a stronger (and more general) result,
a further advantage of our result is that our Calder\'on--Zygmund smoothness conditions
take a simpler form without reference to the integer and the fractional parts.
For $m=1$ and $W\equiv 1$, $\dot A^{s,\tau}_{p,q}(W)$ reduces to $\dot A^{s,\tau}_{p,q}$ and,
in this case, our result is still new.
When $\tau=0$, we obtain the boundedness of Calder\'on--Zygmund operators on
$\dot A^s_{p,q}(W)$,
which is even new for $\dot F^s_{p,q}(W)$
and improves both \cite[Theorem 9.4]{ro03} and \cite[Theorem 4.2(i)]{fr04} for $\dot B^s_{p,q}(W)$
because they only contain some special cases and need stronger assumptions.
Such improvement is again due to the ranges of indices of
a part of the bounded almost diagonal operators
used in its proof and obtained in \cite{bhyy2} are sharp.

In the end, we make some conventions on notation.
A \emph{cube} $Q$ of $\mathbb{R}^n$ always has finite edge length
and edges of cubes are always assumed to be parallel to coordinate axes,
but $Q$ is not necessary to be open or closed.
For any cube $Q$ of $\mathbb{R}^n$,
let $c_Q$ be its center and $\ell(Q)$ its edge length.
For any $\lambda\in(0,\infty)$ and any cube $Q$ of $\mathbb{R}^n$,
let $\lambda Q$ be the cube with the same center of $Q$ and the edge length $\lambda\ell(Q)$.
The \emph{ball} $B$ of $\mathbb{R}^n$,
centered at $x\in\mathbb{R}^n$ with radius $r\in(0,\infty)$,
is defined by setting
$$
B:=\{y\in\mathbb{R}^n:\ |x-y|<r\}=:B(x,r);
$$
moreover, for any $\lambda\in(0,\infty)$, $\lambda B:=B(x,\lambda r)$.
For any $r\in\mathbb{R}$, $r_+$ is defined as $r_+:=\max\{0,r\}$
and $r_-$ is defined as $r_-:=\max\{0,-r\}$.
For any $a,b\in\mathbb{R}$, $a\wedge b:=\min\{a,b\}$ and $a\vee b:=\max\{a,b\}$.
The symbol $C$ denotes a positive constant which is independent
of the main parameters involved, but may vary from line to line.
The symbol $A\lesssim B$ means that $A\leq CB$ for some positive constant $C$,
while $A\sim B$ means $A\lesssim B\lesssim A$.
Let $\mathbb N:=\{1,2,\ldots\}$, $\mathbb Z_+:=\mathbb N\cup\{0\}$, and $\mathbb Z_+^n:=(\mathbb Z_+)^n$.
For any multi-index $\gamma:=(\gamma_1,\ldots,\gamma_n)\in\mathbb Z_+^n$
and any $x:=(x_1,\ldots,x_n)\in\mathbb R^n$,
let $|\gamma|:=\gamma_1+\ldots+\gamma_n$,
$x^\gamma:=x_1^{\gamma_1}\cdots x_n^{\gamma_n}$,
and $\partial^\gamma:=(\frac{\partial}{\partial x_1})^{\gamma_1}
\cdots(\frac{\partial}{\partial x_n})^{\gamma_n}$.
We use $\mathbf{0}$ to denote the \emph{origin} of $\mathbb{R}^n$.
For any set $E\subset\mathbb{R}^n$,
we use $\mathbf 1_E$ to denote its \emph{characteristic function}.
The \emph{Lebesgue space} $L^p(\mathbb{R}^n)$
is defined to be the set of all measurable functions
$f$ on $\mathbb{R}^n$ such that $\|f\|_{L^p(\mathbb{R}^n)}<\infty$, where
$$
\|f\|_{L^p(\mathbb{R}^n)}:=\left\{\begin{aligned}
&\left[\int_{\mathbb{R}^n}|f(x)|^p\,dx\right]^{\frac{1}{p}}
&&\text{if }p\in(0,\infty),\\
&\mathop{\mathrm{\,ess\,sup\,}}_{x\in\mathbb{R}^n}|f(x)|
&&\text{if }p=\infty.
\end{aligned}\right.
$$
The \emph{locally integrable Lebesgue space}
$L^p_{\mathrm{loc}}(\mathbb{R}^n)$ is defined to be the set of
all measurable functions $f$ on $\mathbb{R}^n$ such that,
for any bounded measurable set $E$,
$\|f\|_{L^p(E)}:=\|f\mathbf{1}_E\|_{L^p(\mathbb{R}^n)}<\infty.$
In what follows, we denote $L^p(\mathbb{R}^n)$ and $L^p_{\mathrm{loc}}(\mathbb{R}^n)$
simply, respectively, by $L^p$ and $L^p_{\mathrm{loc}}$.
For any measurable function $w$ on $\mathbb{R}^n$
and any measurable set $E\subset\mathbb{R}^n$, let
$w(E):=\int_Ew(x)\,dx.$
For any measurable function $f$ on $\mathbb{R}^n$
and any measurable set $E\subset\mathbb{R}^n$ with $|E|\in(0,\infty)$, let
$$
\fint_Ef(x)\,dx:=\frac{1}{|E|}\int_Ef(x)\,dx.
$$
For any space $X$, the product space $X^m$ with $m\in\mathbb{N}$
is defined by setting
\begin{equation*}
X^m:=\left\{\vec f:=(f_1,\ldots,f_m)^{\mathrm{T}}:\
\text{for any}\ i\in\{1,\ldots,m\},\ f_i\in X\right\}.
\end{equation*}
Also, when we prove a theorem
(and the like), in its proof we always use the same
symbols as those appearing in
the statement itself of the theorem (and the like).

\section{Matrix-Weighted Besov-Type and Triebel--Lizorkin-Type Spaces}\label{preliminaries}

In this section, we recall the definitions of $A_p$-matrix weights and
matrix-weighted Besov-type and Triebel--Lizorkin-type (sequence) spaces.
Let us begin with some basic concepts of matrices.

For any $m,n\in\mathbb{N}$,
the set of all $m\times n$ complex-valued matrices is denoted by $M_{m,n}(\mathbb{C})$
and $M_{m,m}(\mathbb{C})$ is simply denoted by $M_{m}(\mathbb{C})$.
For any $A\in M_{m,n}(\mathbb{C})$,
the \emph{conjugate transpose} of $A$ is denoted by $A^*$.

For any $A\in M_m(\mathbb{C})$, let
$$
\|A\|:=\sup_{\vec z\in\mathbb{C}^m,\,|\vec z|=1}|A\vec z|.
$$

In what follows, we regard $\mathbb{C}^m$ as $M_{m,1}(\mathbb{C})$
and let $\vec{\mathbf{0}}:=(0,\ldots,0)^\mathrm{T}\in\mathbb{C}^m$.
Moreover, for any $\vec z:=(z_1,\ldots,z_m)^\mathrm{T}\in\mathbb{C}^m$,
let $|\vec z|:=(\sum_{i=1}^m|z_i|^2)^{\frac12}$.

A matrix $A\in M_m(\mathbb{C})$ is said to be \emph{positive definite}
if, for any $\vec z\in\mathbb{C}^m\setminus\{\vec{\mathbf{0}}\}$, $\vec z^*A\vec z>0$,
and $A$ is said to be \emph{nonnegative definite} if,
for any $\vec z\in\mathbb{C}^m$, $\vec z^*A\vec z\geq0$.
The matrix $A$ is said to be \emph{invertible} if
there exists a matrix $A^{-1}\in M_m(\mathbb{C})$ such that $A^{-1}A=I_m$,
where $I_m$ is identity matrix.

Now, we recall the concept of matrix weights (see, for instance, \cite{nt96,tv97,v97}).

\begin{definition}
A matrix-valued function $W:\ \mathbb{R}^n\to M_m(\mathbb{C})$ is called
a \emph{matrix weight} if $W$ satisfies that
\begin{enumerate}[\rm(i)]
\item for any $x\in\mathbb{R}^n$, $W(x)$ is nonnegative definite;
\item for almost every $x\in\mathbb{R}^n$, $W(x)$ is invertible;
\item the entries of $W$ are all locally integrable.
\end{enumerate}
\end{definition}

Now, we recall the concept of $A_p$-matrix weights
(see, for instance, \cite[p.\,490]{fr21}).

\begin{definition}\label{def ap}
Let $p\in(0,\infty)$. A matrix weight $W$ on $\mathbb{R}^n$
is called an $A_p(\mathbb{R}^n,\mathbb{C}^m)$-\emph{matrix weight}
if $W$ satisfies that, when $p\in(0,1]$,
$$
[W]_{A_p(\mathbb{R}^n,\mathbb{C}^m)}
:=\sup_{\mathrm{cube}\,Q}\mathop{\mathrm{\,ess\,sup\,}}_{y\in Q}
\fint_Q\left\|W^{\frac{1}{p}}(x)W^{-\frac{1}{p}}(y)\right\|^p\,dx
<\infty
$$
or that, when $p\in(1,\infty)$,
$$
[W]_{A_p(\mathbb{R}^n,\mathbb{C}^m)}
:=\sup_{\mathrm{cube}\,Q}
\fint_Q\left[\fint_Q\left\|W^{\frac{1}{p}}(x)W^{-\frac{1}{p}}(y)\right\|^{p'}
\,dy\right]^{\frac{p}{p'}}\,dx
<\infty,
$$
where $\frac{1}{p}+\frac{1}{p'}=1$.
\end{definition}

In what follows, if there exists no confusion,
we denote $A_p(\mathbb{R}^n,\mathbb{C}^m)$ simply by $A_p$.
Next, we recall the concept of reducing operators (see, for instance, \cite[(3.1)]{v97}).

\begin{definition}\label{reduce}
Let $p\in(0,\infty)$, $W$ be a matrix weight,
and $E\subset\mathbb{R}^n$ a bounded measurable set satisfying $|E|\in(0,\infty)$.
The matrix $A_E\in M_m(\mathbb{C})$ is called a \emph{reducing operator} of order $p$ for $W$
if $A_E$ is positive definite and,
for any $\vec z\in\mathbb{C}^m$,
\begin{equation}\label{equ_reduce}
\left|A_E\vec z\right|
\sim\left[\fint_E\left|W^{\frac{1}{p}}(x)\vec z\right|^p\,dx\right]^{\frac{1}{p}},
\end{equation}
where the positive equivalence constants depend only on $m$ and $p$.
\end{definition}

\begin{remark}
In Definition \ref{reduce}, the existence of $A_E$ is guaranteed by
\cite[Proposition 1.2]{g03} and \cite[p.\,1237]{fr04}; we omit the details.
\end{remark}

To obtain the sharp estimate of $\|A_QA_R^{-1}\|$,
we need the concept of $A_p$-dimensions.

\begin{definition}
Let $p\in(0,\infty)$, $d\in\mathbb{R}$, and $W$ be a matrix weight.
Then $W$ is said to have the \emph{$A_p$-dimension $d$}
if there exists a positive constant $C$ such that,
for any cube $Q\subset\mathbb{R}^n$ and any $i\in\mathbb{Z}_+$,
when $p\in(0,1]$,
\begin{equation*}
\mathop{\mathrm{\,ess\,sup\,}}_{y\in2^iQ}\fint_Q \left\|W^{\frac{1}{p}}(x)W^{-\frac{1}{p}}(y)\right\|^p\,dx
\leq C2^{id}
\end{equation*}
or, when $p\in(1,\infty)$,
\begin{equation*}
\fint_Q\left[\fint_{2^iQ}\left\|W^{\frac{1}{p}}(x)W^{-\frac{1}{p}}(y)
\right\|^{p'}\,dy\right]^{\frac{p}{p'}}\,dx
\leq C2^{id},
\end{equation*}
where $\frac{1}{p}+\frac{1}{p'}=1$.
\end{definition}

The following conclusion is just \cite[Corollary 2.37]{bhyy1}.

\begin{lemma}\label{22 precise}
Let $p\in(0,\infty)$, let $W\in A_p$ have the $A_p$-dimension $d\in[0,n)$,
and let $\{A_Q\}_{\mathrm{cube}\,Q}$ be a family of
reducing operators of order $p$ for $W$. If $p\in(1,\infty)$,
let further $\widetilde W:=W^{-\frac 1{p-1}}$ (which belongs to $A_{p'}$)
have the $A_{p'}$-dimension $\widetilde d$,
while, if $p\in(0,1]$, let $\widetilde d:=0$.
Let
\begin{equation}\label{Delta}
\Delta:=\frac{d}{p}+\frac{\widetilde d}{p'}.
\end{equation}
Then there exists a positive constant $C$ such that,
for any cubes $Q$ and $R$ of $\mathbb{R}^n$,
$$
\left\|A_QA_R^{-1}\right\|
\leq C\max\left\{\left[\frac{\ell(R)}{\ell(Q)}\right]^{\frac{d}{p}},
\left[\frac{\ell(Q)}{\ell(R)}\right]^{\frac{\widetilde d}{p'}}\right\}
\left[1+\frac{|c_Q-c_R|}{\ell(Q)\vee\ell(R)}\right]^\Delta.
$$
\end{lemma}

\begin{definition}
Let $p\in(0,\infty)$ and $W\in A_p$ be a matrix weight. We say that $W$ has $A_p$-dimensions $(d,\widetilde d,\Delta)$ if
\begin{enumerate}[{\rm(i)}]
\item $W$ has the $A_p$-dimension $d$,
\item $p\in(0,1]$ and $\widetilde d=0$, or $p\in(1,\infty)$
and $W^{-\frac 1{p-1}}$ (which belongs to $A_{p'}$)
has the $A_{p'}$-dimension $\widetilde d$, and
\item $\Delta$ is the same as in \eqref{Delta}.
\end{enumerate}
\end{definition}

Let $s\in\mathbb{R}$, $\tau\in[0,\infty)$, and $p,q\in(0,\infty]$.
For any sequence $\{f_j\}_{j\in\mathbb Z}$ of measurable functions on $\mathbb{R}^n$,
any subset $J\subset\mathbb Z$, and any measurable set $E\subset\mathbb{R}^n$, let
\begin{align*}
\|\{f_j\}_{j\in\mathbb Z}\|_{L\dot B_{pq}(E\times J)}
:=\|\{f_j\}_{j\in\mathbb Z}\|_{\ell^qL^p(E\times J)}
:=\|\{f_j\}_{j\in\mathbb Z}\|_{\ell^q(J;L^p(E))}
:=\left[\sum_{j\in J}\|f_j\|_{L^p(E)}^q\right]^{\frac{1}{q}}
\end{align*}
and
\begin{align*}
\|\{f_j\}_{j\in\mathbb Z}\|_{L\dot F_{pq}(E\times J)}
:=\|\{f_j\}_{j\in\mathbb Z}\|_{L^p\ell^q(E\times J)}
:=\|\{f_j\}_{j\in\mathbb Z}\|_{L^p(E;\ell^q(J))}
:=\left\|\left(\sum_{j\in J}|f_j|^q\right)^{\frac{1}{q}}\right\|_{L^p(E)}
\end{align*}
with the usual modification made when $q=\infty$.
For simplicity of the presentation, in what follows, we may drop the domain $E\times J$ from these symbols, when it is the full space $E\times J=\mathbb R^n\times\mathbb Z$.
We use $L\dot A_{pq}\in\{L\dot B_{pq},L\dot F_{pq}\}$ as a generic notation in statements that apply to both types of spaces.

In particular, for any $P\in\mathscr{Q}$, we abbreviate $\widehat{P}:=P\times\{j_P,j_P+1,\ldots\}$ so that
$$
\|\{f_j\}_{j\in\mathbb Z}\|_{L\dot B_{pq}(\widehat{P})}
=\|\{f_j\}_{j\in\mathbb Z}\|_{\ell^qL^p(\widehat{P})}
=\left[\sum_{j=j_P}^\infty\|f_j\|_{L^p(P)}^q\right]^{\frac{1}{q}}
$$
and
$$
\|\{f_j\}_{j\in\mathbb Z}\|_{L\dot F_{pq}(\widehat{P})}
=\|\{f_j\}_{j\in\mathbb Z}\|_{L^p\ell^q(\widehat{P})}
=\left\|\left(\sum_{j=j_P}^\infty|f_j|^q\right)^{\frac{1}{q}}\right\|_{L^p(P)}.
$$
Let us further define
\begin{equation}\label{LApq}
\|\{f_j\}_{j\in\mathbb Z}\|_{L\dot A_{p,q}^\tau}
:=\sup_{P\in\mathscr{Q}}|P|^{-\tau}\|\{f_j\}_{j\in\mathbb Z}\|_{L\dot A_{pq}(\widehat{P})}
\end{equation}
for both choices of $L\dot A_{p,q}^\tau\in\{L\dot B_{p,q}^\tau,L\dot F_{p,q}^\tau\}$.

Moreover, for any $k\in\mathbb{Z}$, let $\{f_j\}_{j\geq k}
:=\{f_j\mathbf{1}_{[k,\infty)}(j)\}_{j\in\mathbb{Z}}$.

Let $\mathcal{S}$ be the space of all Schwartz functions on $\mathbb{R}^n$,
equipped with the well-known topology determined by a countable family of norms,
and let $\mathcal{S}'$ be the set of all continuous linear functionals on $\mathcal{S}$,
equipped with the weak-$*$ topology.
For any $f\in L^1$ and $\xi\in\mathbb{R}^n$, let
$$
\widehat{f}(\xi):=\int_{\mathbb{R}^n}f(x)e^{-ix\cdot\xi}\,dx
$$
to denote the \emph{Fourier transform} of $f$.
This agrees with the normalisation of the Fourier transform used,
for instance, in \cite[p.\,4]{fjw91} and \cite[p.\,452]{yy10},
and allows us to quote some lemmas from these works directly,
whereas using any other normalisation (such as with $2\pi$ in the exponent)
would also necessitate slight adjustments here and there in several other formulas.
For any $f\in\mathcal S$ and $x\in\mathbb{R}^n$, let
$f^\vee(x):=\widehat{f}(-x)$
to denote the \emph{inverse Fourier transform} of $f$.
It is well known that, for any $f\in\mathcal S$,
$(\widehat{f})^\vee=(f^\vee)^\wedge=f.$
We can also define the \emph{Fourier transform $\widehat{f}$}
and the \emph{inverse Fourier transform $f^\vee$} of any Schwartz distribution $f$ as follows.
For any $f\in\mathcal S'$ and $\varphi\in\mathcal S$, let
$\langle\widehat{f},\varphi\rangle
:=\langle f,\widehat{\varphi}\rangle$
and
$\langle f^\vee,\varphi\rangle
:=\langle f,\varphi^\vee\rangle.$

Let $\varphi,\psi\in\mathcal{S}$ satisfy
\begin{equation}\label{19}
\operatorname{supp}\widehat{\varphi},\operatorname{supp}\widehat{\psi}
\subset\left\{\xi\in\mathbb{R}^n:\ \frac12\leq|\xi|\leq2\right\}
\end{equation}
and
\begin{equation}\label{20}
\left|\widehat\varphi(\xi)\right|,\left|\widehat\psi(\xi)\right|\geq C>0
\text{ if }\xi\in\mathbb{R}^n\text{ with }\frac35\leq|\xi|\leq\frac53,
\end{equation}
where $C$ is a positive constant independent of $\xi$ and
\begin{equation}\label{21}
\sum_{j\in\mathbb{Z}}\overline{\widehat{\varphi}\left(2^j\xi\right)}
\widehat{\psi}\left(2^j\xi\right)=1
\text{ if }\xi\in\mathbb{R}^n\setminus\{\mathbf{0}\}.
\end{equation}

For any complex-valued function $g$ on $\mathbb{R}^n$, let
$\operatorname{supp}g:=\{x\in\mathbb{R}^n:\ g(x)\neq0\}.$
For any $f\in\mathcal{S}'$, let
$$
\operatorname{supp}f:=\bigcap\left\{\text{closed set }K\subset\mathbb{R}^n:\
\langle f,\varphi\rangle=0\text{ for any }\varphi\in\mathcal{S}
\text{ with}\operatorname{supp}\varphi\subset\mathbb{R}^n\setminus K\right\},
$$
which can be found in \cite[Definition 2.3.16]{g14c}.

Let $\varphi$ be a complex-valued function on $\mathbb{R}^n$.
For any $j\in\mathbb{Z}$ and $x\in\mathbb{R}^n$, let
$\varphi_j(x):=2^{jn}\varphi(2^jx)$.
For any $Q:=Q_{j,k}\in\mathscr{Q}$ and $x\in\mathbb{R}^n$, let
$$
\varphi_Q(x)
:=|Q|^{-\frac12}\varphi\left(2^jx-k\right)
=|Q|^{\frac12}\varphi_j(x-x_Q).
$$

As in \cite{yy10}, let
\begin{equation*}
\mathcal{S}_\infty:=\left\{\varphi\in\mathcal{S}:\
\int_{\mathbb R^n}x^\gamma\varphi(x)\,dx=0
\text{ for any }\gamma\in\mathbb{Z}_+^n\right\},
\end{equation*}
regarded as a subspace of $\mathcal{S}$ with the same topology.
We denote by $\mathcal{S}_\infty'$ the space of
all continuous linear functionals on $\mathcal{S}_\infty$, equipped with the weak-$*$ topology.
It is well known that $\mathcal{S}_\infty'$ coincides with
the quotient space $\mathcal{S}'/\mathcal{P}$ as topological spaces,
where $\mathcal{P}$ denotes the set of all polynomials on $\mathbb{R}^n$;
see, for instance, \cite[Chapter 5]{t83}, \cite[Proposition 8.1]{ysy10}, or \cite{Sa17}.

We first recall the following Calder\'on reproducing formulae
which are \cite[Lemma 2.1]{yy10}.

\begin{lemma}\label{7}
Let $\varphi,\psi\in\mathcal{S}$ satisfy \eqref{21} and
both $\overline{\operatorname{supp}\widehat{\varphi}}$ and $\overline{\operatorname{supp}\widehat{\psi}}$
are compact and bounded away from the origin.
Then, for any $f\in\mathcal{S}_\infty$,
\begin{equation}\label{7x}
f=\sum_{j\in\mathbb{Z}}2^{-jn}\sum_{k\in\mathbb{Z}^n}
\left(\widetilde{\varphi}_j*f\right)\left(2^{-j}k\right)
\psi_j\left(\cdot-2^{-j}k\right)
=\sum_{Q\in\mathscr{Q}}\left\langle f,\varphi_Q\right\rangle\psi_Q
\end{equation}
in $\mathcal{S}_\infty $, where $\widetilde{\varphi}(x):=\overline{\varphi(-x)}$
for any $x\in\mathbb{R}^n$.
Moreover, for any $f\in\mathcal{S}_\infty'$,
\eqref{7x} also converges in $\mathcal{S}_\infty'$.
\end{lemma}

Next, we recall the concept of matrix-weighted Besov-type
and Triebel--Lizorkin-type (sequence) spaces (see \cite[Definitions 3.5 and 3.24]{bhyy1}).

\begin{definition}
Let $s\in\mathbb{R}$, $\tau\in[0,\infty)$, $p\in(0,\infty)$, and $q\in(0,\infty]$.
Let $\varphi\in\mathcal{S}$ satisfy \eqref{19} and \eqref{20}, and let $W\in A_p$ be a matrix weight.
The \emph{homogeneous matrix-weighted Besov-type space}
$\dot B^{s,\tau}_{p,q}(W,\varphi)$
and the \emph{homogeneous matrix-weighted Triebel--Lizorkin-type space}
$\dot F^{s,\tau}_{p,q}(W,\varphi)$
are defined by setting
$$
\dot A^{s,\tau}_{p,q}(W,\varphi)
:=\left\{\vec{f}\in(\mathcal{S}_\infty')^m:\
\left\|\vec{f}\right\|_{\dot A^{s,\tau}_{p,q}(W,\varphi)}<\infty\right\},
$$
where, for any $\vec{f}\in(\mathcal{S}_\infty')^m$,
$$
\left\|\vec{f}\right\|_{\dot A^{s,\tau}_{p,q}(W,\varphi)}
:=\left\|\left\{2^{js}\left|W^{\frac{1}{p}}\left(\varphi_j*\vec f\right)
\right|\right\}_{j\in\mathbb Z}\right\|_{L\dot A_{p,q}^\tau}
$$
with $\|\cdot\|_{L\dot A_{p,q}^\tau}$ the same as in \eqref{LApq}.
\end{definition}

\begin{remark}
Indeed, in \cite[Proposition 3.35]{bhyy1}, it was showed that
$\dot A^{s,\tau}_{p,q}(W,\varphi)$ is independent of the choice of $\varphi$.
Based on this, in what follows, we denote $\dot A^{s,\tau}_{p,q}(W,\varphi)$
simply by $\dot A^{s,\tau}_{p,q}(W)$.
\end{remark}

In what follows, for any $Q\in\mathscr{Q}$,
let $\widetilde{\mathbf{1}}_Q:=|Q|^{-\frac12}\mathbf{1}_Q$.

\begin{definition}
Let $s\in\mathbb{R}$, $\tau\in[0,\infty)$, $p\in(0,\infty)$, $q\in(0,\infty]$, and $W\in A_p$.
The \emph{homogeneous matrix-weighted Besov-type sequence space} $\dot b^{s,\tau}_{p,q}(W)$
and the \emph{homogeneous matrix-weighted Triebel--Lizorkin-type sequence space}
$\dot f^{s,\tau}_{p,q}(W)$
are defined to be the sets of all sequences
$\vec t:=\{\vec t_Q\}_{Q\in\mathscr{Q}}\subset\mathbb{C}^m$ such that
$$
\left\|\vec{t}\right\|_{\dot a^{s,\tau}_{p,q}(W)}
:=\left\|\left\{2^{js}\left|W^{\frac{1}{p}}\sum_{Q\in\mathscr{Q}_j}\vec{t}_Q\widetilde{\mathbf{1}}_Q
\right|\right\}_{j\in\mathbb Z}\right\|_{L\dot A_{p,q}^\tau}<\infty,
$$
where $\|\cdot\|_{L\dot A_{p,q}^\tau}$ is the same as in \eqref{LApq}.
\end{definition}

\section{Molecules and Their Applications}
\label{molecular characterization}

In this section, we first establish the molecular characterization of $\dot A^{s,\tau}_{p,q}(W)$
in Section \ref{mc} and then, as an application, we obtain the boundedness of pseudo-differential operators
on $\dot A^{s,\tau}_{p,q}(W)$ in Section \ref{pdo}.

\subsection{Molecular Characterization\label{mc}}

To establish the smooth molecular decomposition of these spaces,
we need the sharp boundedness of almost diagonal operators for $\dot A^{s,\tau}_{p,q}(W)$.
Let us begin with some concepts.

Let $B:=\{b_{Q, P}\}_{Q, P\in\mathscr{Q}}\subset\mathbb{C}$.
For any sequence $\vec t:=\{\vec t_R\}_{R\in\mathscr{Q}}\subset\mathbb{C}^m$,
we define $B\vec t:=\{(B\vec t)_Q\}_{Q\in\mathscr{Q}}$ by setting,
for any $Q\in\mathscr{Q}$,
$(B\vec t)_Q:=\sum_{R\in\mathscr{Q}}b_{Q,R}\vec t_R$
if the above summation is absolutely convergent.

Now, we recall the concept of almost diagonal operators.

\begin{definition}
Let $D,E,F\in\mathbb{R}$. We define the special infinite matrix
$B^{DEF}:=\{b_{Q,R}^{DEF}\}_{Q,R\in\mathscr{Q}}\subset\mathbb{C}$
by setting, for any $Q,R\in\mathscr{Q}$,
\begin{equation}\label{bDEF}
b_{Q,R}^{DEF}
:=\left[1+\frac{|x_Q-x_R|}{\ell(Q)\vee\ell(R)}\right]^{-D}
\left\{\begin{aligned}
&\left[\frac{\ell(Q)}{\ell(R)}\right]^E&&\text{if }\ell(Q)\leq\ell(R),\\
&\left[\frac{\ell(R)}{\ell(Q)}\right]^F&&\text{if }\ell(R)<\ell(Q).
\end{aligned}\right.
\end{equation}

An infinite matrix $B:=\{b_{Q,R}\}_{Q,R\in\mathscr{Q}}\subset\mathbb{C}$
is said to be \emph{$(D,E,F)$-almost diagonal}
if there exists a positive constant $C$ such that
$|b_{Q,R}|\leq Cb_{Q,R}^{DEF}.$
\end{definition}

Let
\begin{equation}\label{J}
J:=\left\{\begin{aligned}
&\frac{n}{\min\{1,p\}}&&\text{if we are dealing with a Besov-type space},\\
&\frac{n}{\min\{1,p,q\}}&&\text{if we are dealing with a Triebel--Lizorkin-type space}
\end{aligned}\right.
\end{equation}
and
\begin{equation}\label{tildeJ}
J_\tau:=\begin{cases}
n&\text{if }\tau>\frac{1}{p}\text{ or }(\tau,q)=(\frac{1}{p},\infty)\quad\textup{(}\text{``supercritical case''}\textup{)},\\
\displaystyle\frac{n}{\min\{1,q\}}&\text{if }\dot a^{s,\tau}_{p,q}
=\dot f^{s,\frac{1}{p}}_{p,q}\text{ and }q<\infty\quad\textup{(}\text{``critical case''}\textup{)},\\
J&\text{if }\tau<\frac1p\text{, or }\dot a^{s,\tau}_{p,q}=\dot b^{s,\frac1p}_{p,q}\text{ and }q<\infty\quad\textup{(}\text{``subcritical case''}\textup{)}.
\end{cases}
\end{equation}

The following lemma is just \cite[Theorem 9.1]{bhyy1}.

\begin{lemma}\label{ad BF2}
Let $s\in\mathbb R$, $\tau\in[0,\infty)$, $p\in(0,\infty)$, $q\in(0,\infty]$, and $d\in[0,n)$.
Let $J$ and $J_\tau$ be the same as, respectively, in \eqref{J} and \eqref{tildeJ}, and let
\begin{equation}\label{tauJ2}
\widehat\tau:=\left[\left(\tau-\frac{1}{p}\right)+\frac{d}{np}\right]_+,\
\widetilde{J}:=J_{\tau}+\left[\left(n\widehat\tau\right)\wedge\frac{d}{p}\right],\text{ and }
\widetilde{s}:=s+n\widehat\tau.
\end{equation}
If $B$ is $(D,E,F)$-almost diagonal with
\begin{equation}\label{ad new2}
D>\widetilde{J},\
E>\frac{n}{2}+\widetilde{s},\text{ and }
F>\widetilde{J}-\frac{n}{2}-\widetilde{s},
\end{equation}
then $B$ is bounded on $\dot a^{s,\tau}_{p,q}(W)$
whenever $W\in A_p$ has the $A_p$-dimension $d$.
\end{lemma}

\begin{definition}\label{def ad tau}
Let $s\in\mathbb R$, $\tau\in[0,\infty)$, $p\in(0,\infty)$, $q\in(0,\infty]$, and $d\in[0,n)$.
An infinite matrix $B:=\{b_{Q,R}\}_{Q,R\in\mathscr{Q}}\subset\mathbb{C}$
is said to be \emph{$\dot a^{s,\tau}_{p,q}(d)$-almost diagonal}
if it is $(D,E,F)$-almost diagonal with $D,E,F$ the same as in \eqref{ad new2}.
\end{definition}

For any $r\in\mathbb R$, let
\begin{equation}\label{ceil}
\begin{cases}
\lfloor r\rfloor:=\max\{k\in\mathbb Z:\ k\leq r\},\\
\lfloor\!\lfloor r\rfloor\!\rfloor:=\max\{k\in\mathbb Z:\ k< r\},
\end{cases}
\begin{cases}
\lceil r\rceil:=\min\{k\in\mathbb Z:\ k\geq r\},\\
\lceil\!\lceil r\rceil\!\rceil:=\min\{k\in\mathbb Z:\ k>r\},
\end{cases}
\end{equation}
and
\begin{equation}\label{r**}
\begin{cases}
r^*:=r-\lfloor r\rfloor\in[0,1),\\
r^{**}:=r-\lfloor\!\lfloor r\rfloor\!\rfloor\in(0,1].
\end{cases}
\end{equation}
Thus, by their definitions, we obviously have, for any $r\in\mathbb R\setminus\mathbb Z$
$$
\lfloor r\rfloor
=\lfloor\!\lfloor r\rfloor\!\rfloor
<r
<\lceil r\rceil
=\lceil\!\lceil r\rceil\!\rceil
\text{ and }
r^{*}=r^{**}\in(0,1),
$$
for any $r\in\mathbb Z$
$$
\lfloor r\rfloor=\lceil r\rceil=r,\
\lfloor\!\lfloor r\rfloor\!\rfloor=r-1,\
\lceil\!\lceil r\rceil\!\rceil=r+1,\
r^*=0,\text{ and }
r^{**}=1,
$$
and always $\lceil r\rceil=\lfloor\!\lfloor r\rfloor\!\rfloor+1$
and $\lceil\!\lceil r\rceil\!\rceil=\lfloor r\rfloor+1$.

In the literature on spaces of Besov and Triebel--Lizorkin type,
there exists a tradition of defining a concept of molecules
(often two types of them, respectively called ``analysis'' and ``synthesis'')
for a space $\dot A^s_{p,q}$.
The molecules are functions essentially localized in the neighbourhood of a cube $Q$,
with appropriate smoothness, cancellation,
and decay determined by the parameters of the space in question.
Here, we choose to slightly depart from this tradition
by first introducing molecules in their own right
and only link them to function spaces after proving some results that
motivate the choice of the particular parameters in making this link.

\begin{definition}\label{moleKLMN}
Let $K,M\in[0,\infty)$ and $L,N\in\mathbb{R}$.
For any $K\in[0,\infty)$, $Q\in\mathscr{Q}$, and $x\in\mathbb{R}^n$, let
\begin{equation*}
u_K(x):=(1+|x|)^{-K}\text{ and }
(u_K)_Q(x):=|Q|^{-\frac12}u_K\left(\frac{x-x_Q}{\ell(Q)}\right).
\end{equation*}
A function $m_Q$ is called a \emph{(smooth) $(K,L,M,N)$-molecule on a cube $Q$}
if, for any $x,y\in\mathbb R^n$ and any multi-index $\gamma\in\mathbb{Z}_+^n$
in the specified ranges below, it satisfies
$$
|m_Q(x)|\leq(u_{K})_Q(x), \
\int_{\mathbb R^n}x^\gamma m_Q(x)\,dx=0\text{ if }|\gamma|\leq L,
$$
$$
|\partial^\gamma m_Q(x)|\leq[\ell(Q)]^{-|\gamma|}(u_{M})_Q(x)\text{ if }|\gamma|<N,
$$
and
$$
|\partial^\gamma m_Q(x)-\partial^\gamma m_Q(y)|
\leq[\ell(Q)]^{-|\gamma|}\left[\frac{|x-y|}{\ell(Q)}\right]^{N^{**}}
\sup_{|z|\leq|x-y|}(u_{M})_Q(x+z)
$$
if $|\gamma|=\lfloor\!\lfloor N\rfloor\!\rfloor$,
where $N^{**}$ and $\lfloor\!\lfloor N\rfloor\!\rfloor$ are the same as, respectively, in \eqref{r**} and \eqref{ceil}.
\end{definition}

Therefore, the parameters $(K,L,M,N)$ describe the decay $(K)$
and the cancellation $(L)$ of $m_Q$,
and the decay $(M)$ of its derivatives up to a specific,
possibly fractional, order of the smoothness $(N)$.
We will only consider smooth molecules and hence drop the word ``smooth'' for brevity.

Notice that the length $|\gamma|$ of a multi-index is a non-negative integer and hence
\begin{itemize}
\item any condition for multi-indices $\gamma$ of length
(less than or) equal to a negative number is void;
\item the condition ``$|\gamma|\leq L$'' (resp. ``$|\gamma|<N$'')
is equivalent to ``$|\gamma|\leq\lfloor L\rfloor$'' (resp. ``$|\gamma|\leq\lfloor\!\lfloor N\rfloor\!\rfloor$'').
\end{itemize}
The single real parameter $N$ contains the information about
both the order $\lfloor\!\lfloor N\rfloor\!\rfloor$ of the bounded derivatives
and the H\"older continuity index $N^{**}$,
and leads to cleaner formulations than
stating separate conditions for $\lfloor\!\lfloor N\rfloor\!\rfloor$ and $N^{**}$.
On the other hand, only the integer part of $L$ actually matters in these conditions.
Later on we will encounter situations,
where we have a lower bound for $L$ in terms of some real number.
By what we have observed, conditions like ``$L\geq J-n-s$''
and ``$L\geq\lfloor J-n-s\rfloor$'' are equivalent,
so we can make the formulas slightly simpler by dropping the redundant integer part operation.

The following two results are respectively \cite[Lemmas B.1 and B.2]{fj90},
where we correct a typo in the relation between $j$ and $k$
in the first lemma and also change some symbols to avoid conflicting notation.

\begin{lemma}\label{fj B1}
Let $D>n$, $E\in\mathbb Z_+$, $\theta\in(0,1]$, and $F>n+E+\theta$.
Suppose that $g,h\in L^1$ satisfy,
for some $j,k\in\mathbb Z$ with $j\leq k$ and
for any $x,y\in\mathbb{R}^n$,
\begin{equation*}
\left|\partial^\gamma g(x)\right|
\leq2^{j(\frac{n}{2}+|\gamma|)}
\left(1+2^j|x|\right)^{-D}\text{ if }|\gamma|\leq E,
\end{equation*}
\begin{equation*}
\left|\partial^\gamma g(x)-\partial^\gamma g(y)\right|
\leq2^{j(\frac{n}{2}+E+\theta)} |x-y|^{\theta}
\sup_{z\leq |x-y|}\left(1+2^j|z-x|\right)^{-D}\text{ if }|\gamma|=E,
\end{equation*}
\begin{equation*}
|h(x)|\leq2^{k\frac{n}{2}}\left(1+2^k|x|\right)^{-\max(D,F)},\
\mathrm{and}\
\int_{\mathbb R^n}x^\gamma h(x)\,dx=0\text{ if }|\gamma|\leq E.
\end{equation*}
Then there exists a positive constant $C$,
independent of $j$, $k$, and $x$, such that, for any $x\in\mathbb{R}^n$,
\begin{equation*}
|g*h(x)|\leq C2^{-(k-j)(E+\theta+\frac{n}{2})}\left(1+2^j|x|\right)^{-D}.
\end{equation*}
\end{lemma}

\begin{lemma}\label{fj B2}
Let $D\in(n,\infty)$ and suppose that $g,h\in L^1$ satisfy,
for some $j,k\in\mathbb Z$ with $j\leq k$  and for any $x\in\mathbb{R}^n$,
\begin{equation*}
|g(x)|\leq2^{j\frac{n}{2}}\left(1+2^j|x|\right)^{-D}
\text{ and }
|h(x)|\leq2^{k\frac{n}{2}}\left(1+2^k|x|\right)^{-D}.
\end{equation*}
Then there exists a positive constant $C$,
independent of $j$, $k$, and $x$, such that, for any $x\in\mathbb{R}^n$,
\begin{equation*}
|g*h(x)|\leq C2^{-(k-j)\frac{n}{2}}\left(1+2^j|x|\right)^{-D}.
\end{equation*}
\end{lemma}

The following lemma is analogous to \cite[Corollary B.3]{fj90},
but stated in terms of generic $(K,L,M,N)$-molecules,
instead of molecules of $\dot F^s_{p,q}$ in \cite{fj90}.
This will allow us to identify and
motivate a reasonable concept of molecules of $\dot A^{s,\tau}_{p,q}(W)$.

\begin{lemma}\label{like B3}
Let $m_Q$ be a $(K_m,L_m,M_m,N_m)$-molecule on a cube $Q$
and let $b_P$ be a $(K_b,L_b,M_b,N_b)$-molecule on a cube $P$,
where $K_m,M_m,K_b,M_b\in(n,\infty)$ and $L_m,N_m,L_b,N_b$ are real numbers.
Then, for any $\alpha\in(0,\infty)$,
there exists a positive constant $C$ such that
$|\langle m_Q,b_P\rangle|
\leq C b_{Q,P}^{MGH},$
where $b_{Q,R}^{MGH}$ is the same as in \eqref{bDEF},
$M:=K_m\wedge M_m\wedge K_b\wedge M_b\in(n,\infty),$
and, with $\lceil\!\lceil\cdot\rceil\!\rceil$ being the same as in \eqref{ceil},
\begin{equation*}
G:=\frac{n}{2}+[N_b\wedge\lceil\!\lceil L_m\rceil\!\rceil\wedge(K_m-n-\alpha)]_+,\
H:=\frac{n}{2}+[N_m\wedge\lceil\!\lceil L_b\rceil\!\rceil\wedge(K_b-n-\alpha)]_+.
\end{equation*}
\end{lemma}

\begin{proof}
Suppose first that $\ell(Q)=2^{-j}\geq2^{-k}=\ell(P)$, and consider the functions
\begin{equation*}
g(\cdot):=m_Q(x_Q-\cdot)
\text{ and }
h(\cdot):=b_P(x_P+\cdot)
\end{equation*}
which are molecules on cubes of respective lengths $\ell(Q)$ and $\ell(P)$,
but both close to the origin. Moreover, by a change of variables, we have
\begin{align}\label{182}
\langle m_Q,m_P\rangle
&=\int_{\mathbb R^n}g(x_Q-y)h(y-x_P)\,dy\\
&=\int_{\mathbb R^n}g(x)h(x_Q-x_P-x)\,dx
=(g*h)(x_Q-x_P).\notag
\end{align}
Let us immediately observe that both $g$ and $h$ satisfy
the assumptions of Lemma \ref{fj B2} with $D:=K_Q\wedge K_P$.
Thus, \eqref{182} and Lemma \ref{fj B2} prove that
\begin{align}\label{from B2}
|\langle m_Q,b_P\rangle|
&\lesssim2^{-(k-j)\frac{n}{2}}\left(1+2^j|x_Q-x_P|\right)^{-D}\\
&\leq\left[\frac{\ell(P)}{\ell(Q)}\right]^{\frac{n}{2}}
\left[1+\frac{|x_Q-x_P|}{\ell(Q)}\right]^{-(K_m\wedge M_m\wedge K_b\wedge M_b)},\notag
\end{align}
where replacing $D$ by the smaller quantity in the last step only increases the bound.

We then assume that $N_m>0$ and $L_b\geq 0$ to derive an alternative bound.
Then, under these assumptions, $g$ satisfies the assumptions of Lemma \ref{fj B1},
provided that the parameters are chosen so that
$n<D\leq M_m$ and $E+\theta\leq N_m,$
whereas $h$ satisfies the assumptions of Lemma \ref{fj B1}, provided that
$\max(D,F)\leq K_b,$ $E\leq\lfloor L_b\rfloor,$ and $E+\theta<F-n.$
Let us begin by choosing
$D=M_m\wedge K_b$ and $F=K_b.$
Turning to $E$ and $\theta$, let us observe that
these are uniquely determined by the sum $E+\theta$
because $E=\lfloor\!\lfloor E+\theta\rfloor\!\rfloor$ and $\theta=(E+\theta)^{**}$,
where $\lfloor\!\lfloor E+\theta\rfloor\!\rfloor $ and $ (E+\theta)^{**}$ are the same as,
respectively, in \eqref{ceil} and \eqref{r**},
so it suffices to consider $E+\theta$ as a whole.
Since $\theta\in(0,1]$ and $E\in\mathbb{Z}_+$, the condition $E\leq\lfloor L_b\rfloor$
is equivalent to $E+\theta\leq\lfloor L_b\rfloor+1=\lceil\!\lceil L_b\rceil\!\rceil$
and hence the totality of restrictions imposed on $E+\theta$ is
$E+\theta\leq N_m\wedge\lceil\!\lceil L_b\rceil\!\rceil$ and
$E+\theta<K_b-n.$
Thus, with $\alpha>0$, we can choose
\begin{equation*}
E+\theta=[N_m\wedge\lceil\!\lceil L_b\rceil\!\rceil\wedge (K_b-n-\alpha)]_+,
\end{equation*}
where we are free to take the positive part
because, under these assumptions, both $N_m>0$ and $\lceil\!\lceil L_b\rceil\!\rceil>0$
and also because $K_b-n-\alpha>0$ for small enough $\alpha>0$ (because $K_b>n$).

With these choices, using \eqref{182} and Lemma \ref{fj B1}, we obtain
\begin{align}\label{from B1}
|\langle m_Q,b_P\rangle|
&\lesssim2^{-(k-j)(E+\theta+\frac{n}{2})}\left(1+2^j|x_Q-x_P|\right) ^{-D}\\
&\leq\left[\frac{\ell(P)}{\ell(Q)}\right]^{\frac{n}{2}
+[N_m\wedge\lceil\!\lceil L_b\rceil\!\rceil\wedge(K_b-n-\alpha)]_+}
\left[ 1+\frac{|x_Q-x_P|}{\ell(Q)}\right]^{-(K_m\wedge M_m\wedge K_b\wedge M_b)},\notag
\end{align}
where again we symmetrised the last exponent which only makes the bound bigger.

Bound \eqref{from B1} was achieved under the assumption that $N_m>0$ and $L_b\geq0$.
But, if either of these conditions fails,
then $[N_m\wedge\lceil\!\lceil L_b\rceil\!\rceil\wedge(K_b-n-\alpha)]_+=0$
and \eqref{from B1} reduces to \eqref{from B2}, and hence \eqref{from B1} is always valid.
Thus, in fact, \eqref{from B1} is a valid bound under no other restrictions on
the molecular parameters than those in the statement of the present lemma.
This bound agrees with that stated in the present lemma under the assumption that $\ell(Q)\geq\ell(P)$.
For $\ell(Q)\leq\ell(P)$, we simply exchange the roles of $P$ and $Q$ to arrive at the other bound.
This finishes the proof of Lemma \ref{like B3}.
\end{proof}

We can now identify minimal smoothness,
decay, and cancellation conditions on families of molecules
to produce $\dot a^{s,\tau}_{p,q}(d)$-almost diagonal operators.

\begin{theorem}\label{adMol}
Let $s\in\mathbb R$, $\tau\in[0,\infty)$, $p\in(0,\infty)$, $q\in(0,\infty]$, and $d\in[0,n)$.
Let $\{m_Q\}_{Q\in\mathscr Q}$ be a family of $(K_m,L_m,M_m,N_m)$-molecules
and $\{b_P\}_{P\in\mathscr Q}$ another family of $(K_b,L_b,M_b,N_b)$-molecules,
each on the cube indicated by its subscript.
Then the infinite matrix $\{\langle b_Q,m_P\rangle\}_{Q,P\in\mathscr Q}$
is $\dot a^{s,\tau}_{p,q}(d)$-almost diagonal provided that
\begin{equation}\label{synMol}
K_b>\widetilde J+(\widetilde{s})_-,\
L_b\geq\widetilde J-n-\widetilde s,\
M_b>\widetilde J,\
N_b>\widetilde s,
\end{equation}
and
\begin{equation}\label{anaMol}
K_m>\widetilde J+\left(\widetilde J-n-\widetilde s\right)_-,\
L_m\geq\widetilde{s},\
M_m>\widetilde{J},\
N_m>\widetilde{J}-n-\widetilde{s},
\end{equation}
where $\widetilde J$ and $\widetilde s$ are the same as in \eqref{tauJ2}.
\end{theorem}

\begin{proof}
For $\widetilde J$ in \eqref{tauJ2}, we notice that $\widetilde J\geq n$, and hence $K_b,K_m,M_b,M_m>n$ under the assumptions of the theorem.
Thus, Lemma \ref{like B3} guarantees that
$|\langle b_Q,m_P\rangle|\lesssim b_{Q,P}^{MGH}$ for any $Q,P\in\mathscr Q$
with $M,G,H$ the same as in Lemma \ref{like B3}.

Let us compare this conclusion with Definition \ref{def ad tau} of the $\dot a^{s,\tau}_{p,q}(d)$-almost diagonal matrix
which is $(D,E,F)$-almost diagonal with
$D>\widetilde{J},$ $E>\frac{n}{2}+\widetilde{s},$ and
$F>\widetilde J-\frac{n}{2}-\widetilde s.$
Thus, we find that, to show the present theorem,
it suffices to prove that, for some $\alpha>0$,
\begin{equation}\label{3x >}
\begin{aligned}
&K_b\wedge M_b\wedge K_m\wedge M_m>\widetilde J,\\
&[N_b\wedge\lceil\!\lceil L_m\rceil\!\rceil\wedge(K_m-n-\alpha)]_+>\widetilde s,\text{ and }\\
&[N_m\wedge\lceil\!\lceil L_b\rceil\!\rceil\wedge(K_b-n-\alpha)]_+>\widetilde J-n-\widetilde s.
\end{aligned}
\end{equation}

To simplify these conditions, we claim that, without loss of generality,
we may drop the positive part operator $(\cdot)_+$ from the above conditions.
In fact, first, if $x_+>y\geq 0$, then indeed $x>0$ (because otherwise $x_+=0$),
and hence $x=x_+>y$. Suppose then that $y<0$. Since $x_+\geq 0>y$,
in general we could not make any conclusions about $x$ in this case.
However, let us check that, in the particular case of \eqref{3x >},
both quantities $x$ inside $(\cdot)_+$ may also be assumed to satisfy $x>y$
for any negative $y$ without loss of generality.
This depends on the fact that the numbers here are molecular parameters.
If a $(K,L,M,N)$-molecule has $N<0$ (resp. $L<0$),
then there are no derivative (resp. no cancellation) conditions assumed,
and the precise value of $N$ (resp. $L$) is irrelevant,
as long as it remains negative.
If $y<0$, then a condition of the type $N>y$ still allows $N\in(y,0)$,
which is a void condition. Similarly, $\lceil\!\lceil L\rceil\!\rceil>y$ allows $\lceil\!\lceil L\rceil\!\rceil=0$,
and hence $L\in[-1,0)$, which is again a void condition.
Finally, the first condition in \eqref{3x >} already implies that
both $K\in\{K_b,K_m\}$ satisfy $K>\widetilde J\geq n$,
and hence $K-n-\alpha>y$ for any given negative $y$ when $\alpha>0$ is small enough.
Thus, if $y\in\{\widetilde s,\widetilde J-n-\widetilde s\}$ is negative
and the corresponding bound in \eqref{3x >} holds as stated,
it also holds without the $(\cdot)_+$.

Also note that ``$K-n-\alpha>y$ for some $\alpha>0$'' is simply equivalent to ``$K-n>y$''.
Finally, if $\lceil\!\lceil L\rceil\!\rceil>y$, it means that
$\lfloor L\rfloor+1
=\lceil\!\lceil L\rceil\!\rceil\geq\lceil\!\lceil y\rceil\!\rceil
=\lfloor y\rfloor+1$,
and hence $L\geq\lfloor L\rfloor\geq\lfloor y\rfloor$.
On the other hand, the molecular condition involving $L$ only depends on the integer part,
so, if $L\geq\lfloor y\rfloor$,
we may as well take $L\geq y$, without changing the condition.

With these simplifications observed, we can readily write down
the conditions for each of the eight molecular parameters from \eqref{3x >}.
For $K_m$, we have two different lower bounds that combine to give
\begin{equation*}
K_m
>\widetilde J\vee(n+\widetilde s)
=\widetilde J\vee\left[\widetilde J-\left(\widetilde J-n-\widetilde s\right)\right]
=\widetilde J+\left(\widetilde J-n-\widetilde s\right)_-.
\end{equation*}
Similarly, for $K_m$, we have
$K_b>\widetilde{J}\vee\left(\widetilde{J}-\widetilde{s}\right)
=\widetilde J+(\widetilde s)_-.$
For all other six parameters, there exists only one constraint for each in \eqref{3x >},
and these are seen to be consistent with the assertions of the present lemma
by the observations that we made above.
This finishes the proof of Theorem \ref{adMol}.
\end{proof}

To facilitate talking about molecular conditions as in \eqref{anaMol} and \eqref{synMol},
we give the following definition.

\begin{definition}\label{def Js mol}
A \emph{$(J,s)$-molecule} is any $(K,L,M,N)$-molecule for parameters that satisfy
$K>J+s_-,$ $L\geq J-n-s,$ $M>J,$ and $N>s.$
Whenever we talk about a family of $(J,s)$-molecules,
we understand that they should be the $(K,L,M,N)$-molecules
for some fixed quadruple $(K,L,M,N)$ that satisfies the conditions above.
\end{definition}

With this terminology, we further define the analysis and the synthesis molecules as follows.

\begin{definition}\label{def A mol}
Let $p\in(0,\infty)$, $q\in(0,\infty]$, $s\in\mathbb R$, $\tau\in[0,\infty)$, and $d\in[0,n)$.
Let $J$ be the same as in \eqref{J}, and
$\widetilde J$ and $\widetilde s$ the same as in \eqref{tauJ2}. Then
\begin{enumerate}[\rm(i)]
\item A $(\widetilde{J},\widetilde{s})$-molecule on a cube $Q$ is called
an \emph{$\dot A^{s,\tau}_{p,q}(d)$-synthesis molecule} on $Q$;
\item A $(\widetilde{J},\widetilde{J}-n-\widetilde{s})$-molecule on a cube $Q$ is called
an \emph{$\dot A^{s,\tau}_{p,q}(d)$-analysis molecule} on $Q$;
\item A $(J,s)$-molecule on a cube $Q$ is called
an \emph{$\dot A^{s}_{p,q}$-synthesis molecule} on $Q$;
\item A $(J,J-n-s)$-molecule on a cube $Q$ is called
an \emph{$\dot A^{s}_{p,q}$-analysis molecule} on $Q$.
\end{enumerate}
We apply the same convention to families of such molecules
as in Definition \ref{def Js mol}.
Whenever we talk about a family of molecules $\{m_Q\}_{Q\in\mathscr Q}$
indexed by the family $\mathscr Q$ of dyadic cubes, we understand that each $m_Q$
is a molecule on the respective cube $Q\in\mathscr Q$.
\end{definition}

\begin{remark}
Both (iii) and (iv) of Definition \ref{def A mol}
are equivalent to the classical definition (see \cite[p.\,56 and p.\,57]{fj90}).
\end{remark}

While both $\widetilde J$ and $\widetilde s$, as defined in \eqref{tauJ2},
depend on the number $\widehat\tau$, and thus on the $A_p$-dimension $d$,
it is interesting to observe that their difference does not.
In fact, we have the following conclusion.

\begin{lemma}\label{J minus s}
$\widetilde s-\widetilde J
=s-J_\tau+n(\tau-\frac{1}{p})_+.$
\end{lemma}

\begin{proof}
This is a direct computation using the definitions of the various constants in \eqref{tauJ2}:
$$
\widetilde s-\widetilde J
=\left(s+n\widehat\tau\right)-\left[J_\tau+\left(n\widehat\tau\wedge\frac{d}{p}\right)\right]
=s-J_\tau-n+\left(n\widehat\tau-\frac{d}{p}\right)_+
$$
and
$$
\left(n\widehat\tau-\frac{d}{p}\right)_+
=\left(\left[n\tau-\frac{n}{p}+\frac{d}{p}\right]_+-\frac{d}{p}\right)_+
=n\left(\tau-\frac{1}{p}\right)_+,
$$
which complete the proof of Lemma \ref{J minus s}.
\end{proof}

\begin{lemma}\label{mol no canc}
A $(\widetilde J,\widetilde s)$-molecule has no cancellation requirements if and only if
\begin{equation*}
s>
\begin{cases}
\displaystyle-n\left(\tau-\frac{1}{p}\right)&\text{if }\tau>\frac{1}{p}\text{ or }(\tau,q)=(\frac{1}{p},\infty)\quad\textup{(}\text{``supercritical case''}\textup{)},\\
\displaystyle n\left(\frac{1}{q}-1\right)_+ &\text{if }\dot a^{s,\tau}_{p,q}=\dot f^{s,\frac{1}{p}}_{p,q}\text{ and }q<\infty\quad\textup{(}\text{``critical case''}\textup{)},\\
J-n&\text{if }\tau<\frac1p\text{, or }\dot a^{s,\tau}_{p,q}=\dot b^{s,\frac1p}_{p,q}\text{ and }q<\infty\quad\textup{(}\text{``subcritical case''}\textup{)}.
\end{cases}
\end{equation*}
\end{lemma}

\begin{proof}
From Definitions \ref{def Js mol} and \ref{def A mol},
we deduce that the degree of cancellation of a $(\widetilde J,\widetilde s)$-molecule
is $L>\widetilde J-n-\widetilde s$.
In order to have no cancellation conditions,
it is necessary and sufficient that $L<0$,
which is possible if and only if $\widetilde J-n-\widetilde s<0$.
By Lemma \ref{J minus s}, we find that
\begin{equation*}
\widetilde s+n-\widetilde J
=s-(J_\tau-n)+n\left(\tau-\frac{1}{p}\right)_+,
\end{equation*}
where, by \eqref{tildeJ},
\begin{equation*}
J_\tau-n
=\begin{cases}
0&\text{if }\tau>\frac{1}{p}\text{ or }(\tau,q)=(\frac{1}{p},\infty)\quad\textup{(}\text{``supercritical case''}\textup{)},\\
\displaystyle n\left(\frac{1}{q}-1\right)_+ &\text{if }\dot a^{s,\tau}_{p,q}=\dot f^{s,\frac{1}{p}}_{p,q}\text{ and }q<\infty\quad\textup{(}\text{``critical case''}\textup{)},\\
J-n&\text{if }\tau<\frac1p\text{, or }\dot a^{s,\tau}_{p,q}=\dot b^{s,\frac1p}_{p,q}\text{ and }q<\infty\quad\textup{(}\text{``subcritical case''}\textup{)}.
\end{cases}
\end{equation*}
The claim of the present lemma is immediate from a combination of these identities.
This finishes the proof of Lemma \ref{mol no canc}.
\end{proof}

The following lemma connects
$\dot A^{s,\tau}_{p,q}(d)$-analysis molecules (resp. synthesis molecules)
and $\dot A^{s}_{p,q}$-analysis molecules (resp. synthesis molecules),
and its proof is essentially the same as the proof of \cite[Lemma 9.5]{bhyy2};
we omit the details here.

\begin{lemma}\label{mol vs old}
Let $s\in\mathbb R$, $\tau\in[0,\infty)$, $p\in(0,\infty)$, $q\in(0,\infty]$, $d\in[0,n)$,
and both $\widetilde J$ and $\widetilde s$ be the same as in \eqref{tauJ2}, and let
$\widetilde r:=\frac{n}{\widetilde{J}}\in(0,1].$
Let $\widetilde p\in(0,\infty)$ and $\widetilde q\in(0,\infty]$ be any numbers such that
$$
\begin{cases}
\widetilde p=\widetilde r&\text{if }\dot A=\dot B,\\
\widetilde p\wedge\widetilde q=\widetilde r&\text{if }\dot A=\dot F,
\end{cases}
$$
and let $Q\in\mathscr{Q}$.
Then the following conditions are equivalent for a function $f$:
\begin{enumerate}[\rm(i)]
\item $f$ is an $\dot A^{s,\tau}_{p,q}(d)$-analysis molecule on $Q$;
\item $f$ is an $\dot A^{\widetilde s}_{\widetilde p,\widetilde q}$-analysis molecule on $Q$;
\item $f$ is an $\dot A^{\widetilde s}_{\widetilde r,\widetilde r}$-analysis molecule on $Q$.
\end{enumerate}
The same is true with ``analysis'' replaced by ``synthesis'' throughout.
\end{lemma}

As a consequence of Theorem \ref{adMol} and Definition \ref{def A mol},
we have the following corollary;
see also \cite[Lemma 4.1]{yy10} or \cite[Corollaries 5.2 and 5.3]{bh06}.

\begin{corollary}\label{83}
Let $s\in\mathbb R$, $\tau\in[0,\infty)$, $p\in(0,\infty)$, $q\in(0,\infty]$, and $d\in[0,n)$.
Let $\varphi,\psi\in\mathcal{S}_\infty$ satisfy \eqref{19} and \eqref{20}.
Suppose, for both $i=1,2$, that $\{m_Q^{(i)}\}_{Q\in\mathscr{Q}}$
are families of $\dot A^{s,\tau}_{p,q}(d)$-analysis molecules and
$\{b_Q^{(i)}\}_{Q\in\mathscr{Q}}$
families of $\dot A^{s,\tau}_{p,q}(d)$-synthesis molecules.
Then
\begin{enumerate}[\rm(i)]
\item\label{831}
each of the matrices
\begin{equation*}
\left\{\left\langle m_P^{(1)},b_Q^{(1)}\right\rangle\right\}_{P,Q\in\mathscr{Q}},\
\left\{\left\langle m_P^{(1)},\psi_Q\right\rangle\right\}_{P,Q\in\mathscr{Q}},\text{ and }
\left\{\left\langle \varphi_P, b_Q^{(1)}\right\rangle\right\}_{P,Q\in\mathscr{Q}}
\end{equation*}
is $\dot a^{s,\tau}_{p,q}(d)$-almost diagonal;
\item\label{832} if $\vec t=\{\vec t_R\}_{R\in\mathscr Q}\in\dot a^{s,\tau}_{p,q}(W)$, where $W\in A_p$ has the $A_p$-dimension $d$, then
\begin{equation*}
\vec s_P:=\sum_{Q,R\in\mathscr Q}
\left\langle m_P^{(1)},b_Q^{(1)}\right\rangle
\left\langle m_Q^{(2)},b_R^{(2)}\right\rangle\vec t_R
\end{equation*}
converges unconditionally for each $P\in\mathscr Q$,
and $\vec s:=\{\vec s_P\}_{P\in\mathscr Q}$ satisfies
\begin{equation*}
\left\|\vec s\right\|_{\dot a^{s,\tau}_{p,q}(W)}
\leq C\left\|\vec t\right\|_{\dot a^{s,\tau}_{p,q}(W)},
\end{equation*}
where the positive constant $C$ is independent of $\vec t$,
$\{m_Q\}_{Q\in\mathscr Q}$, and $\{b_Q\}_{Q\in\mathscr Q}$.
\end{enumerate}
\end{corollary}

\begin{proof}
Part \eqref{831} is immediate from Theorem \ref{adMol},
observing that both $\varphi_Q$ and $\psi_Q$ are (harmless constant multiples of)
both an $\dot A^{s,\tau}_{p,q}(d)$-analysis molecule
and an $\dot A^{s,\tau}_{p,q}(d)$-synthesis molecule on $Q\in\mathscr{Q}$.

Using part \eqref{831} and the trivial observation that almost diagonality
only depends on the absolute values, we find that
$\{|\langle m_P^{(1)}, b_Q^{(1)}\rangle|\}_{P,Q\in\mathscr{Q}}$
and $\{|\langle m_Q^{(2)}, b_R^{(2)}\rangle|\}_{Q,R\in\mathscr{Q}}$
are $\dot a^{s,\tau}_{p,q}(d)$-almost diagonal.
Then \cite[Corollary 9.6]{bhyy1} guarantees that the composition
$b_{P,R}:=\sum_{Q\in\mathscr Q}
|\langle m_P^{(1)},b_Q^{(1)}\rangle|
|\langle m_Q^{(2)}, b_R^{(2)}\rangle|$
also defines an $\dot a^{s,\tau}_{p,q}(d)$-almost diagonal
$B=\{b_{P,R}\}_{P,R\in\mathscr Q}$.
Now Lemma \ref{ad BF2} implies the absolute convergence
\begin{equation*}
\sum_{Q,R\in\mathscr Q}
\left|\left\langle m_P^{(1)},b_Q^{(1)}\right\rangle\right|
\left|\left\langle m_Q^{(2)}, b_R^{(2)}\right\rangle\right|
\left|\vec t_R\right|
=\sum_{R\in\mathscr Q}b_{P,R}\left|\vec t_R\right|<\infty
\end{equation*}
and the claimed norm estimate for $\vec s$.
This finishes the proof of part \eqref{832} and hence Corollary \ref{83}.
\end{proof}

If $\vec f\in\dot A^{s,\tau}_{p,q}(W)$ and $\Phi$ is an analysis molecule,
then $\vec f\in(\mathcal{S}_\infty')^m$, but $\Phi$ might not be in $\mathcal{S}_\infty$,
so we need to justify the meaningfulness of the pairing $\langle\vec f,\Phi\rangle$;
see \cite[p.\,155]{fj90} or \cite[Lemma 5.7]{bh06} or \cite[Lemma 4.2]{yy10}.

\begin{lemma}\label{88}
Let $s\in\mathbb R$, $\tau\in[0,\infty)$, $p\in(0,\infty)$, $q\in(0,\infty]$,
and $W\in A_p$ have the $A_p$-dimension $d\in[0,n)$.
Let $\vec f\in\dot A^{s,\tau}_{p,q}(W)$ and $\Phi$ be an
$\dot A^{s,\tau}_{p,q}(d)$-analysis molecule on a cube $P\in\mathscr{Q}$.
Then, for any $\varphi,\psi\in\mathcal{S}$
satisfying \eqref{19}, \eqref{20}, and \eqref{21}, the pairing
\begin{equation}\label{881}
\left\langle\vec f,\Phi\right\rangle
:=\sum_{R\in\mathscr{Q}}\left\langle\vec f,\varphi_R\right\rangle\langle\psi_R,\Phi\rangle
\end{equation}
is well defined: the series above converges absolutely and its value is independent of the choices of $\varphi$ and $\psi$.
\end{lemma}

\begin{proof}
Suppose that $\varphi^{(i)},\psi^{(i)}$, for $i=1,2$,
are two pairs of functions that satisfy  \eqref{19}, \eqref{20}, and \eqref{21}.
By Corollary \ref{83}\eqref{832} applied to
$\vec t_R:=\langle\vec f,\varphi_R^{(1)}\rangle$,
$b_R:=\psi_R^{(1)}$, and $m_P:=\Phi$, we find that
\begin{equation}\label{880}
\sum_{Q,R\in\mathscr Q}\left\langle\vec f,\varphi_R^{(1)}\right\rangle
\left\langle\psi_R^{(1)},\varphi_Q^{(2)}\right\rangle
\left\langle\psi_Q^{(2)},\Phi\right\rangle
\end{equation}
converges absolutely. By Lemma \ref{7},
the claim \eqref{881} holds if either $\Phi$ or $\vec f$
is replaced by a test function $\Psi\in\mathcal S_\infty$.
Thus, if we sum up the double series \eqref{880} in one or the other order,
as justified by the absolute convergence that we just showed,
we then conclude that
\begin{align*}
\sum_{R\in\mathscr Q}\left\langle\vec f,\varphi_R^{(1)}\right\rangle
\left\langle\psi_R^{(1)},\Phi\right\rangle
&=\sum_{R\in\mathscr Q}\left\langle\vec f,\varphi_R^{(1)}\right\rangle
\left(\sum_{Q\in\mathscr Q}
\left\langle\psi_R^{(1)},\varphi_Q^{(2)}\right\rangle
\left\langle\psi_Q^{(2)},\Phi\right\rangle\right) \\
&=\sum_{Q\in\mathscr Q}\left(\sum_{R\in\mathscr Q}
\left\langle\vec f,\varphi_R^{(1)}\right\rangle
\left\langle\psi_R^{(1)},\varphi_Q^{(2)}\right\rangle\right)
\left\langle\psi_Q^{(2)},\Phi\right\rangle
=\sum_{Q\in\mathscr Q}\left\langle\vec f,\varphi_Q^{(2)}\right\rangle
\left\langle\psi_Q^{(2)},\Phi\right\rangle.
\end{align*}
This proves both the absolute convergence of the right-hand side of \eqref{881}
and its independence of $\varphi$ and $\psi$,
which completes the proof of Lemma \ref{88}.
\end{proof}

Applying Corollary \ref{83}, Lemma \ref{88}, and a method pioneered by
Frazier and Jawerth used in the proofs of \cite[Theorems 3.5 and 3.7]{fj90},
we obtain the following conclusion.

\begin{theorem}\label{89}
Let $s\in\mathbb R$, $\tau\in[0,\infty)$, $p\in(0,\infty)$, $q\in(0,\infty]$,
and $W\in A_p$ have the $A_p$-dimension $d\in[0,n)$.
\begin{enumerate}[{\rm (i)}]
\item\label{890} If $\{m_Q\}_{Q\in\mathscr{Q}}$
is a family of $\dot A^{s,\tau}_{p,q}(d)$-analysis molecules, each on the cube indicated by its subscript,
then there exists a positive constant $C$ such that,
for any $\vec{f}\in\dot A^{s,\tau}_{p,q}(W)$,
$$
\left\|\left\{\left\langle\vec f,
m_Q\right\rangle\right\}_{Q\in\mathscr{Q}}\right\|_{\dot a^{s,\tau}_{p,q}(W)}
\leq C\left\|\vec{f}\right\|_{\dot A^{s,\tau}_{p,q}(W)}.
$$
\item\label{892} If $\{b_Q\}_{Q\in\mathscr{Q}}$
is a family of $\dot A^{s,\tau}_{p,q}(d)$-synthesis molecules,
each on the cube indicated by its subscript,
then, for any $\vec t:=\{\vec t_Q\}_{Q\in\mathscr{Q}}
\in\dot a^{s,\tau}_{p,q}(W)$, there exist $\vec f\in\dot A^{s,\tau}_{p,q}(W)$ such that
$\vec f=\sum_{Q\in\mathscr{Q}}\vec t_Qb_Q$
in $(\mathcal{S}_\infty')^m$ and a positive constant $C$,
independent of both $\{\vec t_Q\}_{Q\in\mathscr{Q}}$
and $\{b_Q\}_{Q\in\mathscr{Q}}$, such that
$\|\vec f\|_{\dot A^{s,\tau}_{p,q}(W)}
\leq C\|\vec t\|_{\dot a^{s,\tau}_{p,q}(W)}.$
\end{enumerate}
\end{theorem}

\begin{proof}
We first show (i).
Let $\varphi,\psi\in\mathcal{S}$ satisfy \eqref{19}, \eqref{20}, and \eqref{21}.
By Lemma \ref{88}, we find that, for any $Q\in\mathscr{Q}$,
\begin{equation}\label{91}
\left\langle\vec f,m_Q\right\rangle
=\sum_{R\in\mathscr{Q}}
\left\langle\vec f,\varphi_R\right\rangle
\left\langle\psi_R,m_Q\right\rangle
=\sum_{R\in\mathscr{Q}}
\left\langle\psi_R,m_Q\right\rangle
\left(S_\varphi\vec f\right)_R.
\end{equation}
Using Corollary \ref{83}(i), we conclude that
$B:=\{\left\langle\psi_R,m_Q\right\rangle\}_{Q,R\in\mathscr{Q}}$
is an $\dot a^{s,\tau}_{p,q}(d)$-almost diagonal operator.
From this, \eqref{91}, Lemma \ref{ad BF2}, and \cite[Theorem 3.29]{bhyy1}, we infer that
$$
\left\|\left\{\left\langle\vec f,
m_Q\right\rangle\right\}_{Q\in\mathscr{Q}}\right\|_{\dot a^{s,\tau}_{p,q}(W)}
=\left\|B\left(S_\varphi\vec f\right)\right\|_{\dot a^{s,\tau}_{p,q}(W)}
\lesssim\left\|S_\varphi\vec f\right\|_{\dot a^{s,\tau}_{p,q}(W)}
\lesssim\left\|\vec f\right\|_{\dot A^{s,\tau}_{p,q}(W)}.
$$
This finishes the proof of (i).

Now, we prove (ii). By Corollary \ref{83}\eqref{832}, with any $\phi\in\mathcal S_\infty$ in place of the molecule $m_P$, we find the absolute convergence of
\begin{equation*}
\sum_{Q,R\in\mathscr Q}\vec{t}_R\langle b_R,\varphi_Q\rangle\langle\psi_Q,\phi\rangle
=\sum_{R\in\mathscr Q}\vec{t}_R\langle b_R,\phi\rangle
=:\left\langle \vec f,\phi\right\rangle,
\end{equation*}
where the first identity follows from summing first over $Q\in\mathscr Q$
(as we can, by absolute convergence), and applying Lemma \ref{7}.
In particular, the sum of the series is a well-defined element
$\vec f\in\mathcal S_\infty'$. Taking $\phi=\varphi_P$,
we have $\langle \vec f,\varphi_P\rangle=(S_\varphi f)_P$, and we deduce
\begin{equation*}
\left\|\vec f\right\|_{\dot A^{s,\tau}_{p,q}(W)}
\sim\left\|S_\varphi\vec f\right\|_{\dot a^{s,\tau}_{p,q}(W)}
\lesssim\left\|\vec t\right\|_{\dot a^{s,\tau}_{p,q}(W)}
\end{equation*}
from the $\varphi$-transform characterization of $\dot A^{s,\tau}_{p,q}(W)$
(see \cite[Theorem 3.29]{bhyy1}) in the first step
and the norm estimate of Corollary \ref{83}\eqref{832} in the second step.
This finishes the proof of Theorem \ref{89}.
\end{proof}

For any $\phi\in\mathcal{S}$ and $M\in\mathbb Z_+$,
$$
\|\phi\|_{S_M}
:=\sup_{\gamma\in\mathbb{Z}_+^n,\,|\gamma|\leq M}
\sup_{x\in\mathbb{R}^n}|\partial^\gamma\phi(x)|(1+|x|)^{n+M+|\gamma|}.
$$
As an application of Theorem \ref{89}, we obtain the following conclusion.

\begin{proposition}\label{6.18}
Let $s\in\mathbb R$, $\tau\in[0,\infty)$, $p\in(0,\infty)$, $q\in(0,\infty]$, and $W\in A_p$.
Then $(\mathcal{S}_\infty)^m\subset\dot A^{s,\tau}_{p,q}(W)$.
Moreover, there exist $M\in\mathbb{Z}_+$
and a positive constant $C$ such that,
for any $\vec f\in(\mathcal{S}_\infty)^m$,
$$
\left\|\vec f\right\|_{\dot A^{s,\tau}_{p,q}(W)}
\leq C\left\|\vec f\right\|_{\vec{S}_M}
:=C\sup_{\gamma\in\mathbb{Z}_+^n,\,|\gamma|\leq M}
\sup_{x\in\mathbb{R}^n}\left|\left(\partial^\gamma\vec f\right)(x)\right|(1+|x|)^{n+M+|\gamma|}.
$$
\end{proposition}

\begin{proof}
Let $\vec f:=(f_1,\ldots,f_m)\in(\mathcal{S}_\infty)^m$,
$Q\in\mathscr{Q}$, and,
for any $i\in\{1,\ldots,m\}$,
$$
\vec e_i:=(0,\ldots,0,1,0,\ldots,0)^{\mathrm{T}}\in\mathbb{C}^m,
$$
where only the $i$-th component is $1$.
Then $\vec f=\sum_{i=1}^m\vec e_if_i$.

Let $W\in A_p$ have the $A_p$-dimension $d\in[0,n)$
and $\widetilde J$ and $\widetilde s$ be the same as in \eqref{tauJ2}.
Notice that, for any $i\in\{1,\ldots,m\}$, $f_i\in\mathcal{S}_\infty$.
By this and the definitions of synthesis molecules
and $\|\cdot\|_{S_M}$, we conclude that,
for any $M>\widetilde J-n+(\widetilde s)_+$,
the scaled function $[c_n\|f_i\|_{S_M}]^{-1}f_i$
is an $\dot A^{s,\tau}_{p,q}(d)$-synthesis molecule on $Q_{0,\mathbf{0}}$,
where $c_n:=2\wedge\sqrt{n}$.
From this and Theorem \ref{89}(ii), we infer that,
for any $i\in\{1,\ldots,m\}$,
$\|\vec e_if_i\|_{\dot A^{s,\tau}_{p,q}(W)}
\lesssim\|f_i\|_{S_M}$
and hence
$$
\left\|\vec f\right\|_{\dot A^{s,\tau}_{p,q}(W)}
\sim\sum_{i=1}^m\left\|\vec e_if_i\right\|_{\dot A^{s,\tau}_{p,q}(W)}
\lesssim\sum_{i=1}^m\|f_i\|_{S_M}
\sim\left\|\vec f\right\|_{\vec{S}_M},
$$
which completes the proof of Proposition \ref{6.18}.
\end{proof}

\begin{remark}
Notice that $(\mathcal{S}_\infty)^m$ need not be dense in $\dot A^{s,\tau}_{p,q}(W)$ when $\tau\in(0,\infty)$.
Indeed, by \cite[Proposition 3.2(vii)]{yy08}, we have $\dot A_{2,2}^{0,\frac12}=\mathrm{BMO}$,
but it is well known that $\mathcal{S}_\infty$ is not dense in $\mathrm{BMO}$.
\end{remark}

\subsection{Pseudo-Differential Operators\label{pdo}}

Now, we recall the concepts of the class $\dot{S}_{1,1}^u$
and its related pseudo-differential operators
(see, for instance, \cite[pp.\,261-263]{gt99}).

\begin{definition}
Let $u\in\mathbb{Z}_+$.
The class $\dot{S}_{1,1}^u$ is defined to be the set of all functions
$a\in C^\infty(\mathbb{R}^n\times(\mathbb{R}^n\setminus\{\mathbf{0}\}))$
such that, for any $\alpha,\beta\in\mathbb{Z}_+^n$,
$$
\sup_{x\in\mathbb{R}^n,\,\xi\in\mathbb{R}^n\setminus\{\mathbf{0}\}}
|\xi|^{-u-|\alpha|+|\beta|}
\left|\partial_x^\alpha\partial_\xi^\beta a(x,\xi)\right|
<\infty.
$$
\end{definition}

\begin{definition}
Let $u\in\mathbb{Z}_+$ and $a\in\dot{S}_{1,1}^u$.
Define the pseudo-differential operator $a(x,D)$ with symbol $a$ by setting,
for any $f\in\mathcal{S}_\infty$ and $x\in\mathbb{R}^n$,
$$
a(x,D)(f):=\int_{\mathbb{R}^n}a(x,\xi)\widehat{f}(\xi)e^{ix\cdot\xi}\,d\xi.
$$
\end{definition}

The following lemma is just \cite[Lemma 2.1]{gt99}.

\begin{lemma}
Let $u\in\mathbb{Z}_+$ and $a\in\dot{S}_{1,1}^u$.
Then $a(x,D)$ is a continuous linear operator
from $\mathcal{S}_\infty$ to $\mathcal{S}$.
In particular, its formal adjoint $a(x,D)^{\#}$, defined by setting,
for any $f\in\mathcal{S}'$ and $\phi\in\mathcal{S}_\infty$,
$$
\left\langle a(x,D)^{\#}f,\phi\right\rangle
:=\langle f,a(x,D)\phi\rangle,
$$
is a continuous linear operator from $\mathcal{S}'$ to $\mathcal{S}_\infty'$.
\end{lemma}

Applying an argument similar to that used in the proof of
\cite[Theorem 1.5]{syy10} (see also the proofs of \cite[Theorems 1.1 and 1.2]{gt99}),
we obtain the following theorem.
For the convenience of the reader, we give the details of its proof here.

\begin{theorem}\label{pseudo}
Let $s\in\mathbb R$, $\tau\in[0,\infty)$, $p\in(0,\infty)$, $q\in(0,\infty]$,
and $W\in A_p$ have the $A_p$-dimension $d\in[0,n)$.
Let $\widetilde J$ and $\widetilde s$ be the same as in \eqref{tauJ2}.
Let $u\in\mathbb{Z}_+$, $a\in\dot{S}_{1,1}^u $,
and $a(x,D)$ be a pseudo-differential operator with symbol $a$.
Assume that its formal adjoint $a(x,D)^{\#}$ satisfies,
for each $\beta\in\mathbb{Z}_+^n$ with $|\beta|\leq\widetilde J-n-\widetilde s$,
\begin{equation}\label{227}
a(x,D)^{\#}\left(x^\beta\right)=0\in\mathcal{S}_\infty',
\end{equation}
where \eqref{227} is void when $\widetilde J-n-\widetilde s<0$.
Then $a(x,D)$ can be extended to a continuous linear mapping
from $\dot A_{p,q}^{s+u,\tau}(W)$ to $\dot A^{s,\tau}_{p,q}(W)$.
\end{theorem}

\begin{proof}
To simplify the presentation of the present proof,
in this proof, we denote $a(x,D)$ simply by $T$.
Let $\vec f\in\dot A_{p,q}^{s+u, \tau}(W)$
and $\varphi\in\mathcal{S}$
satisfy that $\overline{\operatorname{supp}\widehat{\varphi}}$ is compact and bounded away from the origin
and, for any $\xi\in\mathbb{R}^n\setminus\{\mathbf{0}\}$,
$\sum_{j\in\mathbb{Z}}|\widehat{\varphi}(2^{-j}\xi)|^2
=1.$
By this and Lemma \ref{7}, we find that
$$
\vec f=\sum_{Q\in\mathscr{Q}}
\left\langle\vec f,\varphi_Q\right\rangle\varphi_Q
$$
in $(\mathcal{S}_\infty')^m$, and hence it is natural to define
$$
T\left(\vec f\right)
:=\sum_{Q\in\mathscr{Q}}
\left\langle\vec f,\varphi_Q\right\rangle T\left(\varphi_Q\right)
=\sum_{Q\in\mathscr{Q}}
\left[|Q|^{-\frac{u}{n}}\left(S_\varphi\vec f\right)_Q\right]
\left[|Q|^{\frac{u}{n}}T\left(\varphi_Q\right)\right]
$$
in $(\mathcal{S}_\infty')^m$.
Now, we show that $T(\vec f)$ is well defined.
From \cite[Theorem 3.29]{bhyy1}, we deduce that
\begin{equation}\label{102}
\left\|\left\{|Q|^{-\frac{u}{n}}\left(S_\varphi\vec f\right)_Q
\right\}_{Q\in\mathscr{Q}}\right\|_{\dot a^{s,\tau}_{p,q}(W)}
=\left\|S_\varphi\vec f\right\|_{\dot a_{p,q}^{s+u,\tau}(W)}
\lesssim\left\|\vec f\right\|_{\dot A_{p,q}^{s+u,\tau}(W)}.
\end{equation}
By \eqref{227} and the proof of \cite[Theorem 1.5]{syy10},
we conclude that, for any $Q\in\mathscr{Q}$, the function
$|Q|^{\frac{u}{n}}T(\varphi_Q)$ is a harmless constant multiple of
an $\dot A^{s,\tau}_{p,q}(d)$-synthesis molecule on $Q$.
Together with \eqref{102} and Theorem \ref{89}(ii), this further implies that
$T(\vec f)$ is well defined and
$$
\left\|T\left(\vec f\right)\right\|_{\dot A^{s,\tau}_{p,q}(W)}
\lesssim\left\|\left\{|Q|^{-\frac{u}{n}}\left(S_\varphi\vec f\right)_Q
\right\}_{Q\in\mathscr{Q}}\right\|_{\dot a_{p,q}^{s,\tau}(W)}
\lesssim\left\|\vec f\right\|_{\dot A_{p,q}^{s+u,\tau}(W)}.
$$
This finishes the proof of Theorem \ref{pseudo}.
\end{proof}

\begin{remark}
When $\tau=0$, we obtain the boundedness of pseudo-differential operators
on $\dot A^s_{p,q}(W)$, which is also new.
When $m=1$ and $W\equiv 1$,
Theorem \ref{pseudo} in this case contains \cite[Theorem 1.5]{syy10}
in which $\tau$ has an upper bound.
\end{remark}

\section{Wavelet Characterizations}
\label{wavelet characterization}

In this section, we characterize the
matrix-weighted Besov-type and Triebel--Lizorkin-type spaces, respectively,
via the Meyer wavelet and the Daubechies wavelet.
Let us first recall the wavelet constructed by Lemari\'e and Meyer
(see, for instance, \cite[Theorem 2]{lm86}).

\begin{lemma}\label{wavelet basis}
There exists a sequence $\{\theta^{(i)}\}_{i=1}^{2^n-1}\subset\mathcal{S}_\infty$
of real-valued basic wavelets such that
$$
\left\{\theta^{(i)}_Q:\ i\in\{1,\ldots,2^n-1\},\ Q\in\mathscr{Q}\right\}
$$
form an orthonormal basis of $L^2$.
\end{lemma}

\begin{remark}\label{121}
In Lemma \ref{wavelet basis}, for any $f\in\mathcal{S}_\infty'$,
\begin{equation*}
f=\sum_{i=1}^{2^n-1}\sum_{Q\in\mathscr{Q}}
\left\langle f,\theta^{(i)}_Q\right\rangle\theta^{(i)}_Q
\end{equation*}
in $\mathcal{S}_\infty'$
(see, for instance, \cite[The proof of Theorem 7.20]{fjw91}).
\end{remark}

\begin{theorem}\label{wavelet}
Let $s\in\mathbb R$, $\tau\in[0,\infty)$, $p\in(0,\infty)$, $q\in(0,\infty]$, and $W\in A_p$.
Let $\{\theta^{(i)}\}_{i=1}^{2^n-1}\subset\mathcal{S}_\infty$
be the same as in Lemma \ref{wavelet basis}.
Then $\vec f\in\dot A^{s,\tau}_{p,q}(W)$ if and only if $\vec f\in(\mathcal{S}_\infty')^m$ and
$$
\left\|\vec f\right\|_{\dot A_{p,q}^{s,\tau}(W)_\mathrm{w}}
:=\sum_{i=1}^{2^n-1}
\left\|\left\{\left\langle\vec f,\theta^{(i)}_Q \right\rangle
\right\}_{Q\in\mathscr{Q}}\right\|_{\dot a^{s,\tau}_{p,q}(W)}
<\infty.
$$
Moreover, for any $\vec f\in(\mathcal{S}_\infty')^m$,
$\|\vec f\|_{\dot A^{s,\tau}_{p,q}(W)}
\sim\|\vec f\|_{\dot A_{p,q}^{s,\tau}(W)_\mathrm{w}},$
where the positive equivalence constants are independent of $\vec f$.
\end{theorem}

\begin{proof}
We first prove that, for any $\vec f\in\dot A^{s,\tau}_{p,q}(W)$,
\begin{equation}\label{123}
\left\|\vec f\right\|_{\dot A_{p,q}^{s,\tau}(W)_\mathrm{w}}
\lesssim\left\|\vec f\right\|_{\dot A^{s,\tau}_{p,q}(W)}.
\end{equation}
Let $\varphi,\psi\in\mathcal{S}$ satisfy \eqref{19}, \eqref{20}, and \eqref{21}.
By $\theta^{(i)}\in\mathcal{S}_\infty$ and Lemma \ref{7},
we find that, for any $i\in\{1,\ldots,2^n-1\}$ and $Q\in\mathscr{Q}$,
$$
\theta_Q^{(i)}=\sum_{R\in\mathscr{Q}}
\left\langle\theta^{(i)}_Q,\varphi_R\right\rangle\psi_R
$$
in $\mathcal{S}_\infty$.
Thus, for any $i\in\{1,\ldots,2^n-1\}$ and $\vec f\in(\mathcal{S}_\infty')^m$,
\begin{equation}\label{196}
\left\langle\vec f,\theta^{(i)}_Q\right\rangle
=\sum_{R\in\mathscr{Q}}
\left\langle\theta^{(i)}_Q,\varphi_R\right\rangle
\left\langle\vec f,\psi_R\right\rangle.
\end{equation}
Using $\theta^{(i)}\in\mathcal{S}_\infty$ again, we conclude that
$\{\theta^{(i)}_Q\}_{Q\in\mathscr{Q}}$ is a harmless constant multiple of
a family of $\dot A^{s,\tau}_{p,q}(d)$-synthesis molecules,
respectively, on $\{Q\}_{Q\in\mathscr{Q}}$.
From this and Corollary \ref{83}, we infer that
$B^{(i)}:=\{\langle\theta^{(i)}_Q,\varphi_R\rangle\}_{Q,R\in\mathscr{Q}}$
is an $\dot a^{s,\tau}_{p,q}(d)$-almost diagonal operator.
This, combined with \eqref{196}, Lemma \ref{ad BF2}, and \cite[Theorem 3.29]{bhyy1},
further implies that, for any $i\in\{1,\ldots,2^n-1\}$ and $\vec f\in\dot A^{s,\tau}_{p,q}(W)$,
\begin{align*}
\left\|\left\{\left\langle\vec f,\theta^{(i)}_Q\right\rangle
\right\}_{Q\in\mathscr{Q}}\right\|_{\dot a^{s,\tau}_{p,q}(W)}
&=\left\|B^{(i)}\left(S_\psi\vec f\right)\right\|_{\dot a^{s,\tau}_{p,q}(W)}
\lesssim\left\|S_\psi\vec f\right\|_{\dot a^{s,\tau}_{p,q}(W)}
\lesssim\left\|\vec f\right\|_{\dot A^{s,\tau}_{p,q}(W)}.
\end{align*}
This finishes the proof of \eqref{123}.

Now, we show that, for any $\vec f\in(\mathcal{S}_\infty')^m$
with $\|\vec f\|_{\dot A_{p,q}^{s,\tau}(W)_\mathrm{w}}<\infty$,
\begin{equation}\label{120}
\left\|\vec f\right\|_{\dot A^{s,\tau}_{p,q}(W)}
\lesssim\left\|\vec f\right\|_{\dot A_{p,q}^{s,\tau}(W)_\mathrm{w}}.
\end{equation}
Let $\vec f\in(\mathcal{S}_\infty')^m$
satisfy $\|\vec f\|_{\dot A_{p,q}^{s,\tau}(W)_\mathrm{w}}<\infty$.
Then, from the definition of $\|\cdot\|_{\dot A_{p,q}^{s,\tau}(W)_\mathrm{w}}$,
it follows that, for any $i\in\{1,\ldots,2^n-1\}$,
$$
\vec t^{(i)}
:=\left\{\vec t^{(i)}_Q \right\}_{Q\in\mathscr{Q}}
:=\left\{\left\langle\vec f,
\theta^{(i)}_Q\right\rangle\right\}_{Q\in\mathscr{Q}}
\in\dot a^{s,\tau}_{p,q}(W).
$$
By Remark \ref{121}, we obtain
\begin{equation}\label{195}
\left\|\vec f\right\|_{\dot A^{s,\tau}_{p,q}(W)}
=\left\|\sum_{i=1}^{2^n-1}\sum_{Q\in\mathscr{Q}}
\vec t^{(i)}_Q \theta^{(i)}_Q\right\|_{\dot A^{s,\tau}_{p,q}(W)}
\lesssim\sum_{i=1}^{2^n-1}\left\|\sum_{Q\in\mathscr{Q}}
\vec t^{(i)}_Q \theta^{(i)}_Q\right\|_{\dot A^{s,\tau}_{p,q}(W)}.
\end{equation}
Using $\theta^{(i)}\in \mathcal{S}_\infty$,
we conclude that $\{\theta^{(i)}_Q\}_{Q\in\mathscr{Q}}$ is a harmless constant multiple of
a family of $\dot A^{s,\tau}_{p,q}(d)$-synthesis molecules,
respectively, on $\{Q\}_{Q\in\mathscr{Q}}$.
From this, the fact that $\vec t^{(i)}\in\dot a^{s,\tau}_{p,q}(W)$, Theorem \ref{89}(ii),
and \eqref{195}, we deduce that
$$
\left\|\vec f\right\|_{\dot A^{s,\tau}_{p,q}(W)}
\lesssim\sum_{i=1}^{2^n-1}\left\|\vec t^{(i)}\right\|_{\dot a^{s,\tau}_{p,q}(W)}
=\left\|\vec f\right\|_{\dot A_{p,q}^{s,\tau}(W)_\mathrm{w}},
$$
which completes the proof of \eqref{120}
and hence Theorem \ref{wavelet}.
\end{proof}

\begin{remark}
When $m=1$ and $W\equiv 1$, Theorem \ref{wavelet}
is just \cite[Theorem 8.2]{ysy10}.
Furthermore, if $\tau=0$, Theorem \ref{wavelet}
reduces to the classical result \cite[Theorem 7.20]{fjw91}.
\end{remark}

In what follows, for any $\mathscr{N}\in\mathbb{N}$,
we use $C^{\mathscr{N}}$ to denote the set of all
$\mathscr{N}$ times continuously differentiable functions on $\mathbb{R}^n$.

\begin{definition}\label{def Dau}
Let $\mathscr{N}\in\mathbb{N}$. We say that $\{\theta^{(i)}\}_{i=1}^{2^n-1}$
are Daubechies wavelets of class $C^{\mathscr N}$ if each
$\theta^{(i)}\in C^{\mathscr N}$ is real-valued with bounded support and if
$$
\left\{\theta^{(i)}_Q :\ i\in\{1,\ldots,2^n-1\},\ Q\in\mathscr{Q}\right\}
$$
is an orthonormal basis of $L^2$.
\end{definition}

The following wavelet basis was constructed by Daubechies
(see, for instance, \cite{d88}).

\begin{lemma}\label{wavelet basis 2}
For any $\mathscr{N}\in\mathbb{N}$,
there exist Daubechies wavelets of class $C^{\mathscr N}$.
\end{lemma}

\begin{remark}\label{Dau can}
Let $\{\theta^{(i)}:\ i\in\{1,\ldots,2^n-1\}\}$ be Daubechies wavelets of class $C^{\mathscr N}$.
By \cite[Corollary 5.5.2]{d92}, we find that,
for every $\alpha\in\mathbb{Z}_+^n$ with $|\alpha|\leq\mathscr{N}$, the following cancellation conditions are satisfied:
$$
\int_{\mathbb{R}^n}x^\alpha\theta^{(i)}(x)\,dx=0.
$$
\end{remark}

\begin{corollary}\label{88 corollary}
Let $s\in\mathbb R$, $\tau\in[0,\infty)$, $p\in(0,\infty)$, $q\in(0,\infty]$, and $d\in[0,n)$.
Let  $\widetilde J$ and $\widetilde s$ be the same as in \eqref{tauJ2},
and let $\mathscr{N}\in\mathbb{N}$ satisfy
\begin{equation}\label{NDau}
\mathscr{N}>\max\left\{\widetilde J-n-\widetilde s,\widetilde s\right\},
\end{equation}
and let $\{\theta^{(i)}\}_{i=1}^{2^n-1}$ be Daubechies wavelets of class $C^{\mathscr{N}}$.
\begin{enumerate}[\rm(i)]
\item\label{88c1} For each $i\in\{1,\ldots,2^n-1\}$ and $Q\in\mathscr Q$,
the wavelet $\theta^{(i)}_Q$ is a constant multiple of both an $\dot A^{s,\tau}_{p,q}(d)$-analysis molecule and an  $\dot A^{s,\tau}_{p,q}(d)$-synthesis molecule on $Q$, where the constant is independent of $Q\in\mathscr Q$.
\item\label{88c2} Let $W\in A_p$ have the $A_p$-dimension $d$, let $\varphi,\psi\in\mathcal{S}$ satisfy \eqref{19}, \eqref{20}, and \eqref{21}, and let $\vec f\in\dot A^{s,\tau}_{p,q}(W)$.
Then, for any $ i\in\{1,\ldots,2^n-1\}$ and $Q\in\mathscr{Q}$,
$$
\left\langle\vec f,\theta^{(i)}_Q\right\rangle
:=\sum_{R\in\mathscr{Q}}
\left\langle\vec f,\varphi_R\right\rangle
\left\langle\psi_R,\theta^{(i)}_Q\right\rangle
$$
is well defined: the series converges absolutely
and its value is independent of the choice of $\varphi$ and $\psi$.
\end{enumerate}
\end{corollary}

\begin{proof}
\eqref{88c1} Recall from Definition \ref{moleKLMN} that the parameters of a $(K,L,M,N)$-molecule describe its decay, cancellation, decay of derivatives, and smoothness, respectively. For functions of the form $\theta_Q^{(i)}$, where $\theta^{(i)}$ has bounded support, the decay conditions are automatic, and it remains to consider cancellation and smoothness.

Now suppose that $\{\theta^{(i)}\}_{i=1}^{2^n-1}$ are Daubechies wavelets
of class $C^{\mathscr{N}}$. Then $\theta^{(i)}_Q$ clearly has smoothness $N\geq\mathscr N$
by definition and also cancellation $L\geq\mathscr N$ by Remark \ref{Dau can}.
If $\mathscr N$ is the same as in \eqref{NDau},
then in particular $L\geq\widetilde J-n-\widetilde s$ and $N>\widetilde s$,
as required for a $(\widetilde J,\widetilde s)$-molecule,
and also $L\geq\widetilde s$ and $N>\widetilde J-n-\widetilde s$,
as required for a $(\widetilde J,\widetilde J-n-\widetilde s)$-molecule,
according to Definition \ref{def Js mol}.

Finally, these are precisely the conditions for an $\dot A^{s,\tau}_{p,q}(d)$-synthesis molecule and an $\dot A^{s,\tau}_{p,q}(d)$-analysis molecule, by Definition \ref{def A mol}.

\eqref{88c2} This is an immediate consequence of part \eqref{88c1} and Lemma \ref{88}.
This finishes the proof of Corollary \ref{88 corollary}.
\end{proof}

\begin{remark}
In Corollary \ref{88 corollary},
if $\vec f\in(L^2)^m\cap\dot A^{s,\tau}_{p,q}(W)$, then,
for any $i\in\{1,\ldots,2^n-1\}$ and $Q\in\mathscr{Q}$,
$$
\left\langle\vec f,\theta^{(i)}_Q\right\rangle
=\int_{\mathbb{R}^n}\vec f(x)\overline{\theta^{(i)}_Q(x)}\,dx
$$
because $\vec f=\sum_{R\in\mathscr{Q}}\langle\vec f,\varphi_R\rangle\psi_R$ in $(L^2)^m$.
\end{remark}

Next, we can establish the wavelet decomposition of $\dot A^{s,\tau}_{p,q}(W)$
via Daubechies wavelets.

\begin{theorem}\label{wavelet 2}
Let $s\in\mathbb R$, $\tau\in[0,\infty)$, $p\in(0,\infty)$, $q\in(0,\infty]$,
and $W\in A_p$ have the $A_p$-dimension $d\in[0,n)$.
Let $\widetilde J$ and $\widetilde s$ be the same as in \eqref{tauJ2},
$\mathscr{N}\in\mathbb{N}$ satisfy \eqref{NDau},
and $\{\theta^{(i)}\}_{i=1}^{2^n-1}$ be Daubechies wavelets of class $C^{\mathscr{N}}$.
Then, for any $\vec f\in\dot A^{s,\tau}_{p,q}(W)$,
\begin{equation}\label{204}
\vec f=\sum_{i=1}^{2^n-1}\sum_{Q\in\mathscr{Q}}
\left\langle\vec f,\theta^{(i)}_Q\right\rangle\theta^{(i)}_Q
\end{equation}
in $(\mathcal{S}_\infty')^m$ and
\begin{equation}\label{2160}
\left\|\vec f\right\|_{\dot A^{s,\tau}_{p,q}(W)}
\sim\left\|\vec f\right\|_{\dot A_{p,q}^{s,\tau}(W)_\mathrm{w}}
:=\sum_{i=1}^{2^n-1}
\left\|\left\{\left\langle\vec f,\theta^{(i)}_Q \right\rangle
\right\}_{Q\in\mathscr{Q}}\right\|_{\dot a^{s,\tau}_{p,q}(W)},
\end{equation}
where the positive equivalence constants are independent of $\vec f$.
\end{theorem}

\begin{proof}
Let $\varphi,\psi\in\mathcal{S}$ satisfy \eqref{19}, \eqref{20}, and \eqref{21}.
For any $\phi\in\mathcal{S}_\infty$, we consider the series
\begin{align}\label{222}
S:=\sum_{i=1}^{2^n-1}\sum_{Q\in\mathscr{Q}}\sum_{R\in\mathscr{Q}}
\left\langle\vec f,\varphi_R\right\rangle
\left\langle\psi_R,\theta^{(i)}_Q\right\rangle
\left\langle\theta^{(i)}_Q,\phi\right\rangle.
\end{align}
Here $\vec t_R:=\langle\vec f,\varphi_R\rangle$ for any $R\in\mathscr Q$
satisfies $\vec t:=\{\vec t_R\}_{R\in\mathscr Q}\in \dot a^{s,\tau}_{p,q}(W)$
by \cite[Theorem 3.29]{bhyy1}. Moreover, both $\psi_R$ and $\theta_Q^{(i)}$
are (harmless constant multiples of) both $\dot A^{s,\tau}_{p,q}(W)$-analysis molecules
and $\dot A^{s,\tau}_{p,q}(W)$-synthesis molecules on the cubes
indicated by their subscripts, and $\phi\in\mathcal{S}_\infty$ is also
a harmless constant multiple of a molecule of either type.
Thus, the absolute convergence of \eqref{222} is an application of Corollary \ref{83}.

If we first sum over $Q$ and $i$ and use the fact that
$\{\theta_Q^{(i)}:\ Q\in\mathscr Q,\ i\in\{1,\ldots,2^n-1\}\}$
is an orthonormal basis of $L^2(\mathbb R^n)$, followed by Lemma \ref{7}, we obtain
\begin{equation*}
S=\sum_{R\in\mathscr{Q}}\left\langle\vec f,\varphi_R\right\rangle
\sum_{i=1}^{2^n-1}\sum_{Q\in\mathscr{Q}}
\left\langle\psi_R,\theta^{(i)}_Q\right\rangle
\left\langle\theta^{(i)}_Q,\phi\right\rangle
=\sum_{R\in\mathscr{Q}}\left\langle\vec f,\varphi_R\right\rangle 
\left\langle\psi_R,\phi\right\rangle
=\left\langle\vec f,\phi\right\rangle.
\end{equation*}
On the other hand, if we first sum over $R$,
then we recognise the definition of the pairing of $\vec f\in\dot A^{s,\tau}_{p,q}(W)$
with an $\dot A^{s,\tau}_{p,q}(d)$-analysis molecule $\theta^{(i)}_Q$, to the result that
\begin{equation*}
S=\sum_{i=1}^{2^n-1}\sum_{Q\in\mathscr{Q}}\left(\sum_{R\in\mathscr{Q}}\left\langle\vec f,\varphi_R\right\rangle  \left\langle\psi_R,\theta^{(i)}_Q\right\rangle\right)
\left\langle\theta^{(i)}_Q,\phi\right\rangle
=\sum_{i=1}^{2^n-1}\sum_{Q\in\mathscr{Q}}\left\langle\vec f,\theta^{(i)}_Q\right\rangle  \left\langle\theta^{(i)}_Q,\phi\right\rangle.
\end{equation*}
A combination of the last two displays proves \eqref{204} in the sense of
convergence in $(\mathcal S_\infty')^m$.

The estimate
$\|\{\langle\vec f,\theta^{(i)}_Q\rangle\}_{Q\in\mathscr{Q}}\|_{\dot a^{s,\tau}_{p,q}(W)}
\lesssim\|\vec f\|_{\dot A^{s,\tau}_{p,q}(W)}$
is immediate from Theorem \ref{89}\eqref{890} and the fact that
$\{\theta^{(i)}_Q\}_{Q\in\mathscr{Q}}$ are (harmless constant multiples of)
$\dot A^{s,\tau}_{p,q}(d)$-analysis molecules.
Thus, $\|\vec f\|_{\dot A^{s,\tau}_{p,q}(W)_\mathrm{w}}
\lesssim\|\vec f\|_{\dot A_{p,q}^{s,\tau}(W)}$
follows from summing over $i\in\{1,\ldots,2^n-1\}$.

For the converse estimate, since $\{\langle\vec f,\theta^{(i)}_Q\rangle\}_{Q\in\mathscr{Q}}
\in\dot a^{s,\tau}_{p,q}(W)$ and the functions $\{\theta^{(i)}_Q\}_{Q\in\mathscr{Q}}$
are (harmless constant multiples of) $\dot A^{s,\tau}_{p,q}(d)$-synthesis molecules,
it follows from Theorem \ref{89}\eqref{892} that the series
$\vec f^{(i)}:=\sum_{Q\in\mathscr Q}\langle\vec f,\theta^{(i)}_Q\rangle \theta^{(i)}_Q$
converges in $(\mathcal S_\infty')^m$ and satisfies the estimate
\begin{equation*}
\left\|\vec f^{(i)}\right\|_{\dot A^{s,\tau}_{p,q}(W)}
\lesssim  \left\|\left\{\left\langle\vec f,\theta^{(i)}_Q \right\rangle\right\}_{Q\in\mathscr{Q}}\right\|_{\dot a^{s,\tau}_{p,q}(W)}.
\end{equation*}
By \eqref{204}, which we already showed, there holds $\vec f=\sum_{i=1}^{2^n-1}\vec f^{(i)}$, and thus
\begin{equation*}
\left\|\vec f\right\|_{\dot A^{s,\tau}_{p,q}(W)}
\lesssim\sum_{i=1}^{2^n-1}  \left\|\vec f^{(i)}\right\|_{\dot A^{s,\tau}_{p,q}(W)}
\lesssim \sum_{i=1}^{2^n-1} \left\|\left\{\left\langle\vec f,\theta^{(i)}_Q \right\rangle\right\}_{Q\in\mathscr{Q}}\right\|_{\dot a^{s,\tau}_{p,q}(W)}
=  \left\|\vec f\right\|_{\dot A^{s,\tau}_{p,q}(W)_{\mathrm w}}
\end{equation*}
by the quasi-triangle inequality in $\dot A^{s,\tau}_{p,q}(W)$.
This finishes the proof of \eqref{2160} and hence Theorem \ref{wavelet 2}.
\end{proof}

\begin{remark}
When $\tau=0$, Theorems \ref{wavelet} and \ref{wavelet 2}
reduce to \cite[Theorem 10.2 and Corollary 10.3]{ro03} in the case of Besov spaces
with $p\in[1,\infty)$, to \cite[Theorem 4.4]{fr04}
in the case of Besov spaces with $p\in(0,1)$,
and to \cite[Theorem 1.2]{fr21} in the case of Triebel--Lizorkin spaces.
When $m=1$ and $W\equiv 1$,
Theorem \ref{wavelet 2} contains \cite[Theorem 8.3]{ysy10}, where $ s\in(0,\infty)$,
and is contained by \cite[Theorem 6.3]{lsuyy}, where Liang et al.
considered the biorthogonal wavelet system
and established the wavelet characterization of the unweighted $\dot A^{s,\tau}_{p,q}$.
\end{remark}

Wavelet characterizations for many other function spaces related to $\dot A^{s,\tau}_{p,q}(W)$ are also known;
see, for instance, \cite{fr04, is09, lyysu, r13, ro03, ysy10} for more details.

We conclude this section with a characterization of $\dot A^{s,\tau}_{p,q}(W)$
in terms of the following atoms.

\begin{definition}
Let $r,L,N\in(0,\infty)$.
A function $ a_Q $ is called an \emph{$(r,L,N)$-atom} on a cube $Q$ if,
for any $\gamma\in\mathbb{Z}_+^n$ and $x\in\mathbb{R}^n$,
\begin{enumerate}[\rm(i)]
\item $\operatorname{supp}a_Q\subset rQ$;
\item $\int_{\mathbb{R}^n}x^\gamma a_Q(x)\,dx=0$ if $|\gamma|\leq L$;
\item $|D^\gamma a_Q(x)|\leq|Q|^{-\frac12-\frac{|\gamma|}{n}}$ if $|\gamma|\leq N$.
\end{enumerate}
\end{definition}

\begin{theorem}\label{90}
Let $s\in\mathbb R$, $\tau\in[0,\infty)$, $p\in(0,\infty)$, $q\in(0,\infty]$,
and $W\in A_p$ have the $A_p$-dimension $ d\in[0,n) $.
Let $L,N\in\mathbb{R}$ satisfy
\begin{equation}\label{900}
L\geq\widetilde J-n-\widetilde s\text{ and }N>\widetilde s,
\end{equation}
where $\widetilde J$ and $\widetilde s$ are the same as in \eqref{tauJ2}.
Then there exists $r\in(0,\infty)$, depending only on $L\vee N$,
such that the following two assertions hold.
\begin{enumerate}[{\rm (i)}]
\item\label{901} For any $\vec f\in\dot A^{s,\tau}_{p,q}(W)$,
there exist a sequence $\vec t:=\{\vec t_Q \}_{Q\in\mathscr{Q}}\in\dot a^{s,\tau}_{p,q}(W)$
and $(r,L,N)$-atoms $\{a_Q\}_{Q\in\mathscr{Q}}$, each on the cube indicated by its subscript,
such that $\vec f=\sum_{Q\in\mathscr{Q}}\vec t_Q a_Q$ in $(\mathcal{S}_\infty')^m$.
Moreover, there exists a positive constant $C$,
independent of $\vec f$, such that
$$
\left\|\vec t \right\|_{\dot a^{s,\tau}_{p,q}(W)}
\leq C\left\|\vec{f}\right\|_{\dot A^{s,\tau}_{p,q}(W)}.
$$
\item\label{902} Conversely, if $\{a_Q\}_{Q\in\mathscr{Q}}$
is a family of $(r,L,N)$-atoms,
each on the cube indicated by its subscript,
then, for any $\vec t:=\{\vec t_Q\}_{Q\in\mathscr{Q}}\in\dot a^{s,\tau}_{p,q}(W)$,
there exists $\vec f\in\dot A^{s,\tau}_{p,q}(W)$ such that
$\vec f=\sum_{Q\in\mathscr{Q}}\vec t_Qa_Q$ in $(\mathcal{S}_\infty')^m$.
Moreover, there exists a positive constant $C$,
independent of both $\vec t$ and $\{a_Q\}_{Q\in\mathscr{Q}}$, such that
$\|\vec f\|_{\dot A^{s,\tau}_{p,q}(W)}
\leq C\|\vec t\|_{\dot a^{s,\tau}_{p,q}(W)}.$
\end{enumerate}
\end{theorem}

\begin{proof}
\eqref{902} An $(r,L,N)$-atom is a $(K,L,M,N)$-molecule for every $K$ and $M$
and hence an $\dot A^{s,\tau}_{p,q}(d)$-synthesis molecule
provided that $L$ and $N$ satisfy \eqref{900}.
Thus, \eqref{902} is a direct application of Theorem \ref{89}\eqref{892}.

\eqref{901} Let $\mathscr N>\max\{L,N\}$,
and let $\theta^{(i)}$ be Daubechies wavelets of class $C^{\mathscr N}$,
which exist by Lemma \ref{wavelet basis 2}.
For any $L$ and $N$ in \eqref{900}, it also follows that
$\mathscr N$ satisfies \eqref{NDau}, and hence Theorem \ref{wavelet 2} guarantees that
\begin{equation*}
\vec f=\sum_{i=1}^{2^n-1}\sum_{Q\in\mathscr Q}\vec t_Q^{(i)}
\theta_Q^{(i)},\
\sum_{i=1}^{2^n-1}\left\|\vec t^{(i)}\right\|_{\dot a^{s,\tau}_{p,q}(W)}
\lesssim\left\|\vec f\right\|_{\dot A^{s,\tau}_{p,q}(W)},\text{ and }
\vec t_Q^{(i)}:=\left\langle\vec f,\theta_Q^{(i)}\right\rangle
\end{equation*}
with convergence in $(\mathcal S_\infty')^m$.

It remains to rearrange this sum into a sum over $Q\in\mathscr Q$ only.
This is achieved as follows. For each $Q\in\mathscr Q$,
let $Q_i$, $i\in\{1,\ldots,2^n\}$,
be an enumeration of the dyadic child-cubes of $Q$.
By Definition \ref{def Dau} and Remark \ref{Dau can},
for some constants $c$ and $r$, the functions
\begin{equation*}
a_{Q_i}:=\begin{cases}
c\theta_Q^{(i)} & \text{if }i\in\{1,\ldots,2^n-1\}, \\
0 & \text{if }i=2^n,
\end{cases}
\end{equation*}
are $(r,\mathscr N,\mathscr N)$-atoms, and hence also $(r,L,N)$-atoms, over the respective cubes indicated by their subscripts. We also define the coefficient
\begin{equation*}
\vec t_{Q_i}:=\begin{cases} c^{-1} \vec t_Q^{(i)} & \text{if }i\in\{1,\ldots,2^n-1\}, \\ \vec{\mathbf{0}} & \text{if }i=2^n.\end{cases}
\end{equation*}
Since every $R\in\mathscr Q$ is of the form $Q_i$ for some $Q\in\mathscr Q$ and $i\in\{1,\ldots,2^n-1\}$, we have
\begin{equation*}
\sum_{R\in\mathscr Q}\vec t_R a_R
=\sum_{i=1}^{2^n-1} \sum_{Q\in\mathscr Q}\vec t_{Q_i} a_{Q_i}
=\sum_{i=1}^{2^n-1} \sum_{Q\in\mathscr Q}\vec t_{Q}^{(i)} \theta_{Q}^{(i)}=\vec f,
\end{equation*}
which is the desired series representation of $\vec f$.
While $\vec t$ has the same {\em set} of coefficients as
the union of $\{\vec t^{(i)}\}_{i=1}^{2^n-1}$,
the indexing of these coefficients is shifted by one level.
However, it is easy to infer from the definition of the norm
$\|\cdot\|_{\dot a^{s,\tau}_{p,q}(W)}$ that such
a shift changes the norm at most by a positive constant $C$ that is independent of $\vec t$.
This finishes the proof of Theorem \ref{90}.
\end{proof}

\section{Trace Theorems}
\label{trace theorems}

In this section, we discuss the boundedness of both trace operators
and extension operators for $\dot A^{s,\tau}_{p,q}(W)$.
Let us make some conventions on symbols first, which are only used in this section.
Let $n\geq2$ and $\mathscr{Q}(\mathbb{R}^n)$
[resp. $\mathscr{Q}(\mathbb{R}^{n-1})$]
be the set of all dyadic cubes on $\mathbb{R}^n$
(resp. $\mathbb{R}^{n-1}$).
Since the underlying space under consideration may vary,
we use $\dot A^{s,\tau}_{p,q}(W,\mathbb{R}^n)$,
$\dot a^{s,\tau}_{p,q}(W,\mathbb{R}^n)$,
and $\dot a^{s,\tau}_{p,q}(\mathbb{A},\mathbb{R}^n)$
instead of $\dot A^{s,\tau}_{p,q}(W)$,
$\dot a^{s,\tau}_{p,q}(W)$, and $\dot a^{s,\tau}_{p,q}(\mathbb{A})$ in this section.
We denote a point $x\in\mathbb{R}^n$ by $x=(x',x_n)$,
where $x'\in\mathbb{R}^{n-1}$ and $x_n\in\mathbb{R}$.
We also denote a $\lambda\in\{0,1\}^n$ by $\lambda=(\lambda',\lambda_n)$,
where $\lambda'\in\{0,1\}^{n-1}$ and $\lambda_n\in\{0,1\}$.
To establish the trace theorem for $\dot A^{s,\tau}_{p,q}(W,\mathbb{R}^n)$,
we need more details about the Daubechies wavelet (see, for instance, \cite{d88}).

\begin{lemma}\label{wavelet basis 3}
For any $\mathscr{N}\in\mathbb{N}$,
there exist two real-valued $C^{\mathscr{N}}(\mathbb{R})$ functions
$\varphi$ and $\psi$ with bounded support such that,
for any $ n\in\mathbb{N}$,
$\{\theta^{(\lambda)}_Q:\ Q\in\mathscr{Q},\
\lambda\in\Lambda_n:=\{0,1\}^n\setminus\{\mathbf 0\}\}$
form an orthonormal basis of $L^2(\mathbb R^n)$,
where, for any $\lambda:=(\lambda_1,\ldots,\lambda_n)\in\Lambda_n$
and $x:=(x_1,\ldots,x_n)\in\mathbb{R}^n$,
$\theta^{(\lambda)}(x):=\prod_{i=1}^n\phi^{(\lambda_i)}(x_i)$
with $\phi^{(0)}:=\varphi$ and $\phi^{(1)}:=\psi$.
\end{lemma}

\begin{remark}\label{k0}
In Lemma \ref{wavelet basis 3}, there exists a sequence
$\{h_k\}_{k\in\mathbb{Z}}\subset\mathbb{C}$ such that,
for any $t\in\mathbb{R}$,
\begin{equation}\label{266}
\varphi(t)=\sqrt{2}\sum_{k\in\mathbb{Z}}h_k\varphi(2t-k);
\end{equation}
see \cite[(4.1) and (4.2)]{d88} or \cite[Theorem 6.3.6]{d92}.
By \eqref{266} and the continuity of $\varphi$,
we find that there exists $k_0\in\mathbb{Z}$ such that $\varphi(-k_0)\neq0$.
\end{remark}

Next, we recall averaging matrix-weighted Besov-type
and Triebel--Lizorkin-type sequence spaces (see \cite[Definition 3.26]{bhyy1}).

\begin{definition}
Let $s\in\mathbb{R}$, $\tau\in[0,\infty)$, $p\in(0,\infty)$, $q\in(0,\infty]$,
$W\in A_p(\mathbb R^n,\mathbb C^m)$,
and $\mathbb{A}:=\{A_Q\}_{Q\in\mathscr{Q}(\mathbb{R}^n)}$ be a sequence of
reducing operators of order $p$ for $W$.
The \emph{homogeneous averaging matrix-weighted Besov-type sequence space}
$\dot{b}^{s,\tau}_{p,q}(\mathbb{A},\mathbb{R}^n)$
and the \emph{homogeneous averaging matrix-weighted Triebel--Lizorkin-type sequence space}
$\dot{f}^{s,\tau}_{p,q}(\mathbb{A},\mathbb{R}^n)$
are defined to be the sets of all sequences
$\vec t:=\{\vec t_Q\}_{Q\in\mathscr{Q}(\mathbb{R}^n)}\subset\mathbb{C}^m$ such that
$$
\left\|\vec{t}\right\|_{\dot{a}^{s,\tau}_{p,q}(\mathbb{A},\mathbb{R}^n)}
:=\left\|\left\{2^{js}\sum_{Q\in\mathscr{Q}_j(\mathbb{R}^n)}
\left|A_Q\vec t_Q\right|\widetilde{\mathbf{1}}_Q\right\}_{j\in\mathbb Z}
\right\|_{L\dot A_{p,q}^\tau(\mathbb{R}^n)}<\infty,
$$
where $\|\cdot\|_{L\dot A_{p,q}^\tau(\mathbb{R}^n)}$ is the same as in \eqref{LApq}.
\end{definition}

The following conclusion is just \cite[Theorem 3.27]{bhyy1},
which proves $\dot{a}^{s,\tau}_{p,q}(W,\mathbb{R}^n)
=\dot{a}^{s,\tau}_{p,q}(\mathbb{A},\mathbb{R}^n)$.

\begin{lemma}\label{37}
Let $s\in\mathbb{R}$, $\tau\in[0,\infty)$,
$p\in(0,\infty)$, $q\in(0,\infty]$, $W\in A_p(\mathbb R^n,\mathbb C^m)$,
and $\mathbb{A}:=\{A_Q\}_{Q\in\mathscr{Q}(\mathbb{R}^n)}$ be a sequence of
reducing operators of order $p$ for $W$.
Then, for any $\vec t:=\{\vec t_Q\}_{Q\in\mathscr{Q}(\mathbb{R}^n)}\subset\mathbb{C}^m$,
$$
\left\|\vec{t}\right\|_{\dot a^{s,\tau}_{p,q}(W,\mathbb R^n)}
\sim\left\|\vec{t}\right\|_{\dot{a}^{s,\tau}_{p,q}(\mathbb{A},\mathbb R^n)},
$$
where the positive equivalence constants are independent of $\vec t$.
\end{lemma}

\subsection{Trace Theorems for $\dot B^{s,\tau}_{p,q}(W,\mathbb{R}^n) $}

Throughout the subsection, unless otherwise mentioned,
let $\mathscr{N}\in\mathbb{N}$ satisfy \eqref{NDau}
for $\widetilde J$ and $\widetilde s$ corresponding to both spaces $\dot B^{s,\tau}_{p,q}(W,\mathbb{R}^n)$
and $\dot B^{s-\frac{1}{p},\frac{n}{n-1}\tau}_{p,q}(V,\mathbb{R}^{n-1})$,
and let both $\{\theta^{(\lambda)}\}_{\lambda\in\Lambda_n}\subset C^{\mathscr{N}}(\mathbb{R}^n)$
and $\{\theta^{(\lambda')}\}_{\lambda'\in\Lambda_{n-1}}
\subset C^{\mathscr{N}}(\mathbb{R}^{n-1})$
be the same as in Lemma \ref{wavelet basis 3}.
For any $I\in\mathscr{Q}(\mathbb{R}^{n-1})$ and $k\in\mathbb{Z}$, let
$$
Q(I,k):=I\times[\ell(I)k,\ell(I)(k+1)).
$$
By their construction, it is easy to show that,
for any $Q\in\mathscr{Q}(\mathbb{R}^n)$,
there exist $I\in\mathscr{Q}(\mathbb{R}^{n-1})$ and $k\in\mathbb{Z}$
such that $Q=Q(I,k)$, and we let $I(Q):=I$.
Let $W\in A_p(\mathbb{R}^n,\mathbb{C}^m)$ and
$\mathbb{A}(W):=\{A_{Q,W}\}_{Q\in\mathscr{Q}(\mathbb{R}^n)}$,
where, for any $Q\in\mathscr{Q}(\mathbb{R}^n)$,
$A_{Q, W}$ denotes the reducing operator of order $p$ for $W$.
Let $V\in A_p(\mathbb{R}^{n-1},\mathbb{C}^m)$ and
$\mathbb{A}(V):=\{A_{I,V}\}_{I\in\mathscr{Q}(\mathbb{R}^{n-1})}$,
where, for any $I\in\mathscr{Q}(\mathbb{R}^{n-1})$,
$A_{I,V}$ denotes the reducing operator of order $p$ for $V$.

\begin{lemma}\label{126}
Let $k\in\mathbb{Z}$.
Then, for any $R\in\mathscr{Q}(\mathbb{R}^{n-1})$,
there exists $P_R\in\mathscr{Q}(\mathbb{R}^n)$
such that
$$
\ell(P_R)=\begin{cases}
2^{\lceil\log_2(k+1)\rceil}\ell(R)&\text{if }k\geq0,\\
2^{\lceil\log_2(-k)\rceil}\ell(R)&\text{if }k\leq-1,
\end{cases}
$$
where $\lceil\cdot\rceil$ is the same as in \eqref{ceil},
and that, for any $I\in\mathscr{Q}(\mathbb{R}^{n-1})$ with $I\subset R$,
$Q(I,k)\subset P_R$.
\end{lemma}

\begin{proof}
Let $k\in\mathbb{Z}$ and $R\in\mathscr{Q}(\mathbb{R}^{n-1})$.
We first consider the case when $k\geq0$.
In this case, choose $\widetilde{R}\in\mathscr{Q}(\mathbb{R}^{n-1})$
satisfying $R\subset\widetilde{R}$
and $\ell(\widetilde{R})=2^{\lceil\log_2(k+1)\rceil}\ell(R)$,
and let $P_R :=Q(\widetilde{R},0)$.
Then $\ell(P_R)=2^{\lceil\log_2(k+1)\rceil}\ell(R)$ and,
for any $I\in\mathscr{Q}(\mathbb{R}^{n-1})$ with $I\subset R$,
\begin{align*}
Q(I,k)=I\times[\ell(I)k,\ell(I)(k+1))
\subset R\times[0,\ell(R)(k+1))
\subset\widetilde{R}\times\left[0,\ell\left(\widetilde{R}\right)\right)
=P_R.
\end{align*}
This finishes the proof of the present lemma in this case.

Now, we consider the case when $k\leq-1$.
In this case, choose $\widetilde{R}\in\mathscr{Q}(\mathbb{R}^{n-1})$
satisfying $R\subset\widetilde{R}$
and $\ell(\widetilde{R})=2^{\lceil\log_2(-k)\rceil}\ell(R)$,
and let $P_R:=Q(\widetilde{R},-1)$.
Then $\ell(P_R)=2^{\lceil\log_2(-k)\rceil}\ell(R)$ and,
for any $I\in\mathscr{Q}(\mathbb{R}^{n-1})$ with $I\subset R$,
\begin{align*}
Q(I,k)
=I\times[\ell(I)k,\ell(I)(k+1))
\subset R\times[\ell(I)k,0)
\subset\widetilde{R}\times\left[-\ell\left(\widetilde{R}\right),0\right)
=P_R.
\end{align*}
This finishes the proof of Lemma \ref{126}.
\end{proof}

Next, we define the trace and the extension operators.
For any $\lambda:=(\lambda',\lambda_n)\in\Lambda_n$,
$Q:=Q(I,k)\in\mathscr{Q}(\mathbb{R}^n)$
with $I\in\mathscr{Q}(\mathbb{R}^{n-1})$ and $k\in\mathbb{Z}$,
and $x'\in\mathbb{R}^{n-1}$, let
$$
\left[\operatorname{Tr}\theta^{(\lambda)}_Q\right](x')
:=\theta^{(\lambda)}_Q(x',0)
=[\ell(Q)]^{-\frac12}\theta^{(\lambda')}_{I(Q)}(x')\phi^{(\lambda_n)}(-k).
$$

For any functions $g$ and $h$, respectively,
on $\mathbb{R}^{n-1}$ and $\mathbb{R}$, let
$(g\otimes h)(x):=g(x')h(x_n)$
for any $x:=(x',x_n)\in\mathbb{R}^n$.
For any $\lambda'\in\Lambda_{n-1}$,
$I\in\mathscr{Q}(\mathbb{R}^{n-1})$,
and $x:=(x',x_n)\in\mathbb{R}^n$, let
\begin{align}\label{identity pre}
\left[\operatorname{Ext}\theta^{(\lambda')}_I\right](x)
:=&\,\frac{[\ell(I)]^{\frac12}}{\varphi(-k_0)}
\left[\theta^{(\lambda')}\otimes\varphi\right]_{Q(I,k_0)}(x)
=\frac{[\ell(I)]^{\frac12}}{\varphi(-k_0)}
\theta^{((\lambda',0))}_{Q(I,k_0)}(x)\\
=&\,\frac{1}{\varphi(-k_0)}\theta^{(\lambda')}_I(x')\varphi\left(\frac{x_n}{\ell(I)}-k_0 \right),\notag
\end{align}
where $\varphi$ and $k_0$ are the same as, respectively, in Lemma \ref{wavelet basis 3} and Remark \ref{k0}.
Then, for any $x'\in\mathbb{R}^{n-1}$,
\begin{equation}\label{identity}
\left(\operatorname{Tr}\circ\operatorname{Ext}\right)\left[\theta^{(\lambda')}_I\right](x')
=\frac{[\ell(I)]^{\frac12}}{\varphi(-k_0)}
\left[\operatorname{Tr}\theta^{((\lambda',0))}_{Q(I,k_0)}\right](x')
=\theta^{(\lambda')}_{I}(x').
\end{equation}

\begin{theorem}\label{trace B}
Let $\tau\in[0,\infty)$, $p\in(0,\infty)$, $q\in(0,\infty]$, and
$s\in(\frac{1}{p}+E,\infty)$, where
\begin{equation}\label{219}
E:=\left\{\begin{aligned}
&\frac{n-1}{p}-n\tau&&\text{if}\,\,\frac{n}{n-1}\tau>\frac{1}{p},\\
&0&&\text{if}\,\,\frac{n}{n-1}\tau=\frac{1}{p}\text{ and }q=\infty,\\
&(n-1)\left(\frac{1}{p}-1\right)_+&&\text{otherwise}.
\end{aligned}\right.
\end{equation}
Let $W\in A_p(\mathbb{R}^n,\mathbb{C}^m)$ and $V\in A_p(\mathbb{R}^{n-1},\mathbb{C}^m)$.
Then the trace operator $\operatorname{Tr}$ can be extended to a
continuous linear operator
$\operatorname{Tr}:\ \dot B^{s,\tau}_{p,q}(W,\mathbb{R}^n)
\to\dot B^{s-\frac{1}{p},\frac{n}{n-1}\tau}_{p,q}(V,\mathbb{R}^{n-1})$
if and only if there exists a positive constant $C$ such that,
for any $I\in\mathscr{Q}(\mathbb{R}^{n-1})$
and $\vec z\in\mathbb{C}^m$,
\begin{equation}\label{116}
\fint_I\left|V^{\frac{1}{p}}(x')\vec z\right|^p\,dx'
\leq C\fint_{Q(I,0)}\left|W^{\frac{1}{p}}(x)\vec z\right|^p\,dx.
\end{equation}
\end{theorem}

\begin{remark}\label{rem trace B}
By Lemma \ref{mol no canc}, we find that the requirement that $s-\frac1p>E$, with $E$ in \eqref{219}, is precisely the condition to guarantee that synthesis molecules of the target space $\dot B^{s-\frac1p,\frac{n}{n-1}\tau}_{p,q}(\mathbb R^{n-1},V)$ are not required to have any cancellation conditions. Note that this condition depends only on the explicit parameters of this space, not on the $A_p$-dimension of the weight $V$.
\end{remark}

To prove Theorem \ref{trace B}, we need the following lemma.

\begin{lemma}\label{l-new}
Let $p\in(0,\infty)$, $V\in A_p(\mathbb{R}^{n-1},\mathbb{C}^m)$,
and $W\in A_p(\mathbb{R}^n,\mathbb{C}^m)$ have $A_p$-dimensions $(d,\widetilde d,\Delta)$.
If \eqref{116} holds, then
there exists a positive constant $C$ such that,
for any $I\in\mathscr{Q}(\mathbb{R}^{n-1})$, $k\in\mathbb{Z}$, and $\vec{z}
\in \mathbb{C}^m$,
$|A_{I,V}\vec z|\le C(1+|k|)^{\Delta}|A_{Q(I,k),W}\vec z|.$
\end{lemma}

\begin{proof}
If \eqref{116} holds, then, by
\eqref{equ_reduce} and Lemma \ref{22 precise}, we conclude that,
for any  $I\in\mathscr{Q}(\mathbb{R}^{n-1})$, $k\in\mathbb{Z}$, and  $\vec z\in\mathbb{C}^m$,
\begin{align*}
\left|A_{I,V}\vec z\right|
\lesssim\left|A_{Q(I,0),W}\vec z\right|
\leq\left\|A_{Q(I,0),W}\left[A_{Q(I,k),W}\right]^{-1}\right\|
\left|A_{Q(I,k),W}\vec z\right|
\lesssim(1+|k|)^{\Delta}\left|A_{Q(I,k),W}\vec z\right|.
\end{align*}
This finishes the proof of Lemma \ref{l-new}.
\end{proof}

Concerning the necessity of Theorem \ref{trace B},
we show a slightly more general statement in a separate lemma that we can apply later on.

\begin{lemma}\label{trace nec}
Let $\tau\in[0,\infty)$, $p\in(0,\infty)$, $q,r\in(0,\infty]$, and $s\in\mathbb R$.
Let $W\in A_p(\mathbb{R}^n,\mathbb{C}^m)$ and $V\in A_p(\mathbb{R}^{n-1},\mathbb{C}^m)$.
Let $\mathscr{N}\in\mathbb{N}$ satisfy \eqref{NDau}
for $\widetilde J$ and $\widetilde s$ corresponding to both spaces $\dot A^{s,\tau}_{p,q}(W,\mathbb{R}^n)$
and $\dot B^{s-\frac{1}{p},\frac{n}{n-1}\tau}_{p,r}(V,\mathbb{R}^{n-1})$,
and let both $\{\theta^{(\lambda)}\}_{\lambda\in\Lambda_n}\subset C^{\mathscr{N}}(\mathbb{R}^n)$
and $\{\theta^{(\lambda')}\}_{\lambda'\in\Lambda_{n-1}}
\subset C^{\mathscr{N}}(\mathbb{R}^{n-1})$
be the same as in Lemma \ref{wavelet basis 3}.
Suppose that the trace operator $\operatorname{Tr}:\ \dot A^{s,\tau}_{p,q}(W,\mathbb{R}^n)
\to\dot B^{s-\frac{1}{p},\frac{n}{n-1}\tau}_{p,r}(V,\mathbb{R}^{n-1})$ is continuous.
Then there exists a positive constant $C$ such that \eqref{116} holds
for any dyadic cube $I\subset\mathbb{R}^{n-1}$ and any $\vec z\in\mathbb{C}^m$.
\end{lemma}

\begin{proof}
Fix a dyadic cube $I_0\subset\mathbb{R}^{n-1}$
and $\vec z\in\mathbb{C}^m$.
Let $\vec t:=\{\vec t_I\}_{I\in\mathscr{Q}(\mathbb{R}^{n-1})}$,
where, for any $I\in\mathscr{Q}(\mathbb{R}^{n-1})$,
$$
\vec t_I:=\begin{cases}
[\ell(I_0)]^{n\tau+(s-\frac{1}{p})-(n-1)(\frac{1}{p}-\frac{1}{2})}\vec z
&\text{if }I=I_0,\\
\vec{\mathbf{0}}&\text{otherwise},
\end{cases}
$$
and let $\vec g:=\vec t_{I_0}\theta^{(\lambda')}_{I_0}$ for some $\lambda'\in\Lambda_{n-1}$.
Then, by Theorem \ref{wavelet} and Lemma \ref{37}, we find that
\begin{align}\label{125}
\left\|\vec g\right\|_{\dot A^{s-\frac{1}{p},\frac{n}{n-1}\tau}_{p,q}(V,\mathbb{R}^{n-1})}
&\sim\left\|\vec t\right\|_{\dot a^{s-\frac{1}{p},\frac{n}{n-1}\tau}_{p,q}(V,\mathbb{R}^{n-1})}
\sim\left\|\vec t\right\|_{\dot a^{s-\frac{1}{p},\frac{n}{n-1}\tau}_{p,q}(\mathbb{A}(V), \mathbb{R}^{n-1})}\\
&=[\ell(I_0)]^{-n\tau-(s-\frac{1}{p})+(n-1)(\frac{1}{p}-\frac{1}{2})}
\left|A_{I_0,V}\vec t_{I_0}\right|
=\left|A_{I_0,V}\vec z\right|.\notag
\end{align}
Let $\vec u:=\{\vec u_Q\}_{Q\in\mathscr{Q}(\mathbb{R}^n)}$,
where, for any $Q\in\mathscr{Q}(\mathbb{R}^n)$,
$$
\vec u_Q:=\begin{cases}
[\ell(I_0)]^{\frac12}\vec t_{I_0}&\text{if }Q=Q(I_0,k_0),\\
\vec{\mathbf{0}}&\text{otherwise}.
\end{cases}
$$
For any $x:=(x',x_n)\in\mathbb{R}^n$, let
$$
\vec f
:=\vec t_{I_0}\operatorname{Ext}\left(\theta^{(\lambda')}_{I_0}\right)
=\frac{1}{\varphi(-k_0)}\vec u_{Q(I_0,k_0)}
\left[\theta^{(\lambda')}\otimes \varphi \right]_{Q(I_0,k_0)},
$$
where $\varphi$ and $k_0$ are the same as, respectively, in Lemma \ref{wavelet basis 3} and Remark \ref{k0}.
Then, using \eqref{identity}, we conclude that, for any $x'\in\mathbb{R}^{n-1}$,
\begin{equation}\label{235}
\left(\operatorname{Tr}\vec f\right)(x')
=\vec t_{I_0}\theta^{(\lambda')}_{I_0}(x')
=\vec g(x').
\end{equation}
Let $W\in A_p(\mathbb{R}^n,\mathbb{C}^m)$
have $A_p$-dimensions $(d_W,\widetilde d_W,\Delta_W)$.
Since $\theta^{(\lambda')}\otimes\varphi$ has bounded support
and $\mathscr{N}\in\mathbb{N}$ satisfies \eqref{NDau}
for $\dot A^{s,\tau}_{p,q}(W,\mathbb{R}^n)$,
we deduce that $[\theta^{(\lambda')}\otimes\varphi]_{Q(I_0,k_0)}$ is a harmless constant multiple of
a $\dot A^{s,\tau}_{p,q}(d_W)$-synthesis molecule on $Q(I_0,k_0)$.
By this, Theorem \ref{89}(ii), and Lemma \ref{37}, we obtain
\begin{align*}
\left\|\vec f\right\|_{\dot A^{s,\tau}_{p,q}(W,\mathbb{R}^n)}
&\lesssim\left\|\vec u\right\|_{\dot a^{s,\tau}_{p,q}(W,\mathbb{R}^n)}
\sim\left\|\vec u\right\|_{\dot a^{s,\tau}_{p,q}(\mathbb{A}(W),\mathbb{R}^n)}\\
&=[\ell(Q(I,k_0))]^{-n\tau-s+n(\frac{1}{p}-\frac{1}{2})}
\left|A_{Q(I,k_0),W}\vec u_{Q(I,k_0)}\right|
=\left|A_{Q(I,k_0),W}\vec z\right|,
\end{align*}
which, together with \eqref{125}, \eqref{235}, and the assumption that
$\operatorname{Tr}$ is bounded operator, further implies that
\begin{align*}
\left|A_{I_0,V}\vec z\right|
\sim\left\|\vec g\right\|_{\dot A^{s-\frac{1}{p},\frac{n}{n-1}\tau}_{p,q}(V,\mathbb{R}^{n-1})}
=\left\|\operatorname{Tr}\vec f\right\|_{\dot A^{s-\frac{1}{p},\frac{n}{n-1}\tau}_{p,q}(V,\mathbb{R}^{n-1})}
\lesssim\left\|\vec f\right\|_{\dot A^{s,\tau}_{p,q}(W,\mathbb{R}^n)}
\lesssim\left|A_{Q(I_0,k_0),W}\vec z\right|.
\end{align*}
From Lemma \ref{22 precise}, we infer that, for any $\vec z\in\mathbb{C}^m$,
\begin{align*}
\left|A_{Q(I_0,k_0),W}\vec z\right|
\leq\left\|A_{Q(I_0,k_0),W}\left[A_{Q(I_0,0),W}\right]^{-1}\right\|
\left|A_{Q(I_0,0),W}\vec z\right|
\lesssim(1+|k_0|)^{\Delta_W}\left|A_{Q(I_0,0),W}\vec z\right|
\end{align*}
and hence
$|A_{I_0,V}\vec z|
\lesssim|A_{Q(I_0,0),W}\vec z|.$
By this and \eqref{reduce}, we find that \eqref{116} holds for $I=I_0$.
Since the choice of $I_0$ was arbitrary, this finishes the proof of Lemma \ref{trace nec}.
\end{proof}

Now, we prove Theorem \ref{trace B}.

\begin{proof}[Proof of Theorem \ref{trace B}]
We first show the necessity.
Suppose that
$$\operatorname{Tr}:\ \dot B^{s,\tau}_{p,q}(W,\mathbb{R}^n)
\to\dot B^{s-\frac{1}{p},\frac{n}{n-1}\tau}_{p,q}(V,\mathbb{R}^{n-1})$$
is continuous.
By the standing assumptions stated in the beginning of the section,
the assumptions of Lemma \ref{trace nec} are satisfied with
$\dot A^{s,\tau}_{p,q}(W,\mathbb{R}^n)=\dot B^{s,\tau}_{p,q}(W,\mathbb{R}^n)$
and $r=q$. Thus, the conclusion of that lemma gives the necessity of \eqref{116}.

Now, we prove the sufficiency.
Suppose that \eqref{116} holds.
Using Theorem \ref{wavelet 2}, we obtain,
for any $\vec f\in\dot B^{s,\tau}_{p,q}(W,\mathbb{R}^n)$,
$\vec f=\sum_{\lambda\in\Lambda_n}\sum_{Q\in\mathscr{Q}(\mathbb{R}^n)}
\langle\vec f,\theta^{(\lambda)}_Q\rangle\theta^{(\lambda)}_Q$
in $[\mathcal{S}_\infty'(\mathbb{R}^n)]^m$, and hence it is natural to define
\begin{equation}\label{traceop}
\operatorname{Tr}\vec f
:=\sum_{\lambda\in\Lambda_n}\sum_{Q\in\mathscr{Q}(\mathbb{R}^n)}
\left\langle\vec f,\theta^{(\lambda)}_Q\right\rangle\operatorname{Tr}\theta^{(\lambda)}_Q
\end{equation}
in $[\mathcal{S}_\infty'(\mathbb{R}^{n-1})]^m$.
Next, we show that $\operatorname{Tr}\vec f$ is well defined and
$$
\operatorname{Tr}:\ \dot B^{s,\tau}_{p,q}(W,\mathbb{R}^n)
\to\dot B^{s-\frac{1}{p},\frac{n}{n-1}\tau}_{p,q}(V,\mathbb{R}^{n-1})
$$
is a continuous linear operator.
For any $\lambda\in\Lambda_n$,
let $\vec t^{(\lambda)}:=\{\vec t^{(\lambda)}_Q\}_{Q\in\mathscr{Q}(\mathbb{R}^n)}$,
where, for any $Q\in\mathscr{Q}(\mathbb{R}^n)$,
\begin{equation}\label{dequen}
\vec t^{(\lambda)}_Q:=[\ell(Q)]^{-\frac12}\left\langle\vec f,\theta^{(\lambda)}_Q\right\rangle.
\end{equation}
Since $\theta^{(\lambda)}$ has bounded support,
we deduce that there exists $N\in\mathbb{N}$ such that
$\operatorname{supp}\theta^{(\lambda)}\subset B(\mathbf{0},N)$.
Thus, by the definition of $\theta^{(\lambda)}_{Q(I,k)}$,
we obtain, for any $\lambda\in\Lambda_n $, $I\in\mathscr{Q}(\mathbb{R}^{n-1})$,
$k\in\mathbb{Z}$ with $|k|>N$, and $x'\in\mathbb{R}^{n-1}$,
$$
\theta^{(\lambda)}_{Q(I,k)}(x',0)
=[\ell(I)]^{-\frac{n}{2}}\theta^{(\lambda)}\left(\frac{x'-x_I}{\ell(I)},-k\right)=0.
$$
From this, we infer that, for any $x'\in\mathbb{R}^{n-1}$,
\begin{align*}
\sum_{\lambda\in\Lambda_n}
\sum_{Q\in\mathscr{Q}(\mathbb{R}^n)}\left\langle\vec f,\theta^{(\lambda)}_Q
\right\rangle\theta^{(\lambda)}_Q(x',0)
&=\sum_{\lambda\in\Lambda_n}\sum_{Q\in\mathscr{Q}(\mathbb{R}^n)}
\vec t^{(\lambda)}_Q[\ell(Q)]^{\frac12}\theta^{(\lambda)}_Q(x',0)\\
&=\sum_{\lambda\in\Lambda_n}\sum_{I\in\mathscr{Q}(\mathbb{R}^{n-1})}\sum_{k=-N}^N
\vec t^{(\lambda)}_{Q(I,k)}[\ell(I)]^{\frac12}\theta^{(\lambda)}_{Q(I,k)}(x',0).
\end{align*}
Using this and the definition of $\operatorname{Tr}\vec f$, we obtain
\begin{equation}\label{223}
\operatorname{Tr}\vec f
=\sum_{\lambda\in\Lambda_n}\sum_{k=-N}^N\sum_{I\in\mathscr{Q}(\mathbb{R}^{n-1})}
\vec t^{(\lambda)}_{Q(I,k)}[\ell(I)]^{\frac12}\theta^{(\lambda)}_{Q(I,k)}(\cdot,0)
\end{equation}
in $[\mathcal{S}_\infty'(\mathbb{R}^{n-1})]^m$.
Applying Lemma \ref{l-new},  we conclude that,
for any $\lambda\in\Lambda_n$, $I\in\mathscr{Q}(\mathbb{R}^{n-1})$,
and $k\in\mathbb{Z}$ with $|k|\leq N$,
\begin{align}\label{inequa-n}
\left|A_{I,V}\vec t^{(\lambda)}_{Q(I,k)}\right|
\lesssim(1+|k|)^{\Delta_W}\left|A_{Q(I,k),W}\vec t^{(\lambda)}_{Q(I,k)}\right|
\leq(1+N)^{\Delta_W}\left|A_{Q(I,k),W}\vec t^{(\lambda)}_{Q(I,k)}\right|.
\end{align}
This, combined with Lemma \ref{37}, further implies that,
for any $\lambda\in\Lambda_n $ and $k\in\mathbb{Z}$ with $ |k|\leq N $,
\begin{align}\label{218}
&\left\|\left\{\vec t^{(\lambda)}_{Q(I,k)}\right\}_{I\in\mathscr{Q}(\mathbb{R}^{n-1})}
\right\|_{\dot b^{s-\frac{1}{p},\frac{n}{n-1}\tau}_{p,q}(V,\mathbb{R}^{n-1})}\\
&\quad\sim\left\|\left\{\vec t^{(\lambda)}_{Q(I,k)}\right\}_{I\in\mathscr{Q}(\mathbb{R}^{n-1})}
\right\|_{\dot b^{s-\frac{1}{p},\frac{n}{n-1}\tau}_{p,q}(\mathbb{A}(V), \mathbb{R}^{n-1})}\notag\\
&\quad=\left\|\left\{\left|A_{I,V}\vec t^{(\lambda)}_{Q(I,k)}\right|
\right\}_{I\in\mathscr{Q}(\mathbb{R}^{n-1})}
\right\|_{\dot b^{s-\frac{1}{p},\frac{n}{n-1}\tau}_{p,q}(\mathbb{R}^{n-1})}\notag\\
&\quad\lesssim\left\|\left\{\left|A_{Q(I,k),W}\vec t^{(\lambda)}_{Q(I,k)}\right|
\right\}_{I\in\mathscr{Q}(\mathbb{R}^{n-1})}
\right\|_{\dot b^{s-\frac{1}{p},\frac{n}{n-1}\tau}_{p,q}(\mathbb{R}^{n-1})}\notag\\
&\quad\leq\left\|\left\{\left\langle\vec f,
\theta^{(\lambda)}_Q \right\rangle\right\}_{Q\in\mathscr{Q}(\mathbb{R}^n)}\right\|_{\dot b^{s,\tau}_{p,q}(\mathbb{A}(W), \mathbb{R}^n)}
\sim\left\|\left\{\left\langle\vec f,
\theta^{(\lambda)}_Q \right\rangle\right\}_{Q\in\mathscr{Q}(\mathbb{R}^n)}
\right\|_{\dot b^{s,\tau}_{p,q}(W,\mathbb{R}^n)}.\notag
\end{align}
Using Theorem \ref{wavelet 2}, we find that
$\{\langle\vec f,\theta^{(\lambda)}_Q\rangle\}_{Q\in\mathscr{Q}(\mathbb{R}^n)}
\in\dot b^{s,\tau}_{p,q}(W,\mathbb{R}^n)$ and hence
\begin{equation}\label{217}
\left\{\vec t^{(\lambda)}_{Q(I,k)}\right\}_{I\in\mathscr{Q}(\mathbb{R}^{n-1})}
\in\dot b^{s-\frac{1}{p},\frac{n}{n-1}\tau}_{p,q}(V,\mathbb{R}^{n-1}).
\end{equation}
Let $V\in A_p(\mathbb{R}^{n-1},\mathbb{C}^m)$ have the $A_p$-dimension $d_V\in[0, n-1)$.
Notice that $s\in(\frac{1}{p}+E,\infty)$, where $E$ is the same as in \eqref{219}.
Hence, by Remark \ref{rem trace B}, the vanishing moment condition of synthesis molecules is void for spaces
$\dot B^{s-\frac{1}{p},\frac{n}{n-1}\tau}_{p,q}(V,\mathbb{R}^{n-1})$.
Together with the fact that $\theta^{(\lambda)}_{Q(I,k)}(\cdot,0)$ has bounded support, this further implies that the function
$[\ell(I)]^{\frac12}\theta^{(\lambda)}_{Q(I,k)}(\cdot,0)$ is a harmless constant multiple of
a $\dot B^{s-\frac{1}{p},\frac{n}{n-1}\tau}_{p,q}(d_V)$-synthesis molecule on $I$.
From this, \eqref{223}, \eqref{217}, Theorem \ref{89}(ii),
\eqref{218}, and Theorem \ref{wavelet 2},
it follows that $\operatorname{Tr}\vec f$ is well defined and
\begin{align*}
\left\|\operatorname{Tr}\vec f\right\|_{\dot B^{s-\frac{1}{p},\frac{n}{n-1}\tau}_{p,q}(V,\mathbb{R}^{n-1})}
&\lesssim\sum_{\lambda\in\Lambda_n}\sum_{k=-N}^N
\left\|\sum_{I\in\mathscr{Q}(\mathbb{R}^{n-1})}
\vec t^{(\lambda)}_{Q(I,k)} [\ell(I)]^{\frac12}\theta^{(\lambda)}_{Q(I,k)}(\cdot,0)
\right\|_{\dot B^{s-\frac{1}{p},\frac{n}{n-1}\tau}_{p,q}(V,\mathbb{R}^{n-1})}\\
&\lesssim\sum_{\lambda\in\Lambda_n}\sum_{k=-N}^N
\left\|\left\{\vec t^{(\lambda)}_{Q(I,k)}\right\}_{I\in\mathscr{Q}(\mathbb{R}^{n-1})}
\right\|_{\dot b^{s-\frac{1}{p},\frac{n}{n-1}\tau}_{p,q}(V,\mathbb{R}^{n-1})}\\
&\lesssim\sum_{\lambda\in\Lambda_n}
\left\|\left\{\left\langle\vec f,
\theta^{(\lambda)}_Q \right\rangle\right\}_{Q\in\mathscr{Q}(\mathbb{R}^n)}
\right\|_{\dot B^{s,\tau}_{p,q}(W,\mathbb{R}^n)}
\sim\left\|\vec f\right\|_{\dot B^{s,\tau}_{p,q}(W,\mathbb{R}^n)}.
\end{align*}
This finishes the proof of the sufficiency and hence Theorem \ref{trace B}.
\end{proof}

We now discuss an extension theorem for $\operatorname{Ext}$.

\begin{theorem}\label{extension B}
Let $s\in\mathbb{R}$, $\tau\in[0,\infty)$, $p\in(0,\infty)$, $q\in(0,\infty]$,
$W\in A_p(\mathbb{R}^n,\mathbb{C}^m)$, and $V\in A_p(\mathbb{R}^{n-1},\mathbb{C}^m)$.
Assume that there exists a positive constant $C$ such that,
for any $I\in\mathscr{Q}(\mathbb{R}^{n-1})$
and $\vec z\in\mathbb{C}^m$,
\begin{equation}\label{127}
\fint_{Q(I,0)}\left|W^{\frac{1}{p}}(x)\vec z\right|^p\,dx
\leq C\fint_I\left|V^{\frac{1}{p}}(x')\vec z\right|^p\,dx'.
\end{equation}
Then the extension operator $\operatorname{Ext}$ can be extended to a
continuous linear operator
\begin{equation}\label{225}
\operatorname{Ext}:\
\dot B^{s-\frac{1}{p},\frac{n}{n-1}\tau}_{p,q}(V,\mathbb{R}^{n-1})
\to\dot B^{s,\tau}_{p,q}(W,\mathbb{R}^n).
\end{equation}
Furthermore, if $s\in(\frac{1}{p}+E,\infty)$, where $E$ is the same as in \eqref{219},
and \eqref{116} also holds,
then $\operatorname{Tr}\circ\operatorname{Ext}$ is the identity on
$\dot B^{s-\frac{1}{p},\frac{n}{n-1}\tau}_{p,q}(V,\mathbb{R}^{n-1})$.
\end{theorem}

To prove this theorem, we need the following lemma.

\begin{lemma}\label{l-new2}
Let $p\in(0,\infty)$, $V\in A_p(\mathbb{R}^{n-1},\mathbb{C}^m)$,
and $W\in A_p(\mathbb{R}^{n},\mathbb{C}^m)$
which has $A_p$-dimensions $(d,\widetilde d,\Delta)$.
If \eqref{127} holds,
then there exists a positive constant $C$ such that,
for any $I\in\mathscr{Q}(\mathbb{R}^{n-1})$, $k\in\mathbb{Z}$, and $\vec{z}\in\mathbb{C}^m$,
$|A_{Q(I,k),W}\vec z|
\leq C(1+|k|)^{\Delta}|A_{I,V}\vec z|.$
\end{lemma}

\begin{proof}
If \eqref{127} holds, then, by
\eqref{equ_reduce} and Lemma \ref{22 precise}, we conclude that,
for any  $I\in\mathscr{Q}(\mathbb{R}^{n-1})$, $k\in\mathbb{Z}$, and  $\vec z\in\mathbb{C}^m$,
\begin{align*}
\left|A_{Q(I,k),W}\vec z\right|
\leq\left\|A_{Q(I,k),W}\left[ A_{Q(I,0),W}\right]^{-1}\right\|
\left|A_{Q(I,0),W}\vec z\right|
\lesssim(1+|k|)^{\Delta}\left|A_{I,V}\vec z\right|.
\end{align*}
This finishes the proof of Lemma \ref{l-new2}.
\end{proof}

Now, we show Theorem \ref{extension B}.

\begin{proof}[Proof of Theorem \ref{extension B}]
By \eqref{204}, we find that,
for any $\vec f\in\dot B^{s-\frac{1}{p},\frac{n}{n-1}\tau}_{p,q}(V,\mathbb{R}^{n-1})$,
$$
\vec f=\sum_{\lambda'\in\Lambda_{n-1}}\sum_{I\in\mathscr{Q}(\mathbb{R}^{n-1})}
\left\langle\vec f,\theta^{(\lambda')}_I\right\rangle\theta^{(\lambda')}_I
$$
in $[\mathcal{S}_\infty'(\mathbb{R}^{n-1})]^m$,
and hence it is natural to define
\begin{equation}\label{exten-op}
\operatorname{Ext}\vec f
:=\sum_{\lambda'\in\Lambda_{n-1}}\sum_{I\in\mathscr{Q}(\mathbb{R}^{n-1})}
\left\langle\vec f,\theta^{(\lambda')}_I\right\rangle\operatorname{Ext}\theta^{(\lambda')}_I
\end{equation}
in $[\mathcal{S}_\infty'(\mathbb{R}^n)]^m$.
Next, we prove that $\operatorname{Ext}\vec f$ is well defined and
$$
\operatorname{Ext}:\
\dot B^{s-\frac{1}{p},\frac{n}{n-1}\tau}_{p,q}(V,\mathbb{R}^{n-1})
\to\dot B^{s,\tau}_{p,q}(W,\mathbb{R}^n)
$$
is a continuous linear operator.
For any $\lambda'\in\Lambda_{n-1}$,
let $\vec t^{(\lambda')}:=\{\vec t^{(\lambda')}_Q\}_{Q\in\mathscr{Q}(\mathbb{R}^n)}$,
where, for any $Q\in\mathscr{Q}(\mathbb{R}^n)$,
\begin{equation}\label{eex}
\vec t^{(\lambda')}_Q:=\begin{cases}
[\ell(I)]^{\frac12}\left\langle\vec f,\theta^{(\lambda')}_I\right\rangle
&\text{if }Q=Q(I,k_0)\text{ for some }I\in\mathscr{Q}(\mathbb{R}^{n-1}),\\
\vec{\mathbf{0}}&\text{otherwise}
\end{cases}
\end{equation}
with $k_0$ the same as in Remark \ref{k0}.
Then, using the definitions of $\vec t^{(\lambda')}_Q$ and $\operatorname{Ext}$,
we obtain, for any $\lambda'\in\Lambda_{n-1}$ ,
$I\in\mathscr{Q}(\mathbb{R}^{n-1})$, and $x\in\mathbb{R}^n$,
$$
\left\langle\vec f,\theta^{(\lambda')}_I\right\rangle
\left[\operatorname{Ext}\theta^{(\lambda')}_I\right](x)
=\frac{1}{\varphi(-k_0)}\vec t^{(\lambda')}_{Q(I,k_0)}
\left[\theta^{(\lambda')}\otimes\varphi\right]_{Q(I,k_0)}(x),
$$
where $\varphi$ is the same as in Lemma \ref{wavelet basis 3}.
Therefore, from the definition of $\operatorname{Ext}$, it follows that
\begin{equation}\label{224}
\operatorname{Ext}\vec f
=\frac{1}{\varphi(-k_0)}\sum_{\lambda'\in\Lambda_{n-1}}
\sum_{I\in\mathscr{Q}(\mathbb{R}^{n-1})}\vec t^{(\lambda')}_{Q(I,k_0)}
\left[\theta^{(\lambda')}\otimes\varphi\right]_{Q(I,k_0)}(x)
\end{equation}
in $[\mathcal{S}_\infty'(\mathbb{R}^n)]^m$.
Let $W\in A_p(\mathbb{R}^n,\mathbb{C}^m)$
have $A_p$-dimensions $(d_W,\widetilde d_W,\Delta_W)$.
Applying Lemma \ref{l-new2}, we conclude that,
for any $\lambda'\in\Lambda_{n-1}$ and $I\in\mathscr{Q}(\mathbb{R}^{n-1})$,
\begin{align}\label{oprat-do}
\left|A_{Q(I,k_0),W}\vec t^{(\lambda')}_{Q(I,k_0)}\right|
&\lesssim(1+|k_0|)^{\Delta_W}\left|A_{I,V}\vec t^{(\lambda')}_{Q(I,k_0)}\right|.
\end{align}
This, combined with Lemma \ref{37}, further implies that,
for any $\lambda'\in\Lambda_{n-1}$,
\begin{align}\label{220}
\left\|\left\{\vec t^{(\lambda')}_Q\right\}_{Q\in\mathscr{Q}(\mathbb{R}^n)}\right\|
_{\dot b^{s,\tau}_{p,q}(W,\mathbb{R}^n)}
&\sim\left\|\left\{\vec t^{(\lambda')}_Q\right\}_{Q\in\mathscr{Q}(\mathbb{R}^n)}\right\|
_{\dot b^{s,\tau}_{p,q}(\mathbb{A}(W),\mathbb{R}^n)}\\
&=\left\|\left\{\left|A_{Q,W}\vec t^{(\lambda')}_Q\right|
\right\}_{Q\in\mathscr{Q}(\mathbb{R}^n)}
\right\|_{\dot b^{s,\tau}_{p,q}(\mathbb{R}^n)}\notag\\
&\lesssim\left\|\left\{\left|A_{I,V}
\left\langle\vec f,\theta^{(\lambda')}_I\right\rangle
\right|\right\}_{I\in\mathscr{Q}(\mathbb{R}^{n-1})}
\right\|_{\dot b^{s-\frac{1}{p},\frac{n}{n-1}\tau}_{p,q}(\mathbb{R}^{n-1})}\notag\\
&=\left\|\left\{\left\langle\vec f,\theta^{(\lambda')}_I\right\rangle
\right\}_{I\in\mathscr{Q}(\mathbb{R}^{n-1})}\right\|
_{\dot b^{s-\frac{1}{p},\frac{n}{n-1}\tau}_{p,q}(\mathbb{A}(V),\mathbb{R}^{n-1})}\notag\\
&\sim\left\|\left\{\left\langle\vec f, \theta^{(\lambda')}_I\right\rangle
\right\}_{I\in\mathscr{Q}(\mathbb{R}^{n-1})}
\right\|_{\dot b^{s-\frac{1}{p},\frac{n}{n-1}\tau}_{p,q}(V,\mathbb{R}^{n-1})}.\notag
\end{align}
Using Theorem \ref{wavelet 2}, we obtain
$\{\langle\vec f,\theta^{(\lambda')}_I\rangle\}_{I\in\mathscr{Q}(\mathbb{R}^{n-1})}
\in\dot b^{s-\frac{1}{p},\frac{n}{n-1}\tau}_{p,q}(V,\mathbb{R}^{n-1})$ and hence
\begin{equation}\label{221}
\left\{\vec t^{(\lambda')}_Q\right\}_{Q\in\mathscr{Q}(\mathbb{R}^n)}
\in\dot b^{s,\tau}_{p,q}(W,\mathbb{R}^n).
\end{equation}
Since $\theta^{(\lambda')}\otimes\varphi$ has bounded support
and $\mathscr{N}\in\mathbb{N}$ satisfies \eqref{NDau}
for $\dot B^{s,\tau}_{p,q}(W,\mathbb{R}^n)$,
we deduce that $[\theta^{(\lambda')}\otimes\varphi]_Q$
is a harmless constant multiple of a $\dot B^{s,\tau}_{p,q}(d_W)$-synthesis molecule on $Q$.
From this, \eqref{224}, \eqref{221}, Theorem \ref{89}(ii),
\eqref{220}, and Theorem \ref{wavelet 2},
it follows that $\operatorname{Ext}\vec f$ is well defined and
\begin{align*}
\left\|\operatorname{Ext}\vec f\right\|_{\dot B^{s,\tau}_{p,q}(W,\mathbb{R}^n)}
&\lesssim\sum_{\lambda'\in\Lambda_{n-1}}
\left\|\sum_{Q\in\mathscr{Q}(\mathbb{R}^n)}\vec t^{(\lambda')}_Q
\left[\theta^{(\lambda')}\otimes \varphi\right]_Q\right\|
_{\dot B^{s,\tau}_{p,q}(W,\mathbb{R}^n)}\\
&\lesssim\sum_{\lambda'\in\Lambda_{n-1}}
\left\|\left\{\vec t^{(\lambda')}_Q\right\}_{Q\in\mathscr{Q}(\mathbb{R}^n)}
\right\|_{\dot b^{s,\tau}_{p,q}(W,\mathbb{R}^n)}\\
&\lesssim\sum_{\lambda'\in\Lambda_{n-1}}
\left\|\left\{\left\langle\vec f,\theta^{(\lambda')}_I\right\rangle
\right\}_{I\in\mathscr{Q}(\mathbb{R}^{n-1})}
\right\|_{\dot b^{s-\frac{1}{p},\frac{n}{n-1}\tau}_{p,q}(V,\mathbb{R}^{n-1})}
=\left\|\vec f\right\|_{\dot B^{s-\frac{1}{p},\frac{n}{n-1}\tau}_{p,q}(V,\mathbb{R}^{n-1})}.
\end{align*}
This finishes the proof of \eqref{225}.

Finally, assume that $s\in(\frac{1}{p}+E,\infty)$, where $E$ is the same as in \eqref{219},
and \eqref{116} holds.
Then, by the definition of $\operatorname{Ext}$, we obtain,
for any $\vec f\in\dot B^{s-\frac{1}{p},\frac{n}{n-1}\tau}_{p,q}(V,\mathbb{R}^{n-1})$,
$$
(\operatorname{Tr}\circ\operatorname{Ext})\left(\vec f\right)
=\operatorname{Tr}\left(\sum_{\lambda'\in\Lambda_{n-1}}\sum_{I\in\mathscr{Q}(\mathbb{R}^{n-1})}
\left\langle\vec f,\theta^{(\lambda')}_I\right\rangle \operatorname{Ext}\theta^{(\lambda')}_I\right),
$$
where, by \eqref{identity pre}, we find that
$\{\operatorname{Ext}\theta^{(\lambda')}_I
\}_{\lambda'\in\Lambda_{n-1},I\in\mathscr{Q}(\mathbb{R}^{n-1})}
\subset\{\theta^{(\lambda)}_Q\}_{\lambda\in\Lambda_n,Q\in\mathscr{Q}(\mathbb{R}^n)}$
with $\{\theta^{(\lambda)}_Q\}_{\lambda\in\Lambda_n,Q\in\mathscr{Q}(\mathbb{R}^n)}$
being the orthonormal basis of $L^2(\mathbb R^n)$.
From these, the definition of $\operatorname{Tr}$, \eqref{identity}, and \eqref{204}, we infer that,
for any $\vec f\in\dot B^{s-\frac{1}{p},\frac{n}{n-1}\tau}_{p,q}(V,\mathbb{R}^{n-1})$,
\begin{align}\label{8.25x}
(\operatorname{Tr}\circ\operatorname{Ext})\left(\vec f\right)
&=\sum_{\lambda'\in\Lambda_{n-1}}\sum_{I\in\mathscr{Q}(\mathbb{R}^{n-1})}
\left\langle\vec f,\theta^{(\lambda')}_I\right\rangle
(\operatorname{Tr}\circ\operatorname{Ext})\left(\theta^{(\lambda')}_I\right)\\
&=\sum_{\lambda'\in\Lambda_{n-1}}\sum_{I\in\mathscr{Q}(\mathbb{R}^{n-1})}
\left\langle\vec f,\theta^{(\lambda')}_I\right\rangle\theta^{(\lambda')}_I
=\vec f\notag
\end{align}
in $[\mathcal{S}_\infty'(\mathbb R^n)]^m$.
This finishes the proof of Theorem \ref{extension B}.
\end{proof}

\begin{remark}
\begin{enumerate}[{\rm (i)}]
\item When $\tau=0$, Theorem \ref{trace B}
improves the results in \cite[Theorem 1.2]{fr08}
when $W$ is restricted to $A_p({\mathbb R}^n, {\mathbb C}^m)$-matrix weights.
Indeed, \cite[Theorem 1.2]{fr08} when the weight
$W$ is restricted to $A_p({\mathbb R}^n, {\mathbb C}^m)$-matrix weights
needs the assumption that
\begin{equation}\label{s}
s>\frac 1p+(n-1)\left(\frac1p-1\right)_++\frac{\beta_W-n+1}p,
\end{equation}
while Theorem \ref{trace B}
when $\tau=0$ only needs the assumption that
$s>\frac 1p+(n-1)(\frac1p-1)_+$, where $\beta_W\in[n,\infty)$
is the same as in \cite[Definition 2.20]{bhyy1}. However, we should mention
that \cite[Theorem 1.2]{fr08} concerns general weights in which setting
the assumption \eqref{s} was showed to
be necessary in \cite[Remark 2.3]{fr08}, that is, $\frac{\beta_W-n+1}p$ is irremovable.

\item When $m=1$, $W\equiv1$, and $V\equiv1$,
Theorems \ref{trace B} and \ref{extension B} contain \cite[Theorem 1.3]{syy10}, where $\tau\in[0,\frac{1}{p}+\frac{s+n-J}{n})$.
Furthermore, if $\tau=0$,
Theorems \ref{trace B} and \ref{extension B} reduce to the classical result \cite[Theorem 2.1]{ja78}.
\end{enumerate}
\end{remark}

\subsection{Trace Theorems for $\dot F^{s,\tau}_{p,q}(W,\mathbb{R}^n) $}

Throughout this section, let $\mathscr{N}\in\mathbb{N}$ satisfy \eqref{NDau}
for both $\dot F^{s,\tau}_{p,q}(W,\mathbb{R}^n)$
and $\dot F^{s-\frac{1}{p},\frac{n}{n-1}\tau}_{p,p}(V,\mathbb{R}^{n-1})$,
and let both $\{\theta^{(\lambda)}\}_{\lambda\in\Lambda_n}\subset C^{\mathscr{N}}(\mathbb{R}^n)$
and $\{\theta^{(\lambda')}\}_{\lambda'\in\Lambda_{n-1}}
\subset C^{\mathscr{N}}(\mathbb{R}^{n-1})$
be the same as in Lemma \ref{wavelet basis 3}.

\begin{theorem}\label{trace F}
Let $\tau\in[0,\infty)$, $p\in(0,\infty)$, $q\in(0,\infty]$, and
$s\in(\frac{1}{p}+E,\infty)$, where
\begin{equation}\label{E of F}
E:=\left\{\begin{aligned}
&\frac{n-1}{p}-n\tau&&\text{if }\frac{n}{n-1}\tau>\frac{1}{p},\\
&(n-1)\left(\frac{1}{p}-1\right)_+&&\text{otherwise}.
\end{aligned}\right.
\end{equation}
Let $W\in A_p(\mathbb{R}^n,\mathbb{C}^m)$ and $V\in A_p(\mathbb{R}^{n-1},\mathbb{C}^m)$.
Then the trace operator $\operatorname{Tr}$ can be extended to a
continuous linear operator
$\operatorname{Tr}:\ \dot F^{s,\tau}_{p,q}(W,\mathbb{R}^n)
\to\dot F^{s-\frac{1}{p},\frac{n}{n-1}\tau}_{p,p}(V,\mathbb{R}^{n-1})$
if and only if \eqref{116} holds.
\end{theorem}

\begin{remark}\label{rem trace F}
By Lemma \ref{mol no canc}, we find that the requirement that $s-\frac1p>E$,
with $E$ as defined in the theorem, is precisely the condition to
guarantee that synthesis molecules of the target space
$\dot F^{s-\frac1p,\frac{n}{n-1}\tau}_{p,p}(\mathbb R^{n-1},V)$
are not required to have any cancellation conditions.
Observe that this condition depends only on the explicit parameters of this space,
not on the $A_p$-dimension of the weight $V$.
\end{remark}

\begin{proof}[Proof of Theorem \ref{trace F}]
For the necessity, suppose that
$\operatorname{Tr}:\ \dot F^{s,\tau}_{p,q}(W,\mathbb{R}^n)
\to\dot F^{s-\frac{1}{p},\frac{n}{n-1}\tau}_{p,p}(V,\mathbb{R}^{n-1})$ is continuous.
By the standing assumptions stated in the beginning of this section,
it then follows that the assumptions of Lemma \ref{trace nec} are satisfied with
$\dot A^{s,\tau}_{p,q}(W,\mathbb{R}^n)=\dot F^{s,\tau}_{p,q}(W,\mathbb{R}^n)$ and $r=p$,
noting that $\dot F^{s-\frac{1}{p},\frac{n}{n-1}\tau}_{p,p}(V,\mathbb{R}^{n-1})
=\dot B^{s-\frac{1}{p},\frac{n}{n-1}\tau}_{p,p}(V,\mathbb{R}^{n-1})$
in the case of equal exponents. Thus, the necessity of \eqref{116}
follows from Lemma \ref{trace nec}.

Now, we consider the sufficiency and assume that \eqref{116} holds.
From Theorem \ref{wavelet 2}, we deduce that,
for any $\vec f\in\dot F^{s,\tau}_{p,q}(W,\mathbb{R}^n)$,
$\vec f=\sum_{\lambda\in\Lambda_n}\sum_{Q\in\mathscr{Q}(\mathbb{R}^n)}
\langle\vec f,\theta^{(\lambda)}_Q\rangle\theta^{(\lambda)}_Q$
in $[\mathcal{S}_\infty'(\mathbb{R}^n)]^m$.
Let $\operatorname{Tr}\vec f$ be  the same as in  \eqref{traceop}.
To prove that $\operatorname{Tr}\vec f$ is well defined and is continuous from
$\dot F^{s,\tau}_{p,q}(W,\mathbb{R}^n)$ to
$\dot F^{s-\frac{1}{p},\frac{n}{n-1}\tau}_{p,p}(V,\mathbb{R}^{n-1})$,
we let
$\vec t^{(\lambda)}:=\{\vec t^{(\lambda)}_Q\}_{Q\in\mathscr{Q}(\mathbb{R}^n)}$
be the same as in \eqref{dequen}. Clearly, in this case,
\eqref{223} and \eqref{inequa-n} remain true.
From this and Lemma \ref{37}, we infer that,
for any $\lambda\in\Lambda_n$ and $k\in\mathbb{Z}$ with $|k|\leq N$,
\begin{align*}
&\left\|\left\{\vec t^{(\lambda)}_{Q(I,k)}\right\}_{I\in\mathscr{Q}(\mathbb{R}^{n-1})}
\right\|_{\dot f^{s-\frac{1}{p},\frac{n}{n-1}\tau}_{p,p}(V,\mathbb{R}^{n-1})}\\
&\quad\sim\left\|\left\{\vec t^{(\lambda)}_{Q(I,k)}\right\}_{I\in\mathscr{Q}(\mathbb{R}^{n-1})}
\right\|_{\dot f^{s-\frac{1}{p},\frac{n}{n-1}\tau}_{p,p}(\mathbb{A}(V),\mathbb{R}^{n-1})}
=\left\|\left\{\left|A_{I,V}\vec t^{(\lambda)}_{Q(I,k)}\right|
\right\}_{I\in\mathscr{Q}(\mathbb{R}^{n-1})}
\right\|_{\dot f^{s-\frac{1}{p},\frac{n}{n-1}\tau}_{p,p}(\mathbb{R}^{n-1})}\notag\\
&\quad\lesssim\left\|\left\{\left|A_{Q(I,k),W}\vec t^{(\lambda)}_{Q(I,k)}\right|
\right\}_{I\in\mathscr{Q}(\mathbb{R}^{n-1})}
\right\|_{\dot f^{s-\frac{1}{p},\frac{n}{n-1}\tau}_{p,p}(\mathbb{R}^{n-1})}\notag\\
&\quad=\sup_{R\in\mathscr{Q}(\mathbb{R}^{n-1})}\frac{1}{|R|^{\frac{n}{n-1}\tau}}
\left[\sum_{j'=j_R}^\infty\sum_{\substack{I\in\mathscr{Q}_{j'}(\mathbb{R}^{n-1})\\I\subset R}}
|I|^{\frac{n}{n-1}-(\frac{s}{n-1}+\frac12)p}
\left|A_{Q(I,k),W}\vec t_{Q(I,k)}^{(\lambda)}\right|^p\right]^{\frac{1}{p}}\\
&\quad\lesssim\sup_{R\in\mathscr{Q}(\mathbb{R}^{n-1})}\frac{1}{|P_R|^\tau}
\left\{\sum_{j'=j_R}^\infty\int_{P_R}
\sum_{I\in\mathscr{Q}_{j'}(\mathbb{R}^{n-1})}\Bigg[|Q(I,k)|^{-\frac{s}{n}-\frac12}\right.\\
&\qquad\left.\times\left|A_{Q(I,k),W}\left\langle\vec f,\theta^{(\lambda)}_{Q(I,k)}\right\rangle
\right|\mathbf{1}_{E_{Q(I,k)}}(x)\Bigg]^p\,dx\right\}^{\frac{1}{p}},
\end{align*}
where $P_R$ is the same as in Lemma \ref{126} and,
for any $Q:=Q(I,k)\in\mathscr{Q}(\mathbb{R}^n)$,
\begin{equation}\label{EQ}
E_Q:=I\times\left[\ell(I)\left(k+\frac13\right),\ell(I)\left(k+\frac23\right)\right).
\end{equation}
This, together with the fact that
$\{E_{Q(I,k)}\}_{I\in\mathscr{Q}(\mathbb{R}^{n-1})}$
is pairwise disjoint, further implies that
\begin{align}\label{218y}
&\left\|\left\{\vec t^{(\lambda)}_{Q(I,k)}\right\}_{I\in\mathscr{Q}(\mathbb{R}^{n-1})}
\right\|_{\dot f^{s-\frac{1}{p},\frac{n}{n-1}\tau}_{p,p}(V,\mathbb{R}^{n-1})}\\
&\quad\lesssim\sup_{R\in\mathscr{Q}(\mathbb{R}^{n-1})}\frac{1}{|P_R|^\tau}
\left(\int_{P_R}\left\{ \sum_{I\in\mathscr{Q}(\mathbb{R}^{n-1})}
\left[|Q(I,k)|^{-\frac{s}{n}-\frac12}\right.\right.\right.\notag\\
&\qquad\left.\left.\left.\times
\left|A_{Q(I,k),W}\left\langle\vec f,\theta^{(\lambda)}_{Q(I,k)}\right\rangle
\right|\mathbf{1}_{E_{Q(I,k)}}(x)\right]^q\right\}^{\frac{p}{q}}\,dx\right)^{\frac{1}{p}}\notag\\
&\quad\leq\sup_{P\in\mathscr{Q}(\mathbb{R}^n)}\frac{1}{|P|^\tau}
\left(\int_P\left\{\sum_{Q\in\mathscr{Q}(\mathbb{R}^n),Q\subset P}
\left[|Q|^{-\frac{s}{n}}
\left|A_{Q,W}\left\langle\vec f,\theta^{(\lambda)}_Q\right\rangle
\right|\widetilde{\mathbf{1}}_Q (x)\right]^q\right\}^{\frac{p}{q}}\,dx\right)^{\frac{1}{p}}\notag\\
&\quad=\left\|\left\{\left\langle\vec f,
\theta^{(\lambda)}_Q \right\rangle\right\}_{Q\in\mathscr{Q}(\mathbb{R}^n)}
\right\|_{\dot f^{s,\tau}_{p,q}(\mathbb{A}(W),\mathbb{R}^n)}
\sim\left\|\left\{\left\langle\vec f,
\theta^{(\lambda)}_Q \right\rangle\right\}_{Q\in\mathscr{Q}(\mathbb{R}^n)}
\right\|_{\dot f^{s,\tau}_{p,q}(W,\mathbb{R}^n)}.\notag
\end{align}
By Theorem \ref{wavelet 2}, we conclude that
$\{\langle\vec f,\theta^{(\lambda)}_Q\rangle\}_{Q\in\mathscr{Q}(\mathbb{R}^n)}\in\dot f^{s,\tau}_{p,q}(W,\mathbb{R}^n)$ and hence
\begin{equation}\label{217y}
\left\{\vec t^{(\lambda)}_{Q(I,k)}\right\}_{I\in\mathscr{Q}(\mathbb{R}^{n-1})}
\in\dot f^{s-\frac{1}{p},\frac{n}{n-1}\tau}_{p,p}(V,\mathbb{R}^{n-1}).
\end{equation}
Let $V\in A_p(\mathbb{R}^{n-1},\mathbb{C}^m)$ have the $A_p$-dimension $d_V\in[0, n-1)$.
By Remark \ref{rem trace F}, that the assumption $ s\in(\frac{1}{p}+E,\infty) $ implies that
the vanishing moment condition of the synthesis molecule is void for the space
$\dot F^{s-\frac{1}{p},\frac{n}{n-1}\tau}_{p,p}(V,\mathbb{R}^{n-1})$.
This, combined with the fact that $\theta^{(\lambda)}_{Q(I,k)}(\cdot,0)$ has bounded support, further implies that the function
$[\ell(I)]^{\frac12}\theta^{(\lambda)}_{Q(I,k)}(\cdot,0)$ is a harmless constant multiple of
an $\dot F^{s-\frac{1}{p},\frac{n}{n-1}\tau}_{p,p}(d_V)$-synthesis molecule on $I$.
From this, \eqref{223}, \eqref{217y}, Theorem \ref{89}(ii),
\eqref{218y}, and Theorem \ref{wavelet 2},
we deduce that $\operatorname{Tr}\vec f$ is well defined and
\begin{align*}
\left\|\operatorname{Tr}\vec f\right\|_{\dot F^{s-\frac{1}{p},\frac{n}{n-1}\tau}_{p,p}(V,\mathbb{R}^{n-1})}
&\lesssim\sum_{\lambda\in\Lambda_n}\sum_{k=-r}^r
\left\|\sum_{I\in\mathscr{Q}(\mathbb{R}^{n-1})}
\vec t^{(\lambda)}_{Q(I,k)}[\ell(I)]^{\frac12}\theta^{(\lambda)}_{Q(I,k)}(\cdot,0)\right\|
_{\dot F^{s-\frac{1}{p},\frac{n}{n-1}\tau}_{p,p}(V,\mathbb{R}^{n-1})}\\
&\lesssim\sum_{\lambda\in\Lambda_n}\sum_{k=-r}^r
\left\|\left\{\vec t^{(\lambda)}_{Q(I,k)}\right\}_{I\in\mathscr{Q}(\mathbb{R}^{n-1})}\right\|
_{\dot f^{s-\frac{1}{p},\frac{n}{n-1}\tau}_{p,p}(V,\mathbb{R}^{n-1})}\\
&\lesssim\sum_{\lambda\in\Lambda_n}
\left\|\left\{\left\langle\vec f,
\theta^{(\lambda)}_Q \right\rangle\right\}_{Q\in\mathscr{Q}(\mathbb{R}^n)}
\right\|_{\dot f^{s,\tau}_{p,q}(W,\mathbb{R}^n)}
\sim\left\|\vec f\right\|_{\dot F^{s,\tau}_{p,q}(W,\mathbb{R}^n)}.
\end{align*}
This finishes the proof of the sufficiency and hence Theorem \ref{trace F}.
\end{proof}

We now present an extension theorem for $\operatorname{Ext}$.

\begin{lemma}\label{44}
Let $s\in\mathbb{R}$, $\tau\in[0,\infty)$,
$p\in(0,\infty)$, $q\in(0,\infty]$, and $\delta\in(0,1)$.
Suppose that, for any $Q\in\mathscr{Q}$, $E_Q\subset Q$ is a measurable set
with $|E_Q|\geq\delta|Q|$.
Then, for any sequence $t:=\{t_Q\}_{Q\in\mathscr{Q}}\subset\mathbb{C}$,
$$
\|t\|_{\dot f^{s,\tau}_{p,q}(\mathbb R^n)}
\sim\left\|\left\{2^{j(s+\frac{n}{2})}
\sum_{Q\in\mathscr{Q}_j}t_Q\mathbf{1}_{E_Q}
\right\}_{j\in\mathbb Z}\right\|_{L\dot F_{p,q}^\tau(\mathbb R^n)},
$$
where the positive equivalence constants are independent of $t$.
\end{lemma}

\begin{theorem}\label{extension F}
Let $s\in\mathbb{R}$, $\tau\in[0,\infty)$, $p\in(0,\infty)$, $q\in(0,\infty]$,
$W\in A_p(\mathbb{R}^n,\mathbb{C}^m)$, and $V\in A_p(\mathbb{R}^{n-1},\mathbb{C}^m)$.
Assume that \eqref{127} holds.
Then the extension operator $\operatorname{Ext}$ can be extended to a
continuous linear operator
\begin{equation}\label{225y}
\operatorname{Ext}:\
\dot F^{s-\frac{1}{p},\frac{n}{n-1}\tau}_{p,p}(V,\mathbb{R}^{n-1})
\to\dot F^{s,\tau}_{p,q}(W,\mathbb{R}^n).
\end{equation}
Furthermore, if $s\in(\frac{1}{p}+E,\infty)$, where $E$ is the same as in \eqref{E of F},
and \eqref{116} also holds,
then $\operatorname{Tr}\circ\operatorname{Ext}$ is the identity on
$\dot F^{s-\frac{1}{p},\frac{n}{n-1}\tau}_{p,p}(V,\mathbb{R}^{n-1})$.
\end{theorem}

\begin{proof}
The proof of the present theorem is similar to that of
Theorem \ref{extension B} and hence we only describe the difference.

By \eqref{204}, we find that,
for any $\vec f\in\dot F^{s-\frac{1}{p},\frac{n}{n-1}\tau}_{p,p}(V,\mathbb{R}^{n-1})$,
$\vec f=\sum_{\lambda'\in\Lambda_{n-1}}\sum_{I\in\mathscr{Q}(\mathbb{R}^{n-1})}
\langle\vec f,\theta^{(\lambda')}_I\rangle\theta^{(\lambda')}_I$
in $[\mathcal{S}_\infty'(\mathbb{R}^{n-1})]^m$.
Define $\operatorname{Ext}\vec f$ to be the same as in \eqref{exten-op}.
To show that $\operatorname{Ext}\vec f$ is well defined and
is  continuous  from $\dot F^{s-\frac{1}{p},\frac{n}{n-1}\tau}_{p,p}(V,\mathbb{R}^{n-1})$ to
$\dot F^{s,\tau}_{p,q}(W,\mathbb{R}^n)$, we let, for each $\lambda'\in\Lambda_{n-1}$,
$\vec t^{(\lambda')}:=\{\vec t^{(\lambda')}_Q\}_{Q\in\mathscr{Q}(\mathbb{R}^n)}$
be the same as in \eqref{eex}. Then \eqref{224} remains true in this case.

Let $W\in A_p(\mathbb{R}^n,\mathbb{C}^m)$
have $A_p$-dimensions $(d_W,\widetilde d_W,\Delta_W)$.
By Lemma \ref{l-new2}, we still have the estimate \eqref{oprat-do}.
This, together with Lemmas \ref{37} and \ref{44}
with $E_Q$ being the same as in \eqref{EQ},
further implies that, for any $\lambda'\in\Lambda_{n-1}$,
\begin{align*}
&\left\|\left\{\vec t^{(\lambda')}_Q\right\}_{Q\in\mathscr{Q}(\mathbb{R}^n)}
\right\|_{\dot f^{s,\tau}_{p,q}(W,\mathbb{R}^n)}\\
&\quad\sim\left\|\left\{\vec t^{(\lambda')}_Q\right\}_{Q\in\mathscr{Q}(\mathbb{R}^n)}
\right\|_{\dot f^{s,\tau}_{p,q}(\mathbb{A}(W),\mathbb{R}^n)}
=\left\|\left\{\left|A_{Q,W}\vec t^{(\lambda')}_Q\right|
\right\}_{Q\in\mathscr{Q}(\mathbb{R}^n)}
\right\|_{\dot f^{s,\tau}_{p,q}(\mathbb{R}^n)}\\
&\quad\sim\left\|\left\{2^{j(s+\frac{n}{2})}
\sum_{Q\in\mathscr{Q}_j}\left|A_{Q,W}\vec t^{(\lambda')}_Q\right|\mathbf{1}_{E_Q}
\right\}_{j\in\mathbb Z}\right\|_{L\dot F_{p,q}^\tau(\mathbb{R}^n)}\\
&\quad=\sup_{P\in\mathscr{Q}(\mathbb{R}^n)}\frac{1}{|P|^\tau}
\left(\int_P\left\{\sum_{\substack{I\in\mathscr{Q}(\mathbb{R}^{n-1})\\
Q(I,k_0)\subset P}}\left[|Q(I,k_0)|^{-\frac{s}{n}-\frac12}
\left|A_{Q(I,k_0),W}\vec t^{(\lambda')}_{Q(I,k_0)}\right|
\mathbf{1}_{E_{Q(I,k_0)}}(x)\right]^q
\right\}^{\frac{p}{q}}\,dx\right)^{\frac{1}{p}}\\
&\quad\lesssim\sup_{P\in\mathscr{Q}(\mathbb{R}^n)}\frac{1}{|P|^\tau}
\left(\int_P\left\{\sum_{\substack{I\in\mathscr{Q}(\mathbb{R}^{n-1})\\Q(I,k_0)\subset P}}
\left[|I|^{-\frac{s}{n-1}-\frac12}
\left|A_{I,V}\left\langle\vec f,\theta^{(\lambda')}_I\right\rangle\right|
\mathbf{1}_{E_{Q(I,k_0)}}(x)\right]^q
\right\}^{\frac{p}{q}}\,dx\right)^{\frac{1}{p}}.
\end{align*}
Using this, combined with the fact that
$\{E_{Q(I,k_0)}\}_{I\in\mathscr{Q}(\mathbb{R}^{n-1})}$
is pairwise disjoint and Lemma \ref{37}, we obtain,
\begin{align}\label{220y}
&\left\|\left\{\vec t^{(\lambda')}_Q\right\}_{Q\in\mathscr{Q}(\mathbb{R}^n)}\right\|
_{\dot f^{s,\tau}_{p,q}(W,\mathbb{R}^n)}\\
&\quad\lesssim\sup_{P\in\mathscr{Q}(\mathbb{R}^n)}\frac{1}{|P|^\tau}
\left\{\int_P\sum_{\substack{I\in\mathscr{Q}(\mathbb{R}^{n-1})\\Q(I,k_0)\subset P}}
\left[|I|^{-\frac{s}{n-1}-\frac12}
\left|A_{I,V}\left\langle\vec f,\theta^{(\lambda')}_I\right\rangle\right|\right.
\mathbf{1}_{E_{Q(I,k_0)}}(x)\Bigg]^p\,dx\right\}^{\frac{1}{p}}\notag\\
&\quad\lesssim\sup_{P\in\mathscr{Q}(\mathbb{R}^n)}\frac{1}{|I(P)|^{\frac{n}{n-1}\tau}}
\left[\sum_{\substack{I\in\mathscr{Q}(\mathbb{R}^{n-1})\\I\subset I(P)}}
|I|^{\frac{n}{n-1}-(\frac{s}{n-1}-\frac12)p}
\left|A_{I,V}\left\langle\vec f,\theta^{(\lambda')}_I\right\rangle\right|^p \right]^{\frac{1}{p}}\notag\\
&\quad\lesssim\sup_{R\in\mathscr{Q}(\mathbb{R}^{n-1})}\frac{1}{|R|^{\frac{n}{n-1}\tau}}
\left\{\int_R \sum_{\substack{I\in\mathscr{Q}(\mathbb{R}^{n-1})\\I\subset R}}
\left[|I|^{-\frac{s-\frac{1}{p}}{n-1}}
\left|A_{I,V}\left\langle\vec f, \theta^{(\lambda')}_I\right\rangle\right|\right.
\widetilde{\mathbf{1}}_I(x)\Bigg]^p\,dx\right\}^{\frac{1}{p}}\notag\\
&\quad=\left\|\left\{\left|A_{I,V}\left\langle\vec f,\theta^{(\lambda')}_I \right\rangle\right|
\right\}_{I\in\mathscr{Q}(\mathbb{R}^{n-1})}
\right\|_{\dot f^{s-\frac{1}{p},\frac{n}{n-1}\tau}_{p,p}(\mathbb{R}^{n-1})}\notag\\
&\quad=\left\|\left\{\left\langle\vec f,\theta^{(\lambda')}_I \right\rangle\right\}_{I\in\mathscr{Q}(\mathbb{R}^{n-1})}\right\|
_{\dot f^{s-\frac{1}{p},\frac{n}{n-1}\tau}_{p,p}(\mathbb{A}(V),\mathbb{R}^{n-1})}
\sim\left\|\left\{\left\langle\vec f,\theta^{(\lambda')}_I \right\rangle\right\}_{I\in\mathscr{Q}(\mathbb{R}^{n-1})}\right\|
_{\dot f^{s-\frac{1}{p},\frac{n}{n-1}\tau}_{p,p}(V,\mathbb{R}^{n-1})}.\notag
\end{align}
By Theorem \ref{wavelet 2}, we find that
$\{\langle\vec f,\theta^{(\lambda')}_I\rangle\}_{I\in\mathscr{Q}(\mathbb{R}^{n-1})}
\in\dot f^{s-\frac{1}{p},\frac{n}{n-1}\tau}_{p,p}(V,\mathbb{R}^{n-1})$ and hence
\begin{equation}\label{221y}
\left\{\vec t^{(\lambda')}_Q\right\}_{Q\in\mathscr{Q}(\mathbb{R}^n)}
\in\dot f^{s,\tau}_{p,p}(W,\mathbb{R}^n).
\end{equation}

Notice that, for any $Q\in\mathscr{Q}(\mathbb{R}^n)$,  $[\theta^{(\lambda')}\otimes \varphi]_Q$
is a harmless constant multiple of an $\dot F^{s,\tau}_{p,q}(d_W)$-synthesis molecule on $Q$.
From this, \eqref{224}, \eqref{221y}, Theorem \ref{89}(ii),
\eqref{220y}, and Theorem \ref{wavelet 2},
we infer that $\operatorname{Ext}\vec f$ is well defined and
\begin{align*}
\left\|\operatorname{Ext}\vec f\right\|_{\dot F^{s,\tau}_{p,q}(W,\mathbb{R}^n)}
&\lesssim\sum_{\lambda'\in\Lambda_{n-1}}
\left\|\sum_{Q\in\mathscr{Q}(\mathbb{R}^n)}\vec t^{(\lambda')}_Q
\left[\theta^{(\lambda')}\otimes \varphi \right]_Q \right\|
_{\dot F^{s,\tau}_{p,q}(W,\mathbb{R}^n)}\\
&\lesssim\sum_{\lambda'\in\Lambda_{n-1}}
\left\|\left\{\vec t^{(\lambda')}_Q\right\}_{Q\in\mathscr{Q}(\mathbb{R}^n)}\right\|
_{\dot f^{s,\tau}_{p,q}(W,\mathbb{R}^n)}\\
&\lesssim\sum_{\lambda'\in\Lambda_{n-1}}
\left\|\left\{\left\langle\vec f,\theta^{(\lambda')}_I \right\rangle\right\}_{I\in\mathscr{Q}(\mathbb{R}^{n-1})}\right\|
_{\dot f^{s-\frac{1}{p},\frac{n}{n-1}\tau}_{p,p}(V,\mathbb{R}^{n-1})}
\sim\left\|\vec f\right\|_{\dot F^{s-\frac{1}{p},\frac{n}{n-1}\tau}_{p,p}(V,\mathbb{R}^{n-1})}.
\end{align*}
This finishes the proof of \eqref{225y}.

Finally, we assume that $s\in(\frac{1}{p}+E,\infty)$,
where $E$ is the same as in \eqref{E of F}, and \eqref{116} holds.
Repeating the proof of \eqref{8.25x}, we find that \eqref{8.25x}
for any $\vec f\in\dot F^{s-\frac{1}{p},\frac{n}{n-1}\tau}_{p,p}(V,\mathbb{R}^{n-1})$
also holds. This finishes the proof of Theorem \ref{extension F}.
\end{proof}

\begin{remark}
When $m=1$, $W\equiv 1$, and $V\equiv 1$,
Theorems \ref{trace B} and \ref{extension F} in this case contain \cite[Theorem 1.4]{syy10}
in which $\tau\in[0,\frac{1}{p}+\frac{s+n-J}{n})$.
Furthermore, if $\tau=0$, Theorems \ref{trace B} and \ref{extension F} in this case
go back to the classical result \cite[Theorem 5.1]{ja77}.
\end{remark}

\section{Calder\'on--Zygmund Operators}
\label{C-Z operators}

Calder\'on--Zygmund operators are formally given by expressions of the form
\begin{equation*}
Tf(x)=\int_{\mathbb R^n}K(x,y)f(y)\,dy,
\end{equation*}
where the kernel $K(x,y)$ satisfies size, cancellation, and smoothness estimates
resembling those of the Hilbert kernel $K(x,y):=\frac{1}{x-y}$ when $n=1$
or the Riesz kernels $K_i(x,y):=(x_i-y_i)|x-y|^{-n-1}$
with $i\in\{1,\ldots,n\}$ when $n\geq2$.
One of the major themes in harmonic analysis is obtaining sufficient
(and sometimes necessary) conditions for their boundedness on different function spaces.
Many of these results are modeled after the ``$T(1)$ theorem''
of David and Journ\'e \cite{dj84} concerning the boundedness on $L^2$.
For analogous results in Triebel--Lizorkin spaces $\dot F^s_{p,q}$,
an approach via the atomic and the molecular decompositions was developed first
in a limited range of the parameters in \cite{fhjw},
and extended to the full scale of these spaces in \cite{ftw88,tor}.
This approach consists of showing that
\begin{enumerate}[\rm(i)]
\item\label{CZ a2mol} under appropriate assumptions of Calder\'on--Zygmund type,
an operator $T$ maps sufficiently regular atoms to molecules of the same space, and
\item\label{a2mol bd} any operator with this property can be extended to act boundedly on the whole space.
\end{enumerate}

Our goal in this section is to extend such results to
$\dot A^{s,\tau}_{p,q}(W)$ in place of $\dot F^s_{p,q}$.
Since we already know from Lemma \ref{mol vs old} that
molecules for $\dot A^{s,\tau}_{p,q}(W)$
[more precisely, $\dot A^{s,\tau}_{p,q}(d)$-molecules if $W\in A_p$ has the $A_p$-dimension $d$]
coincide with $\dot F^{\widetilde s}_{\widetilde r,\widetilde r}$-molecules
for suitable $\widetilde s$ and $\widetilde r$,
the existing theory of \cite{ftw88,tor} would offer us a soft short-cut to such an extension:
the same Calder\'on--Zygmund conditions that guarantee the mapping of
atoms to $\dot F^{\widetilde s}_{\widetilde r,\widetilde r}$-molecules
also give us the mapping of atoms to $\dot A^{s,\tau}_{p,q}(d)$-molecules
because the two classes of molecules are actually equal.
Then it would remain for us to check the analogue of point \eqref{a2mol bd}
above for $\dot A^{s,\tau}_{p,q}(W)$ in place of $\dot F^s_{p,q}$.

Nevertheless, we decide to take a longer route
and develop both \eqref{CZ a2mol} and \eqref{a2mol bd} in detail in our setting,
as this gives us an opportunity to refine some issues of the theory
and eventually obtain a result that also improves the existing results
in the classical spaces $\dot F^s_{p,q}$ found in \cite{ftw88,tor}.
As we did with molecules, a key is to work with continuous parameters as entities in their own right,
avoiding as much as possible a need to impose separate conditions
on the integer and the fractional parts of these parameters.
While it is still necessary to work with integer and fractional parts in the proofs,
we largely succeed in avoiding them in the final statements.

The structure of this section is as follows.
In Subsection \ref{extension}, we discuss the general problem of
extending a given operator $T:\ \mathcal{S}_\infty\to\mathcal{S}_\infty'$
to $\widetilde T:\ \dot A^{s,\tau}_{p,q}(W)\to\dot A^{s,\tau}_{p,q}(W)$
and prove that such an extension exists (in appropriate sense)
provided that $W\in A_p$ has the $A_p$-dimension $d$ and
$T$ maps a class of atoms into $\dot A^{s,\tau}_{p,q}(d)$-molecules.
In Subsection \ref{ss CZO}, we introduce some Calder\'on--Zygmund conditions
and give some estimates associated with them.
In Subsection \ref{ss T1}, we establish the $T(1)$ theorem for $\dot A^{s,\tau}_{p,q}(W)$.
Finally, in Section \ref{T1 compare}, we compare our results with earlier results.

\subsection{Extension of Weakly Defined Operators}\label{extension}

Let us discuss the problem of extending a given operator
$T:\ \mathcal{S}_\infty\to\mathcal{S}_\infty'$ to $\widetilde T:\ \dot A^{s,\tau}_{p,q}(W)\to\dot A^{s,\tau}_{p,q}(W)$.
Let an infinite matrix $B:=\{b_{Q,P}\}_{Q,P\in\mathscr{Q}}\subset\mathbb{C}$ be given.
For any sequence $t:=\{t_R\}_{R\in\mathscr{Q}}\subset\mathbb{C}$,
we define $Bt:=\{(Bt)_Q\}_{Q\in\mathscr{Q}}$ by setting,
for any $Q\in\mathscr{Q}$,
$(Bt)_Q:=\sum_{R\in\mathscr{Q}}b_{Q,R}t_R$
if the above summation is absolutely convergent.
We begin with the following lemma.

\begin{lemma}\label{That1}
Let $\varphi,\psi\in\mathcal{S}$ satisfy \eqref{19}, \eqref{20}, and \eqref{21}. Suppose that
$T\in\mathcal L(\mathcal{S}_\infty,\mathcal{S}_\infty'),$
and define the infinite matrix
$\widehat T:=\{\langle \varphi_Q,T\psi_R\rangle\}_{Q,R\in\mathscr{Q}}.$
Then $\widehat T$ maps
$S_\varphi\mathcal{S}_\infty:=
\{\{\langle\varphi_Q,f\rangle\}_{Q\in\mathscr Q}:\ f\in\mathcal{S}_\infty\}$
to
$$S_\varphi{\mathcal{S}_\infty'}:=
\left\{\{\langle\varphi_Q,f\rangle\}_{Q\in\mathscr Q}:\ f\in{\mathcal{S}_\infty'}\right\}
$$
and satisfies the identity $\widehat T\circ S_{\varphi}=S_{\varphi}\circ T$ on $\mathcal{S}_\infty$.
\end{lemma}

\begin{proof}
Let $f\in\mathcal{S}_\infty$. By Lemma \ref{7}, we find that
$f=\sum_{R\in\mathscr{Q}}\psi_R\langle\varphi_R,f\rangle$ in $\mathcal{S}_\infty$.
Using $T\in\mathcal L(\mathcal{S}_\infty,\mathcal{S}_\infty')$, we obtain
$Tf=\sum_{R\in\mathscr{Q}}T\psi_R\langle\varphi_R,f\rangle$
with convergence in $\mathcal{S}_\infty'$.
By $\varphi_Q\in\mathcal{S}_\infty$, we further conclude that
the aforementioned convergence in $\mathcal{S}_\infty'$ implies in particular that
\begin{equation*}
\langle\varphi_Q,Tf\rangle
=\sum_{R\in\mathscr{Q}}\langle\varphi_Q,T\psi_R\rangle\langle\varphi_R,f\rangle.
\end{equation*}
Since $\{\langle\varphi_R,f\rangle\}_{R\in\mathscr{Q}}$ is an arbitrary element of $S_\varphi\mathcal{S}_\infty$,
it follows that $\widehat T:\ S_\varphi\mathcal{S}_\infty\to\mathbb C^{\mathscr{Q}}$, where
$$
\mathbb C^{\mathscr{Q}}:=\left\{t:\ t:=\{t_Q\}_{Q\in\mathscr{Q}}\subset\mathbb C\right\},
$$
is well defined.
Noticing $Tf\in\mathcal{S}_\infty'$, this also shows that $\widehat T:\ S_\varphi\mathcal{S}_\infty\to S_\varphi{\mathcal{S}_\infty'}$
as well as the formula $\widehat T\circ S_{\varphi}=S_{\varphi}\circ T$.
This finishes the proof of Lemma \ref{That1}.
\end{proof}

\begin{lemma}\label{That2}
Let $s\in\mathbb R$, $\tau\in[0,\infty)$, $p\in(0,\infty)$, $q\in(0,\infty]$, and $W\in A_p$.
Let $\varphi,\psi\in\mathcal{S}$ satisfy \eqref{19}, \eqref{20}, and \eqref{21}.
Let $T\in\mathcal L(\mathcal{S}_\infty,\mathcal{S}_\infty')$
and suppose that
$\widehat T:=\{\langle\varphi_Q,T\psi_R\rangle\}_{Q,R\in\mathscr{Q}}$
has an extension $\widetilde{\widehat T}\in\mathcal L(\dot a^{s,\tau}_{p,q}(W))$. Then
$\widetilde{T}:=T_\psi\circ\widetilde{\widehat T}\circ S_{\varphi}$
is an extension of $T$ and $\widetilde T\in\mathcal L(\dot A^{s,\tau}_{p,q}(W))$.
\end{lemma}

\begin{proof}
By \cite[Theorem 3.29]{bhyy1}, we conclude that
\begin{equation*}
S_\varphi\in\mathcal L(\dot A^{s,\tau}_{p,q}(W),\dot a^{s,\tau}_{p,q}(W))
\text{ and }
T_\psi\in\mathcal L(\dot a^{s,\tau}_{p,q}(W),\dot A^{s,\tau}_{p,q}(W)),
\end{equation*}
which, together with $\widetilde{\widehat T}\in\mathcal L(\dot a^{s,\tau}_{p,q}(W))$,
further implies that $\widetilde{T}\in\mathcal L(\dot A^{s,\tau}_{p,q}(W))$.
Since $\widetilde{\widehat T}$ is an extension of $\widehat T$,
from Lemma \ref{That1}, we deduce that,
for any $\vec f\in(\mathcal{S}_\infty)^m$,
$$
\widetilde{T}\vec f
=T_\psi\circ\widehat T\circ S_{\varphi}\vec f
=T_\psi\circ S_{\varphi}\circ T\vec f
=T\vec f.
$$
Using this and Proposition \ref{6.18}, we find that $\widetilde{T}$ is an extension of $T$.
This finishes the proof of Lemma \ref{That2}.
\end{proof}

\begin{corollary}\label{That3}
Let $s\in\mathbb R$, $\tau\in[0,\infty)$, $p\in(0,\infty)$, $q\in(0,\infty]$, and $d\in[0,n)$.
Let $\varphi,\psi\in\mathcal{S}$ satisfy \eqref{19}, \eqref{20}, and \eqref{21}, and
let $T\in\mathcal L(\mathcal{S}_\infty,\mathcal{S}_\infty')$. Suppose that
$\widehat T:=\{\langle\varphi_Q,T\psi_R\rangle\}_{Q,R\in\mathscr{Q}}$
is $\dot a^{s,\tau}_{p,q}(d)$-almost diagonal and $W\in A_p$ has the $A_p$-dimension $d$.
Then $T$ has an extension $\widetilde T\in\mathcal L(\dot A^{s,\tau}_{p,q}(W))$.
\end{corollary}

\begin{proof}
By Lemma \ref{ad BF2} and Definition \ref{def ad tau},
we conclude that the operator $\widehat T$ has an extension $\widetilde{\widehat T}\in\mathcal L(\dot a^{s,\tau}_{p,q}(W))$.
From this and Lemma \ref{That2}, we infer that
$T$ has an extension $\widetilde T\in\mathcal L(\dot A^{s,\tau}_{p,q}(W))$.
This finishes the proof of Corollary \ref{That3}.
\end{proof}

In what follows, we use $C_{\mathrm{c}}^\infty$ to denote
the set of all functions $f\in C^\infty$ on $\mathbb R^n$ with compact support.

\begin{definition}
Let $L,N\in(0,\infty)$.
A function $a_Q\in C_{\mathrm{c}}^\infty$ is called an \emph{$(L,N)$-atom} on a cube $Q$ if
\begin{enumerate}[\rm(i)]
\item $\operatorname{supp}a_Q\subset3Q$;
\item $\int_{\mathbb{R}^n}x^\gamma a_Q(x)\,dx=0$
for any $\gamma\in\mathbb{Z}_+^n$ with $|\gamma|\leq L$;
\item $|D^\gamma a_Q(x)|\leq|Q|^{-\frac12-\frac{|\gamma|}{n}}$
for any $x\in\mathbb{R}^n$ and $\gamma\in\mathbb{Z}_+^n$ with $|\gamma|\leq N$.
\end{enumerate}
\end{definition}

\begin{proposition}\label{ext}
Let $s\in\mathbb R$, $\tau\in[0,\infty)$, $p\in(0,\infty)$, $q\in(0,\infty]$,
$d\in[0,n)$, and $L,N\in(0,\infty)$.
\begin{enumerate}[\rm(i)]
\item\label{ext1} If $T\in\mathcal L(\mathcal S,\mathcal S')$
maps $(L,N)$-atoms to $\dot a^{s,\tau}_{p,q}(d)$-synthesis molecules
and $W\in A_p$ has the $A_p$-dimension $d$,
then there exists an operator $\widetilde T\in\mathcal L(\dot A^{s,\tau}_{p,q}(W))$ that agrees with $T$ on $(\mathcal{S}_\infty)^m$.
\item\label{ext2} If, in addition, there exists an operator $T_2\in\mathcal L(L^2)$
that agrees with $T$ on $\mathcal S$,
then $\widetilde T$ agrees with $T$ on
$(\mathcal S)^m\cap\dot A^{s,\tau}_{p,q}(W)$.
\end{enumerate}
\end{proposition}

\begin{proof}
Let $\varphi,\psi\in\mathcal{S}$ satisfy \eqref{19}, \eqref{20}, and \eqref{21}.
In the first step, we decompose $\psi_R$ in terms of $(L,N)$-atoms with the given numbers $L,N\in(0,\infty)$,
following the proof of \cite[Theorem 4.1]{fj90}.
According to \cite[p.\,783]{fj85}, we can pick a function $\theta\in\mathcal S$ such that
$$
\operatorname{supp}\theta\subset\{x\in\mathbb R^n:\ |x|\leq1\},\
\int_{\mathbb R^n}x^\gamma\theta(x)\,dx=0
\text{ if }\gamma\in\mathbb{Z}_+^n\text{ and }|\gamma|\leq L,
$$
and
$$
\left|\widehat\theta(\xi)\right|\geq C>0
\text{ if }\xi\in\mathbb{R}^n\text{ and }\frac12\leq|\xi|\leq2,
$$
where $C$ is a positive constant.
From \eqref{19}, it follows that $\widehat\psi/\widehat\theta\in\mathcal S$
and hence $\psi=\theta*\eta$, where $\eta:=(\widehat\psi/\widehat\theta)^\vee\in\mathcal S$.
Thus, from the definitions of $\psi$ and $\theta$,
it follows that, for any $x\in\mathbb{R}^n$,
\begin{align*}
\psi(x)=\int_{\mathbb R^n}\theta(x-y)\eta(y)\,dy
=\sum_{k\in\mathbb Z^n}\int_{Q_{0,k}}\theta(x-y)\eta(y)\,dy
=:\sum_{k\in\mathbb Z^n}g_k(x).
\end{align*}
Here, for any $k\in\mathbb Z^n$, $\operatorname{supp}g_k\subset[-1,2)^n+k=3Q_{0,k}$.
By the Fubini theorem, we find that,
for any $\gamma\in\mathbb{Z}_+^n$ with $|\gamma|\leq L$,
\begin{align*}
\int_{\mathbb R^n}x^\gamma g_k(x)\,dx
=\sum_{\beta\in\mathbb{Z}_+^n,\,\beta\leq\gamma}\binom{\gamma}{\beta}
\int_{Q_{0,k}}\left[\int_{\mathbb R^n}(x-y)^{\beta}\theta(x-y)\,dx\right]
y^{\gamma-\beta}\eta(y)\,dy
=0.
\end{align*}
From $\theta,\eta\in\mathcal{S}$, we also deduce that,
for any $\gamma\in\mathbb{Z}_+^n$ with $|\gamma|\leq N$
and for any $x\in\mathbb{R}^n$,
\begin{equation*}
|\partial^\gamma g_k(x)|
\leq\int_{Q_{0,k}}\left|\partial^\gamma\theta(x-y)\right||\eta(y)|\,dy
\lesssim(1+|k|)^{-E},
\end{equation*}
where $E$ is any number that we like.
Thus, there exists a positive constant $C$ such that
$C(1+|k|)^Eg_k$ is an $(L,N)$-atom on $Q_{0,k}$
and, for any $R\in\mathscr{Q}$,
\begin{equation*}
C(1+|k|)^E(g_k)_R
=:C\left[1+\frac{|x_R-x_P|}{\ell(R)}\right]^Eg^R_P
=:a^R_P
\end{equation*}
is an $(L,N)$-atom on $P:=R+k\ell(R)$.

Using the fact that $\theta,\eta\in\mathcal{S}$, we obtain,
for any $\alpha,\beta\in\mathbb{Z}_+^n$ and $x\in\mathbb{R}^n$,
\begin{equation}\label{converge}
\left| x^\alpha\partial^\beta g_k(x)\right|
=\left|\sum_{\gamma\leq\alpha}\binom{\alpha}{\gamma}
\int_{Q_{0,k}} (x-y)^{\gamma}\partial^\beta
\theta(x-y)y^{\alpha-\gamma}\eta(y)\,dy\right|
\lesssim(1+|k|)^{-M},
\end{equation}
where $M$ is any given positive integer.
Then, by \eqref{converge}, we conclude that the series $\psi=\sum_{k\in\mathbb Z^n}g_k$
converges in $\mathcal S$ and also, for any $R\in\mathscr{Q}$,
\begin{align*}
\psi_R
=\sum_{k\in\mathbb Z^n}(g_k)_R
=\sum_{P\in\mathscr{Q},\,\ell(P)=\ell(R)} g^R_P
=C^{-1}\sum_{P\in\mathscr{Q},\,\ell(P)=\ell(R)}
\left[1+\frac{|x_R-x_P|}{\ell(R)}\right]^{-E}a^R_P
\end{align*}
in $\mathcal S$. From this and the assumption that
$T\in\mathcal L(\mathcal S,\mathcal S')$, it follows that
\begin{equation*}
T\psi_R
=C^{-1}\sum_{P\in\mathscr{Q},\,\ell(P)=\ell(R)}
\left[1+\frac{|x_R-x_P|}{\ell(R)}\right]^{-E}T\left(a^R_P\right)
\end{equation*}
in $\mathcal S'$ and hence, for each $Q,R\in\mathscr{Q}$,
\begin{equation}\label{new add 9}
\langle\varphi_Q,T\psi_R\rangle
=\sum_{P\in\mathscr{Q}}\left\langle\varphi_Q,T\left(a^R_P\right)\right\rangle b_{P,R},
\end{equation}
where, for any $P,R\in\mathscr{Q}$,
$$
b_{P,R}:=
\left\{\begin{aligned}
&\left[1+\frac{|x_R-x_P|}{\ell(R)}\right]^{-E}&&\text{if }\ell(P)=\ell(R),\\
&0&&\text{otherwise}.
\end{aligned}\right.
$$
By the definition of $\dot a^{s,\tau}_{p,q}(d)$-almost diagonal matrices,
it is evident that $\{b_{P,R}\}_{P,R\in\mathscr{Q}}$
is $\dot a^{s,\tau}_{p,q}(d)$-almost diagonal, as soon as $E>\widetilde{J}$,
where $\widetilde{J}$ is the same as in \eqref{tauJ2}.
On the other hand, the assumption that $T$ maps $(L,N)$-atoms
to $\dot A^{s,\tau}_{p,q}(d)$-synthesis molecules,
the fact that each $a^R_P$ is such an atom on the cube $P$, and Corollary \ref{83} prove that,
for any $Q,P,R\in\mathscr{Q}$,
$|\langle\varphi_Q,T(a^R_P)\rangle|\leq a_{Q,P},$
where $\{a_{Q,P}\}_{Q,P\in\mathscr{Q}}$ is $\dot a^{s,\tau}_{p,q}(d)$-almost diagonal.
Using this and \eqref{new add 9}, we conclude that, for any $Q,R\in\mathscr{Q}$,
\begin{equation*}
|\langle\varphi_Q,T\psi_R\rangle|\leq\sum_{P\in\mathscr{Q}}a_{Q,P}b_{P,R}.
\end{equation*}
Applying this and \cite[Corollary 9.6]{bhyy1} on the composition of
$\dot a^{s,\tau}_{p,q}(d)$-almost diagonal matrices, we find that
$\{\langle\varphi_Q,T\psi_R\rangle\}_{Q,R\in\mathscr{Q}}$
is $\dot a^{s,\tau}_{p,q}(d)$-almost diagonal.
Then part \eqref{ext1} is a direct application of Corollary \ref{That3}.

It remains to show part \eqref{ext2}.
We recall from Corollary \ref{That3} and Lemma \ref{That2} that
the extension $\widetilde T$ is constructed by setting
$\widetilde T:=T_{\psi}\circ\widehat T\circ S_{\varphi},$
where $\widehat T:=\{\langle\varphi_Q,T\psi_R\rangle\}_{Q,R\in\mathscr{Q}}$.

Let $f\in L^2=\dot A^{0}_{2,2}$.
From \cite[Theorem 3.29]{bhyy1}, we infer that
$S_{\varphi}:\ \dot A^0_{2,2}\to\dot a^0_{2,2}=\ell^2$
and $T_\psi:\ \dot a^0_{2,2}\to\dot A^0_{2,2}$ are bounded, and hence
$f=\sum_{R\in\mathscr{Q}}\psi_R\langle \varphi_R,f\rangle$
in $L^2$. By $T_2\in\mathcal L(L^2)$, we conclude that
$T_2f=\sum_{R\in\mathscr{Q}}T_2\psi_R\langle \varphi_R,f\rangle$
in $L^2$ and hence
\begin{equation*}
\langle\varphi_Q,T_2f\rangle
=\sum_{R\in\mathscr{Q}}\langle\varphi_Q,T_2\psi_R\rangle\langle\varphi_R,f\rangle.
\end{equation*}
Since $T_2$ agrees with $T$ on $\mathcal S$,
it follows that, for any $f\in\mathcal S$,
\begin{equation*}
\langle\varphi_Q,Tf\rangle
=\sum_{R\in\mathscr{Q}}\langle\varphi_Q,T\psi_R\rangle\langle\varphi_R,f\rangle
\end{equation*}
or in other words that
\begin{equation}\label{new add 10}
S_\varphi\circ T=\widehat T\circ S_{\varphi}
\end{equation}
as operators from $\mathcal S$ to $\ell^2=\dot a^0_{2,2}$.
Applying $T_\psi\in\mathcal L(\ell^2,L^2)$ to both sides of \eqref{new add 10},
we obtain, for any $f\in\mathcal S$,
\begin{equation*}
Tf=T_{\psi}\circ S_{\varphi}\circ Tf=T_{\psi}\circ\widehat T\circ S_{\varphi}f
\end{equation*}
in $L^2\subset\mathcal{S}_\infty'$,
and hence, in particular, for any $f\in(\mathcal S)^m\cap\dot A^{s,\tau}_{p,q}(W)$,
$Tf=\widetilde Tf.$
This finishes the proof of Proposition \ref{ext}.
\end{proof}

\begin{remark}
It seems that, in a lot of the related literature,
the same situation as in part \eqref{ext1} of Proposition \ref{ext}
is (a bit imprecisely) referred to as saying that
``the operator $T$ admits a continuous extension to all $\dot A^{s,\tau}_{p,q}(W)$.''
Strictly speaking, as Torres said in \cite[p.\,21]{tor},
if no extra information on $T$ is provided,
we should say that ``the restriction of $T$ to $\mathcal{S}_\infty$
admits a continuous extension to all $\dot A^{s,\tau}_{p,q}(W)$.''
This type of statement involving $\mathcal{S}_\infty$ is typical
when dealing with spaces of distributions modulo polynomials.
We refer to \cite[pp.\,19--22]{tor} for further discussions
on the problem of extending $T$, once it is known that it maps atoms to molecules.

This issue does not arise in Corollary \ref{That3},
where we are dealing with an operator $T$ only defined on $\mathcal{S}_\infty$ to begin with.
One may then ask, why do we insist in
$T\in\mathcal L(\mathcal S,\mathcal S')$ in Proposition \ref{ext}.
One reason is that this is somewhat necessary in order to even make sense
of the condition that ``$T$ maps atoms to molecules'';
that is, a compactly supported atom can never be in $\mathcal{S}_\infty$
(in particular, be orthogonal to all polynomials), unless it is identically zero,
and hence it is necessary that the action of $T$ is
somehow defined also on functions outside $\mathcal{S}_\infty$.
Even if, by Corollary \ref{That3}, we would only need to estimate
the pairings $\langle\varphi_Q,T\psi_R\rangle$ involving $\varphi_Q,\psi_R\in\mathcal{S}_\infty$,
in practice we need to know about the action of $T$
on other kinds of functions as well, to make these estimates!
\end{remark}

\subsection{Definition and Consequences of Calder\'on--Zygmund Conditions}\label{ss CZO}

We now turn to the actual discussion of Calder\'on--Zygmund operators.
The following classical definitions are standard.
Let $\mathcal D :=C_{\mathrm{c}}^\infty $ with the usual inductive limit topology.
We denote by $\mathcal D' $ the space of
all continuous linear functionals on $\mathcal D$,
equipped with the weak-$*$ topology.
If $T\in\mathcal L(\mathcal S,\mathcal S')$,
then, by the well-known Schwartz kernel theorem,
we find that there exists $\mathcal K\in\mathcal S'(\mathbb R^n\times\mathbb R^n)$
such that, for each $\varphi,\psi\in\mathcal S$,
$\langle T\varphi,\psi\rangle
=\langle\mathcal K,\varphi\otimes\psi\rangle.$
This $\mathcal K$ is called the \emph{Schwartz kernel} of $T$.

\begin{definition}\label{WBP}
Let $T\in\mathcal L(\mathcal S,\mathcal S')$,
and let $\mathcal K\in\mathcal S'(\mathbb R^n\times\mathbb R^n)$ be its Schwartz kernel.
\begin{enumerate}[\rm(i)]
\item We say that $T$ satisfies the \emph{weak boundedness property}
and write $T\in\operatorname{WBP}$ if,
for any bounded subset $\mathcal B$ of $\mathcal D$,
there exists a positive constant $C=C(\mathcal B)$ such that,
for any $\varphi,\eta\in\mathcal B$, $h\in\mathbb R^n$, and $r\in(0,\infty)$,
\begin{equation*}
\left|\left\langle T \left[\varphi\left(\frac{\cdot-h}{r}\right)\right],
\eta \left(\frac{\cdot-h}{r}\right)\right\rangle\right|\leq Cr^n.
\end{equation*}
\item For any $\ell\in(0,\infty)$, we say that $T$ has a \emph{Calder\'on--Zygmund kernel} of order $\ell$
and write $T\in\operatorname{CZO}(\ell)$
if the restriction of $\mathcal K$ on the set $\{(x,y)\in\mathbb R^n\times\mathbb R^n:\ x\neq y\}$
is a continuous function with continuous partial derivatives
in the $x$ variable up to order $\lfloor\!\lfloor\ell\rfloor\!\rfloor$ satisfying that
there exists a positive constant $C$ such that,
for any $\gamma\in\mathbb{Z}_+^n$ with $|\gamma|\leq\lfloor\!\lfloor\ell\rfloor\!\rfloor$
and for any $x,y\in\mathbb{R}^n$ with $x\neq y$,
$|\partial_x^{\gamma}\mathcal K(x,y)|
\leq C|x-y|^{-n-|\gamma|}$
and, for any $\gamma\in\mathbb{Z}_+^n$ with $|\gamma|=\lfloor\!\lfloor\ell\rfloor\!\rfloor$
and for any $x,y,h\in\mathbb{R}^n$ with $|h|<\frac12|x-y|$,
\begin{equation*}
\left|\partial_x^{\gamma}\mathcal K(x,y)-\partial_x^\gamma\mathcal K(x+h,y)\right|
\leq C|x-y|^{-n-\ell}|h|^{\ell^{**}},
\end{equation*}
where $\lfloor\!\lfloor\ell\rfloor\!\rfloor $ and $\ell^{**}$
are the same as, respectively, in \eqref{ceil} and \eqref{r**}.
For any $\ell\in(-\infty,0]$, we interpret $T\in\operatorname{CZO}(\ell)$ as a void condition.
\end{enumerate}
\end{definition}

\begin{remark}
For any $\ell_1,\ell_2\in\mathbb{R}$ with $\ell_1<\ell_2$,
it is easy to prove that $\operatorname{CZO}(\ell_2)\subset\operatorname{CZO}(\ell_1)$.
\end{remark}

Extra conditions that we like to impose on Calder\'on--Zygmund operators
require the extension of their action beyond the initial domain $\mathcal S$ of definition.
For this purpose, we quote the following result
which is a special case of \cite[Lemma 2.2.12]{tor}
(see also the comment right after the proof of the said lemma).

\begin{lemma}\label{tor2.2.12}
Let $\ell\in(0,\infty)$ and $T\in\operatorname{CZO}(\ell)$,
and let $\{\phi_j\}_{j\in\mathbb{N}}\subset\mathcal D$ be a sequence of functions such that
\begin{enumerate}[\rm(i)]
\item for each compact set $K\subset\mathbb R^n$, there exists $j_K\in\mathbb{N}$
such that $\phi_j(x)=1$ for any $x\in K$ and $j\geq j_K$;
\item $\sup_{j\in\mathbb{N}}\|\phi_j\|_{L^\infty}<\infty $.
\end{enumerate}
Then the limit
\begin{equation}\label{Tphijf}
\langle T(f),g\rangle:=\lim_{j\to\infty}\langle T(\phi_jf),g\rangle
\end{equation}
exists for any $f\in\mathcal O^{\lfloor\!\lfloor\ell\rfloor\!\rfloor}$
and $g\in\mathcal D_{\lfloor\!\lfloor\ell\rfloor\!\rfloor}$, where
\begin{align*}
\mathcal O^{\lfloor\!\lfloor\ell\rfloor\!\rfloor}
&:=\left\{f\in C^\infty:\
\text{there exists a positive constant }C\text{ such that }\right.\\
&\qquad\qquad\quad\left.|f(x)|\leq C(1+|x|)^{\lfloor\!\lfloor\ell\rfloor\!\rfloor}
\text{ for any }x\in\mathbb{R}^n\right\},
\end{align*}
\begin{equation*}
\mathcal D_{\lfloor\!\lfloor\ell\rfloor\!\rfloor}
:=\left\{g\in\mathcal D:\
\int_{\mathbb R^n}x^\gamma g(x)\,dx=0\text{ if }
\gamma\in\mathbb{Z}_+^n\text{ with }|\gamma|\leq\lfloor\!\lfloor\ell\rfloor\!\rfloor \right\},
\end{equation*}
and $\lfloor\!\lfloor\ell\rfloor\!\rfloor $ is the same as in \eqref{ceil}.
Moreover, the limit \eqref{Tphijf} is independent of the choice of the sequence $\{\phi_j\}_{j\in\mathbb{N}}$.
\end{lemma}

Based on Lemma \ref{tor2.2.12}, we can give the following definition.

\begin{definition}\label{10.10}
Let $\ell\in(0,\infty)$. For any $T\in\operatorname{CZO}(\ell)$ and $f\in\mathcal O^{\lfloor\!\lfloor\ell\rfloor\!\rfloor}$,
where $\lfloor\!\lfloor\ell\rfloor\!\rfloor$ is the same as in \eqref{ceil},
we define $Tf:\ \mathcal D_{\lfloor\!\lfloor\ell\rfloor\!\rfloor}\to\mathbb{C}$
to be the same as in \eqref{Tphijf}.
In particular, for any $f(y):=y^\gamma$ with $y\in\mathbb R^n$,
where $\gamma\in\mathbb{Z}_+^n$ satisfies $|\gamma|\leq\lfloor\!\lfloor\ell\rfloor\!\rfloor$,
we define $T(y^\gamma)$ in this way.
\end{definition}

As it turns out, providing Calder\'on--Zygmund type conditions
that are sufficient to give boundedness on the full scale of the spaces $\dot A^{s,\tau}_{p,q}(W)$
(or even just $\dot F^s_{p,q}$) requires kernel estimates
that are somewhat more complicated than just those obtained by
combining conditions of the type $T\in\operatorname{CZO}(\ell)$ and $T^*\in\operatorname{CZO}(\ell^*)$.
The following conditions are inspired by,
but not exactly the same as, those used in \cite{ftw88,tor}.

\begin{definition}\label{CZK}
Let $ E,F\in\mathbb{R}$, $T\in\mathcal L(\mathcal S,\mathcal S')$,
and $\mathcal K\in\mathcal S'(\mathbb R^n\times\mathbb R^n)$ be its Schwartz kernel.
We say that $T\in\operatorname{CZK}^0(E;F)$
if the restriction of $\mathcal K$ to
$\{(x,y)\in\mathbb R^n\times\mathbb R^n:\ x\neq y\}$
is a continuous function such that all derivatives below
exist as continuous functions and there exists a positive constant $C$ such that,
for any $\alpha\in\mathbb{Z}_+^n$ with $|\alpha|\leq\lfloor\!\lfloor E\rfloor\!\rfloor_+$
and for any $x,y\in\mathbb R^n$ with $x\neq y$,
\begin{equation}\label{CZK0}
|\partial_x^\alpha\mathcal K(x,y)|
\leq C|x-y|^{-n-|\alpha|},
\end{equation}
for any $\alpha\in\mathbb{Z}_+^n$ with $|\alpha|=\lfloor\!\lfloor E\rfloor\!\rfloor$
and for any $x,y,u\in\mathbb R^n$ with $|u|<\frac12|x-y|$,
\begin{equation}\label{CZKx}
|\partial_x^\alpha\mathcal K(x,y)-\partial_x^\alpha\mathcal K(x+u,y)|
\leq C|u|^{E^{**}}|x-y|^{-n-E}
\end{equation}
and, for any $\alpha,\beta\in\mathbb{Z}_+^n$ with $|\alpha|\leq\lfloor\!\lfloor E\rfloor\!\rfloor_+$
and $|\beta|=\lfloor\!\lfloor F-|\alpha|\rfloor\!\rfloor$
and for any $x,y,v\in\mathbb R^n$ with $|v|<\frac12|x-y|$,
\begin{equation}\label{CZKy}
\begin{aligned}
\left|\partial_x^\alpha\partial_y^\beta\mathcal K(x,y)
-\partial_x^\alpha\partial_y^\beta\mathcal K(x,y+v)\right|
\leq C|v|^{(F-|\alpha|)^{**}}|x-y|^{-n-|\alpha|-(F-|\alpha|)},
\end{aligned}
\end{equation}
where $\lfloor\!\lfloor E\rfloor\!\rfloor$ and $E^{**}$
are the same as, respectively, in \eqref{ceil} and \eqref{r**}.

We say that $T\in\operatorname{CZK}^1(E;F)$
if $T\in\operatorname{CZK}^0(E;F)$ and, in addition,
there exists a positive constant $C$ such that,
for any $\alpha,\beta\in\mathbb{Z}_+^n$ with
$|\alpha|=\lfloor\!\lfloor E\rfloor\!\rfloor$ and $|\beta|=\lfloor\!\lfloor F-E\rfloor\!\rfloor$
and for any $x,y,u,v\in\mathbb R^n$ with $|u|+|v|<\frac12|x-y|$,
\begin{align}\label{CZKxy}
&\left|\partial_x^\alpha\partial_y^\beta\mathcal K(x,y)
-\partial_x^\alpha\partial_y^\beta\mathcal K(x+u,y)
-\partial_x^\alpha\partial_y^\beta\mathcal K(x,y+v)
+\partial_x^\alpha\partial_y^\beta\mathcal K(x+u,y+v)\right|\\
&\quad\leq C|u|^{E^{**}}|v|^{(F-E)^{**}}
|x-y|^{-n-E-(F-E)}. \notag
\end{align}
We may write just $\operatorname{CZK}(E;F)$
if the parameter values are such that \eqref{CZKxy} is void
and hence $\operatorname{CZK}^0(E;F)$ and $\operatorname{CZK}^1(E;F)$ coincide.
\end{definition}

Let us note that the two upper bounds for $|\alpha|$ in both \eqref{CZK0} and \eqref{CZKx}
have the positive part $\lfloor\!\lfloor E\rfloor\!\rfloor_+$,
so that the case $\alpha=\mathbf{0}$ is always included here.
In all other cases, some conditions may become void for some parameter values.
In particular, the condition \eqref{CZKxy} is void unless $F>E>0$.
It might be worth stressing that these conditions are not symmetric
with respect to the $x$ and the $y$ variables,
which one might expect on the basis of the classical theory in $L^2$,
where the boundedness of an operator and its adjoint are equivalent.
This is no longer the case in the much more general spaces that we consider
and hence it is natural that the two variables are playing somewhat different roles.

\begin{lemma}\label{CZKaux}
Let $\varepsilon\in(0,\infty)$, $N\in\mathbb N$, and
$f\in C^N(\mathbb R^n\setminus\{\mathbf{0}\})$
satisfy that there exists a positive constant $C$ such that,
for any $\beta\in\mathbb{Z}_+^n$ with $|\beta|=N$
and for any $y,v\in\mathbb{R}^n\setminus\{\mathbf{0}\}$ with $|y|\geq2|v|$,
\begin{equation}\label{456x}
|f(y)|\leq C|y|^{-n}
\text{ and }
\left|\partial^\beta f(y)-\partial^\beta f(y+v)\right|
\leq C|v|^{\varepsilon}|y|^{-n-N-\varepsilon}.
\end{equation}
Then there exists a positive constant $\widetilde{C}$ such that,
for any $\beta\in\mathbb{Z}_+^n$ and $y,v\in\mathbb{R}^n\setminus\{\mathbf{0}\}$ with $|y|\geq2|v|$, when $|\beta|\leq N$,
$|\partial^\beta f(y)|
\leq\widetilde{C}|y|^{-n-|\beta|}$
and, when $|\beta|<N$,
\begin{equation*}
\left|\partial^\beta f(y)-\partial^\beta f(y+v)\right|
\leq\widetilde{C}|v|\,|y|^{-n-|\beta|-1}.
\end{equation*}
\end{lemma}

\begin{proof}
We first note that the claim concerning the differences
is an easy consequence of the derivative bound.
Indeed, for any $\beta\in\mathbb{Z}_+^n$ with $|\beta|<N$
and for any $y,v\in\mathbb{R}^n\setminus\{\mathbf{0}\}$ with $|y|\geq2|v|$,
by the mean value theorem, we obtain
\begin{align*}
\left|\partial^\beta f(y)-\partial^\beta f(y+v)\right|
=\left|\int_0^1v\nabla\partial^\beta f(y+tv)\,dt\right|
\lesssim|v|\int_0^1|y+tv|^{-n-|\beta|-1}\,dt
\lesssim|v|\,|y|^{-n-|\beta|-1}.
\end{align*}
This is the desired estimate.

For the derivative bound, let us first consider $n\geq2$.
We first claim that, for any $\beta\in\mathbb{Z}_+^n$ with $|\beta|=N$ and
for any $z,y\in\mathbb{R}^n\setminus\{\mathbf{0}\}$ with $|z|\geq|y|$,
\begin{equation}\label{CZKaux1}
\left|\partial^\beta f(y)-\partial^\beta f(z)\right|
\lesssim|y|^{-n-N}.
\end{equation}
To show this, suppose first that $|y|=|z|$.
We restrict our consideration to a plane that contains $\mathbf{0}$, $y$, and $z$
and identify this plane with $\mathbb C$;
this is just for the convenience of the notation with polar coordinates.
After a rotation, we may further assume that $y=r\in(0,\infty)$ and $z=re^{i\phi}$,
where $\phi\in(-\pi,\pi]$. Then, by \eqref{456x},
for any $\beta\in\mathbb{Z}_+^n$ with $|\beta|=N$, we have
\begin{align*}
\left|\partial^\beta f(y)-\partial^\beta f(z)\right|
\leq\sum_{k=1}^7\left|\partial^\beta f\left(re^{i\phi\frac{k}{7}}\right)
-\partial^\beta f\left( r e^{i\phi\frac{k-1}{7}}\right)\right|
\lesssim\sum_{k=1}^7\left(\frac{r}{2}\right)^{\varepsilon}r^{-n-N-\varepsilon}
\sim|y|^{-n-N},
\end{align*}
where we used the assumption with $r e^{i\phi\frac{k-1}{7}}$ in place of $y$ and
$$
v
=r\left(e^{i\phi\frac{k}{7}}-e^{i\phi\frac{k-1}{7}}\right)
=re^{i\phi\frac{k-1}{7}}\left(e^{i\frac{\phi}{7}}-1\right)
$$
with $|v|\leq r\frac{|\phi|}{7}\leq r\frac{\pi}{7}<\frac{r}{2}=\frac{|y|}{2}$.

Suppose then that $z=\rho y$ with $\rho\in(1,\infty)$.
Then, choosing $K\in\mathbb Z$ with $\frac{\rho}{2}<2^K\leq\rho$, by \eqref{456x}, we obtain
\begin{align*}
\left|\partial^\beta f(y)-\partial^\beta f(z)\right|
&\leq\sum_{k=1}^K\left|\partial^\beta f\left(2^{k-1}y\right)
-\partial^\beta f\left(2^ky\right)\right|
+\left|\partial^\beta f\left(2^Ky\right)-\partial^\beta f(z)\right|\\
&\lesssim\sum_{k=1}^K\left|2^{k-1}y\right|^{\varepsilon}
\left|2^ky\right|^{-n-N-\varepsilon}
+\left|\left(2^K-\rho\right)y\right|^{\varepsilon}
|\rho y|^{-n-N-\varepsilon}\\
&\lesssim\left[\sum_{k=1}^K2^{-k(n+N)}+\rho^{-n-N}\right]|y|^{-n-N}
\lesssim|y|^{-n-N}.
\end{align*}
A combination of the above two cases proves the general situation,
by choosing an auxiliary point $u$ such that $|u|=|y|$ and $z=\rho u$ with $\rho>1$.
This finishes the proof of \eqref{CZKaux1}.

From \eqref{CZKaux1}, it follows that,
for any $y\in\mathbb{R}^n\setminus\{\mathbf{0}\}$,
the limit $v_\beta:=\lim_{|y|\to\infty}\partial^\beta f(y)$ exists and
\begin{equation*}
\left|\partial^\beta f(y)-v_{\beta}\right|
\lesssim|y|^{-n-N}.
\end{equation*}
We next show by backwards induction that,
for any $k\in\{0,1,\ldots,N\}$ and $\gamma\in\mathbb{Z}_+^n$ with $|\gamma|=k$,
there exists a polynomial
\begin{equation}\label{CZKaux2}
p_\gamma(y):=\sum_{\alpha\in\mathbb{Z}_+^n,\,|\alpha|\leq N-k}
v_{\gamma+\alpha}\frac{y^\alpha}{\alpha!}
\end{equation}
such that, for any $y\in\mathbb{R}^n\setminus\{\mathbf{0}\}$,
\begin{equation}\label{CZKaux3}
\left|\partial^\gamma f(y)-p_\gamma(y)\right|
\lesssim|y|^{-n-k};
\end{equation}
moreover, it is a part of the assertion that the coefficient $v_{\gamma+\alpha}$
only depends on the sum $\gamma+\alpha$ as indicated by the notation.
We have already established the base case $k=N$ of the induction.

Before proceeding to the induction step,
we observe a key property of polynomials of the form \eqref{CZKaux2}.
For any $i\in\{1,\ldots,n\}$, let
$e_i:=(0,\ldots,0,1,0,\ldots,0)\in\mathbb{Z}_+^n$
where only the $i$-th component is $1$.
Then, for any $k\in\{0,1,\ldots,N-1\}$ and $\gamma\in\mathbb{Z}_+^n$ with $|\gamma|=k$,
\begin{align*}
\partial_ip_\gamma
=\sum_{\alpha\in\mathbb{Z}_+^n,\,|\alpha|\leq N-k,\,\alpha_i>0} v_{\gamma+\alpha}\frac{y^{\alpha-e_i}}{(\alpha-e_i)!}
=\sum_{\beta\in\mathbb{Z}_+^n,\,|\beta|\leq N-k-1} v_{\gamma+e_i+\beta}\frac{y^{\beta}}{\beta!}=p_{\gamma+e_i}.
\end{align*}

Let us now assume that our induction hypothesis is valid for any $\gamma\in\mathbb{Z}_+^n$ with $|\gamma|=k+1$,
and we prove it for any $\gamma\in\mathbb{Z}_+^n$ with $|\gamma|=k$. For any $|y|\leq|z|$,
we choose a path $\Gamma$ from $y$ to $z$ such that (for example)
$\Gamma$ first travels at constant distance $|y|$ from the origin to a point $u$
such that $z=\rho u$ with $\rho\geq1$. Then
\begin{align}\label{CZKaux4}
\partial^\gamma f(z)-\partial^\gamma f(y)
&=\int_{\Gamma}\nabla\partial^\gamma f(x)\cdot d\vec{x}
=\sum_{i=1}^n\int_{\Gamma}\partial^{\gamma+e_i}f(x)\,dx_i\\
&=\sum_{i=1}^n\int_{\Gamma}p_{\gamma+e_i}(x)\,dx_i
+\sum_{i=1}^n\int_{\Gamma}\left[\partial^{\gamma+e_i}f(x)-p_{\gamma+e_i}(x)\right]\,dx_i\notag\\
&=\sum_{i=1}^n\int_{\Gamma}\partial_i p_{\gamma}(x)\,dx_i
+\sum_{i=1}^n\int_{\Gamma} O \left( |x|^{-n-k-1}\right)\,dx_i\notag\\
&=\int_{\Gamma}\nabla p_{\gamma}(x)\cdot d\vec{x}
+O\left(|y|\,|y|^{-n-k-1}\right)
+\int_{|y|}^{|z|}O\left(r^{-n-k-1}\right)\, dr\notag\\
&=p_{\gamma}(z)-p_{\gamma}(y)
+O\left(|y|^{-n-k}\right),\notag
\end{align}
where we used the induction hypothesis to estimate the second term;
note that the length of the path at constant distance $|y|$
from the origin is at most $\pi|y|$. From here we deduce that,
for any $y,z\in\mathbb{R}^n\setminus\{\mathbf{0}\}$ with $|y|\leq|z|$,
\begin{equation*}
\left|\left[\partial^\gamma f(z)-p_{\gamma}(z)\right]
-\left[\partial^\gamma f(y)-p_{\gamma}(y)\right]\right|
\lesssim|y|^{-n-k},
\end{equation*}
and hence $\partial^\gamma f(y)-p_{\gamma}(y)$ converges to a limit as $|y|\to\infty$.
Observing that the constant term $v_{\gamma}$ of $p_\gamma$
is irrelevant for the computation \eqref{CZKaux4},
we can choose this constant term so that this limit is zero, which shows \eqref{CZKaux3}.

By induction, we have now proved \eqref{CZKaux3} for any $k\in\{0,1,\ldots,N\}$
and in particular for $k=0$. However, by assumption,
we also know that $|f(y)|\lesssim|y|^{-n}$, and hence
\begin{equation*}
|p_0(y)|\leq|p_0(y)-f(y)|+|f(y)|\lesssim|y|^{-n}.
\end{equation*}
The only way that a polynomial can satisfy such an estimate is that the zero polynomial
and hence all its coefficients must vanish identically.
Since the coefficients of each $p_\gamma$ are a subset of the coefficients of $p_0$,
we find that all these polynomials vanish.
But then the bound \eqref{CZKaux3} reduces to the assertion of the present lemma.
This finishes the proof of the first claim when $n\geq2$.

The key difference for $n=1$ is that in this case
it is not possible to travel around the origin at a constant distance from the origin.
Instead, one needs to run the previous argument separately on the positive and the negative half-lines,
constructing possibly different polynomials on each of them.
However, since the polynomials are eventually seen to vanish identically,
this has no impact on the final conclusions.
This finishes the proof of Lemma \ref{CZKaux}.
\end{proof}

\begin{remark}\label{CZK cases}
Let $E,F\in\mathbb{R}$, $T\in\mathcal L(\mathcal S,\mathcal S')$,
and $\mathcal K\in\mathcal S'(\mathbb R^n\times\mathbb R^n)$ be its Schwartz kernel.
We observe the following prominent special cases.
\begin{enumerate}[\rm(i)]
\item\label{T CZO} When $F\leq 0<E$,
the condition $T\in\operatorname{CZK}(E;F)=\operatorname{CZK}(E;0)$
consists of just \eqref{CZK0} and \eqref{CZKx}.
This is conventionally denoted by $T\in\operatorname{CZO}(E)$,
meaning Calder\'on--Zygmund estimates of order $E$ in the $x$ variable
and no assumptions in the $y$ variable.

\item\label{T* CZO} When $E\leq 0<F$,
the condition $T\in\operatorname{CZK}(E;F)=\operatorname{CZK}(0;F)$
consists of \eqref{CZK0} and \eqref{CZKy} for $\alpha=0$, that is,
for any $x,y\in\mathbb R^n$ with $x\neq y$,
\begin{equation}\label{CZK00}
|\mathcal K(x,y)|\lesssim|x-y|^{-n}
\end{equation}
and, for any $\beta\in\mathbb{Z}_+^n$ with $|\beta|=\lfloor\!\lfloor F\rfloor\!\rfloor$
and for any $x,y,v\in\mathbb R^n$ with $|v|<\frac12|x-y|$,
\begin{equation}\label{CZKy0}
\left|\partial_y^\beta\mathcal K(x,y)-\partial_y^\beta\mathcal K(x,y+v)\right|
\lesssim|v|^{F^{**}}|x-y|^{-n-F},
\end{equation}
where $\lfloor\!\lfloor F\rfloor\!\rfloor $ and $F^{**}$
are the same as, respectively, in \eqref{ceil} and \eqref{r**}.
By Lemma \ref{CZKaux} [applied to $f(y):=K(x,x+y)$],
these imply the ``intermediate'' smoothness estimates:
for any $\beta\in\mathbb{Z}_+^n$ with $|\beta|\leq\lfloor\!\lfloor F\rfloor\!\rfloor$
and for any $x,y\in\mathbb{R}^n$ with $x\neq y$,
\begin{equation}\label{CZKy1}
\left|\partial_y^\beta \mathcal K(x,y)\right|
\lesssim|x-y|^{-n-|\beta|}.
\end{equation}
The totality of conditions \eqref{CZK00}, \eqref{CZKy0}, and \eqref{CZKy1}
[which is equivalent to just \eqref{CZK00} and \eqref{CZKy0} by Lemma \ref{CZKaux}]
is conventionally denoted by $T^*\in\operatorname{CZO}(F)$,
meaning the Calder\'on--Zygmund estimates of order $F$ in the $y$ variable
and no assumptions in the $x$ variable.

\item\label{factor1} When $0<F<E\wedge 1$,
the condition $T\in\operatorname{CZK}(E;F)$
consists of \eqref{CZK0}, \eqref{CZKx}, and \eqref{CZKy0} (where $\beta=0$)
and hence $T\in\operatorname{CZK}(E;F)$ ``factorises'' into
the classical conditions $T\in\operatorname{CZO}(E)$
and $T^*\in\operatorname{CZO}(F)$ in this case.
\item\label{factor2} When $0<E<F\wedge1$,
the condition $ T\in\operatorname{CZK}^0(E;F)$ consists of
\eqref{CZK0}, \eqref{CZKx} (both with $\alpha=0$), and \eqref{CZKy0}
and hence $T\in\operatorname{CZK}^0(E;F)$ again ``factorises'' into
the conditions $T\in\operatorname{CZO}(E)$ and $T^*\in\operatorname{CZO}(F)$.
However, for these parameter values, this would not be the case with
$T\in\operatorname{CZK}^1(E;F)$, and this is the main reason
for introducing the distinction between the two conditions.
\end{enumerate}

By \eqref{T CZO} and \eqref{T* CZO} of Remark \ref{CZK cases},
the condition $T\in\operatorname{CZK}^0(E;F)$ implies that
\begin{equation}\label{CZK and CZO}
T\in\operatorname{CZO}(E)
\text{ and }
T^*\in\operatorname{CZO}(F).
\end{equation}
\end{remark}

We define the (usual) Taylor expansion of order $ N\in\mathbb{Z}$
of a function $f$ around a point $z$ as, for any $x\in\mathbb{R}^n$,
\begin{equation*}
\operatorname{Tayl}_z^N f(x)
:=\sum_{\alpha\in\mathbb{Z}_+^n,\,|\alpha|\leq N}\frac{(x-z)^\alpha}{\alpha!}\partial^\alpha f(z)
\end{equation*}
with the interpretation that an empty sum is zero;
hence $\operatorname{Tayl}_z^Nf(x)=0$ for $N<0$.
For convenience of reference,
we record the following standard bounds about the error of the Taylor approximation.

\begin{lemma}\label{Taylor}
Let $f$ have a Taylor expansion of order $N\in\mathbb{Z}_+$
around a point $z\in\mathbb{R}^n$ and let $s\in[0,1]$.
Then there exists a positive constant $C$,
depending only on $N$, such that, for any $y\in\mathbb{R}^n\setminus\{z\}$,
\begin{align}\label{10.16x}
\left|f(y)-\left(\operatorname{Tayl}_z^Nf\right)(y)\right|
&\leq C|y-z|^{N+s}\sup_{\alpha\in\mathbb{Z}_+^n,\,|\alpha|=N}\sup_{t\in(0,1)}
\frac{|\partial^\alpha f(z+t(y-z))-\partial^\alpha f(z)|}{|t(y-z)|^s}\\
&\leq C|y-z|^{N+1}\sup_{\beta\in\mathbb{Z}_+^n,\,|\beta|=N+1}\sup_{t\in(0,1)}
\left|\partial^\beta f(z+t(y-z))\right|\notag\\
&\leq C|y-z|^{N+1}\sup_{\beta\in\mathbb{Z}_+^n,\,|\beta|=N+1}\left\|\partial^\beta f \right\|_{L^\infty}.\notag
\end{align}
\end{lemma}

\begin{proof}
For any $y\in\mathbb{R}^n\setminus\{z\}$, from the multivariate Taylor formula
\begin{align*}
f(y)
&=\sum_{\alpha\in\mathbb{Z}_+^n,\,|\alpha|<N}\frac{(y-z)^\alpha}{\alpha!}\partial^\alpha f(z)\\
&\quad+\sum_{\alpha\in\mathbb{Z}_+^n,\,|\alpha|=N}\frac{(y-z)^\alpha}{\alpha!}
\int_0^1\partial^\alpha f(z+t(y-z))N(1-t)^{N-1}\,dt
\end{align*}
and the fact that $\int_0^1N(1-t)^{N-1}\,dt=1$, we infer that
\begin{align*}
&f(y)-\sum_{\alpha\in\mathbb{Z}_+^n,\,|\alpha|\leq N}\frac{(y-z)^\alpha}{\alpha!}\partial^\alpha f(z)\\
&\quad=\sum_{\alpha\in\mathbb{Z}_+^n,\,|\alpha|=N}\frac{(y-z)^\alpha}{\alpha!}
\int_0^1[\partial^\alpha f(z+t(y-z))-\partial^\alpha f(z)]N(1-t)^{N-1}\,dt,
\end{align*}
and hence
\begin{align*}
\left|f(y)-\operatorname{Tayl}_z^Nf(y)\right|
\leq\sum_{\alpha\in\mathbb{Z}_+^n,\,|\alpha|=N}\frac{|(y-z)^\alpha|}{\alpha!}
\int_0^1S|t(y-z)|^sN(1-t)^{N-1}\,dt
\lesssim S|y-z|^{N+s},
\end{align*}
where
$$
S:=\sup_{\alpha\in\mathbb{Z}_+^n,\,|\alpha|=N}\sup_{t\in(0,1)}
\frac{|\partial^\alpha f(z+t(y-z))-\partial^\alpha f(z)|}{|t(y-z)|^s}.
$$
This shows the first bound in \eqref{10.16x}.

For the second bound,
for any $i\in\{1,\ldots,n\}$, let
$e_i:=(0,\ldots,0,1,0,\ldots,0)\in\mathbb{Z}_+^n$
where only the $i$-th component is $1$.
Notice that, for any $t\in(0,1)$ and $v=t(y-z)$,
\begin{align*}
|v|^{-s}|\partial^\alpha f(z+v)-\partial^\alpha f(z)|
&=|v|^{-s}\left|\int_0^1v\cdot\nabla\partial^\alpha f(z+t'v)\,dt' \right| \\
&\leq|v|^{1-s}\sup_{i\in\{1,\ldots,n\}}\sup_{t'\in(0,1)}|\partial^{\alpha+e_i}f(z+t'v)| \\
&\leq|y-z|^{1-s}\sup_{\beta\in\mathbb{Z}_+^n,\,|\beta|=N+1}\sup_{t\in(0,1)}
\left|\partial^{\beta}f(z+t(y-z))\right|.
\end{align*}
Substituting this into the conclusion of the first bound gives the second bound in \eqref{10.16x}.
The third bound is trivial.
This finishes the proof of Lemma \ref{Taylor}.
\end{proof}

A useful connection between the Taylor error bounds
and Calder\'on--Zygmund kernel conditions is provided in the next two lemmas.
The role of the parameter $R$ may appear a bit arbitrary at this point, but this a form in which these technical bounds will be convenient to apply at a later point, where we choose $R$ to depend on the dimension $n$ only.

\begin{lemma}\label{CZK Taylor}
Let $E,F\in\mathbb{R}$, $\mathcal K\in\operatorname{CZK}^0(E;F)$, and $R\in(0,\infty)$.
Then there exists a positive constant $C$ such that,
for any $\alpha\in\mathbb{Z}_+^n$ with $|\alpha|\leq\lfloor\!\lfloor E\rfloor\!\rfloor_+$
and for any $x,y\in\mathbb{R}^n$ with
$0<|y|\leq R$ and $\max\{1,2|y|\}<|x|$,
\begin{equation*}
\left|\partial_x^\alpha\mathcal K(x,y)-
\left[\operatorname{Tayl}_{\mathbf{0}}^{\lfloor\!\lfloor F-|\alpha|\rfloor\!\rfloor}
\partial_x^\alpha\mathcal K(x,\cdot)\right](y)\right|
\leq C|x|^{-n-F},
\end{equation*}
where $\lfloor\!\lfloor E\rfloor\!\rfloor$ is the same as in \eqref{ceil}.
\end{lemma}

\begin{proof}
Let $\alpha\in\mathbb{Z}_+^n$ satisfy $|\alpha|\leq\lfloor\!\lfloor E\rfloor\!\rfloor_+$.
If $F\leq|\alpha|$, then $\operatorname{Tayl}_{\mathbf{0}}^{\lfloor\!\lfloor F-|\alpha|\rfloor\!\rfloor}=0$,
and we simply use \eqref{CZK0} from Definition \ref{CZK} to obtain,
for any $x,y\in\mathbb{R}^n$ with $\max\{1, 2|y|\}<|x|$,
\begin{equation*}
|\partial_x^\alpha\mathcal K(x,y)|
\lesssim|x-y|^{-n-|\alpha|}
\sim|x|^{-n-|\alpha|}
\leq|x|^{-n-F}.
\end{equation*}

If $F>|\alpha|$, we choose $N=\lfloor\!\lfloor F-|\alpha|\rfloor\!\rfloor\geq0$
and $s=(F-|\alpha|)^{**}\in(0,1]$ (where $(F-\alpha)^{**}$ is the same as in \eqref{r**}), so that $N+s=F-|\alpha|$.
We apply Lemma \ref{Taylor} with these $N$ and $s$
and then \eqref{CZKy} from Definition \ref{CZK} to conclude that,
for any $x,y\in\mathbb{R}^n$ with
$0<|y|\leq R$ and $2|y|<|x|$,
\begin{align*}
&\left|\partial_x^\alpha\mathcal K(x,y)-
\left[\operatorname{Tayl}_{\mathbf{0}}^{\lfloor\!\lfloor F-|\alpha|\rfloor\!\rfloor}
\partial_x^\alpha\mathcal K(x,\cdot)\right](y)\right|\\
&\quad\lesssim|y-\mathbf{0}|^{F-|\alpha|}
\sup_{\substack{\beta\in\mathbb{Z}_+^n,\,|\beta|=\lfloor\!\lfloor F-|\alpha|\rfloor\!\rfloor \\t\in(0,1)}}
\frac{|\partial_y^\beta\partial_x^\alpha\mathcal K(x,ty)
-\partial_y^\beta\partial_x^\alpha\mathcal K(x,\mathbf{0})|}{|ty|^{(F-|\alpha|)^{**}}}\\
&\quad\lesssim|y|^{F-|\alpha|}|x-\mathbf{0}|^{-n-|\alpha|-(F-|\alpha|)}
\lesssim|x|^{-n-F}.
\end{align*}
A combination of the above two cases proves Lemma \ref{CZK Taylor}.
\end{proof}

\begin{lemma}\label{CZK Taylor2}
Let $E,F\in\mathbb{R}$, $\mathcal K\in\operatorname{CZK}^1(E;F)$,
and $R\in(0,\infty)$. Then there exists a positive constant $C$ such that,
for any $\alpha\in\mathbb{Z}_+^n$ with $|\alpha|=\lfloor\!\lfloor E\rfloor\!\rfloor$
and for any $x,y,u\in\mathbb{R}^n$ with $0<|y|\leq R$, $\max\{1,2|y|\}<|x|$, and $|u|+|y|<\frac12|x|$,
one has
\begin{equation*}
\left|\left(I-\operatorname{Tayl}_{\mathbf{0}}^{\lfloor\!\lfloor F-E\rfloor\!\rfloor}\right)
[\partial_x^\alpha\mathcal K(x,\cdot)-\partial_x^\alpha\mathcal K(x+u,\cdot)](y)\right|
\leq C|u|^{E^{**}}|x|^{-n-F},
\end{equation*}
where $\lfloor\!\lfloor E\rfloor\!\rfloor$ and $E^{**}$
are the same as, respectively, in \eqref{ceil} and \eqref{r**}.
\end{lemma}

\begin{proof}
Let $\alpha\in\mathbb{Z}_+^n$ satisfy $|\alpha|=\lfloor\!\lfloor E\rfloor\!\rfloor$.
If $E\leq 0$, the claim is void, so we assume that $E>0$.

If $F\leq E$, then $\operatorname{Tayl}_{\mathbf{0}}^{\lfloor\!\lfloor F-E\rfloor\!\rfloor}=0$,
and we simply use \eqref{CZKx} from Definition \ref{CZK} to obtain,
for any $x,y,u\in\mathbb{R}^n$ with $\max\{1,2|y|\}<|x|$ and $|u|<\frac12|x-y|$,
\begin{align*}
|\partial_x^\alpha\mathcal K(x,y)-\partial_x^\alpha\mathcal K(x+u,y)|
\lesssim|u|^{E^{**}}|x-y|^{-n-E}
\sim|u|^{E^{**}}|x|^{-n-E}
\leq|u|^{E^{**}}|x|^{-n-F}.
\end{align*}
Notice that, for any $x,y,u\in\mathbb{R}^n$ with $|u|+|y|<\frac12|x|$,
$|u|<\frac12|x|-|y|\leq\frac12(|x|-|y|)\leq\frac12|x-y|.$
This finishes the proof of the present lemma in this case.

If $F>E$, we choose $N=\lfloor\!\lfloor F-E\rfloor\!\rfloor\geq0$
and $s=(F-E)^{**}\in(0,1]$ (where $(F-E)^{**}$ is the same as in \eqref{r**}),
so that $N+s=F-E$.
We apply Lemma \ref{Taylor} with these $N$ and $s$
and then \eqref{CZKxy} from Definition \ref{CZK}
to conclude that $\mathcal K\in\operatorname{CZK}(E;F)$. Thus,
for any $x,y,u\in\mathbb{R}^n $ with $0<|y|\leq R$, $2|y|<|x|$,
and $|u|+|y|<\frac12|x|$,
\begin{align*}
&\left|\left( I-\operatorname{Tayl}_{\mathbf{0}}^{\lfloor\!\lfloor F-E\rfloor\!\rfloor}\right)
[\partial_x^\alpha\mathcal K(x,\cdot)-\partial_x^\alpha \mathcal K(x+u,\cdot)](y)\right|\\
&\quad\lesssim|y-\mathbf{0}|^{F-E}
\sup_{\substack{\beta\in\mathbb{Z}_+^n,\,|\beta|=\lfloor\!\lfloor F-E\rfloor\!\rfloor\\ t\in(0,1)}}
\Bigg[\Big|\partial^\beta[\partial_x^\alpha\mathcal K(x,\cdot)
-\partial_x^\alpha\mathcal K(x+u,\cdot)](ty)\\
&\qquad-\partial^\beta[\partial_x^\alpha\mathcal K(x,\cdot)
-\partial_x^\alpha \mathcal K(x+u,\cdot)](\mathbf{0})\Big|\Bigg]
\,|ty|^{-(F-E)^{**}}\\
&\quad\lesssim|y|^{F-E} |u|^{E^{**}}|x-\mathbf{0}|^{-n-E-(F-E)}
\lesssim|u|^{E^{**}}|x|^{-n-F}.
\end{align*}
A combination of the two cases shows Lemma \ref{CZK Taylor2}.
\end{proof}

\subsection{$T(1)$ Theorem for $\dot A^{s,\tau}_{p,q}(W)$}\label{ss T1}

The goal of this section is to identify the minimal size,
smoothness, and cancellation conditions of Calder\'on--Zygmund type
that guarantee the mapping of sufficiently regular atoms
into molecules of described parameters and to use this information to
provide reasonably sharp criteria for the boundedness of Calder\'on--Zygmund operators
on $\dot A^{s,\tau}_{p,q}(W)$. We begin by defining a short-hand notation
that combines the various types of assumptions that we have discussed above.
For any $T\in\mathcal L(\mathcal S,\mathcal S')$,
let $T^*\in\mathcal L(\mathcal S,\mathcal S')$
be defined by setting, for any $\varphi,\psi\in\mathcal S$,
$\langle T\varphi,\psi\rangle=\langle T^*\psi,\varphi\rangle.$

\begin{definition}
Let $\sigma\in\{0, 1\}$ and $ E,F,G,H\in\mathbb{R}$.
We say that $T\in\operatorname{CZO}^\sigma(E,F,G,H)$
if $T\in\mathcal L(\mathcal S,\mathcal S')$
and its Schwartz kernel $\mathcal K\in\mathcal S'(\mathbb R^n\times\mathbb R^n)$ satisfy
\begin{enumerate}[\rm(i)]
\item $T\in\operatorname{WBP}$, where $\operatorname{WBP}$
is the same as in Definition \ref{WBP}(i);
\item $\mathcal K\in\operatorname{CZK}^\sigma(E;F)$;
\item $T(y^\gamma)=0$ for any $\gamma\in\mathbb{Z}_+^n$ with $|\gamma|\leq G$;
\item $T^*(x^\theta)=0$ for any $\theta\in\mathbb{Z}_+^n$ with $|\theta|\leq H$.
\end{enumerate}
\end{definition}

The following theorem is the main result on the boundedness of
Calder\'on--Zygmund operators on $\dot A^{s,\tau}_{p,q}(W)$.

\begin{theorem}\label{T1 BF}
Let $s\in\mathbb R$, $\tau\in[0,\infty)$, $p\in(0,\infty)$, $q\in(0,\infty]$, $d\in[0,n)$,
and $\widetilde J,\widetilde s$ be the same as in \eqref{tauJ2}.
Let $W\in A_p$ have the $A_p$-dimension $d$
and let $T\in\operatorname{CZO}^\sigma(E,F,G,H)$, where
$\sigma\in\{0,1\}$ and $E,F,G,H\in\mathbb{R}$ satisfy
\begin{align}\label{T1Amol1}
\sigma\geq\mathbf{1}_{[0,\infty)}(\widetilde s),\
E>(\widetilde s)_+,\
F>\widetilde J-n+(\widetilde s)_-,\
G\geq\lfloor\widetilde s\rfloor_+,\text{ and }
H\geq\left\lfloor\widetilde J-n-\widetilde s\right\rfloor.
\end{align}
Then the following two statements hold.
\begin{enumerate}[\rm(i)]
\item\label{T1test} There exists an operator $\widetilde T\in\mathcal L(\dot A^{s,\tau}_{p,q}(W))$
that agrees with $T$ on $\mathcal{S}_\infty$.
\item\label{T1ext} If, in addition, $H\geq0$ or, in other words, if
\begin{equation}\label{T1Amol2}
\sigma\geq\mathbf{1}_{[0,\infty)}(\widetilde s),\
E>(\widetilde s)_+,\
F>\widetilde J-n+(\widetilde s)_-,\
G\geq\lfloor \widetilde s\rfloor_+,\text{ and }
H\geq\left\lfloor\widetilde J-n-\widetilde s\right\rfloor_+,
\end{equation}
then there exists an operator $\widetilde T\in\mathcal L(\dot A^{s,\tau}_{p,q}(W))$
that agrees with $T$ on $\mathcal S\cap\dot A^{s,\tau}_{p,q}(W)$.
\end{enumerate}
\end{theorem}

Notice that case \eqref{T1ext} involves an extra assumption
compared to the case \eqref{T1test} only when $\widetilde J-n-\widetilde s<0$.
We are now in a position to describe the conditions for
mapping of atoms into $(K,L,M,N)$-molecules for any prescribed tuple of parameters.
Once we have this result, we can then obtain criteria
for $\dot A^{s,\tau}_{p,q}(d)$-molecules
and eventually the $\dot A^{s,\tau}_{p,q}(W)$-boundedness (Theorem \ref{T1 BF} above),
as relatively quick corollaries further below.

\begin{proposition}\label{T1EFGH}
Let $\sigma\in\{0,1\}$, $ E,F,G,H\in\mathbb{R}$,
and $ K,L,M,N\in\mathbb{R}$.
Suppose that $T\in\operatorname{CZO}^\sigma(E,F,G,H)$.
Then $T$ maps sufficiently regular atoms to $(K,L,M,N)$-molecules provided that the following conditions are satisfied:
\begin{align*}
\sigma&\geq\mathbf{1}_{(0,\infty)}(N),\
\begin{cases}
E\geq N,\\
E>\lfloor N\rfloor_+,
\end{cases}
\begin{cases}
F\geq (K\vee M)-n,\\
F>\lfloor L\rfloor,
\end{cases}
G\geq\lfloor N\rfloor_+,
\end{align*}
and $H\geq\lfloor L\rfloor$.
\end{proposition}

\begin{proof}
Without essential loss of generality,
we may consider atoms and molecules on $Q_{0,\mathbf{0}}$ only.
Let $\mathcal K\in\mathcal S'(\mathbb R^n\times\mathbb R^n)$
be the Schwartz kernel of $T$.
Let $a$ be an $(L_a,N_a)$-atom, where
$$
L_a\in[\lfloor\!\lfloor F\rfloor\!\rfloor\vee\lfloor\!\lfloor F-E\rfloor\!\rfloor,\infty)
\text{ and }
N_a\in[(\lfloor\!\lfloor E\rfloor\!\rfloor\wedge\lfloor G\rfloor)+1,\infty).
$$

\textbf{Size and derivatives:}\quad
We need to estimate $\partial^\alpha Ta$ for any
$\alpha\in\mathbb{Z}_+^n $ with $|\alpha|\leq\lfloor\!\lfloor N\rfloor\!\rfloor_+$,
where $\lfloor\!\lfloor N\rfloor\!\rfloor $ is the same as in \eqref{ceil}.
When $\alpha=\mathbf{0}$, this includes in particular the size bound for $Ta$.

By the vanishing moments of the atom $a$ (with $L_a \geq \lfloor\!\lfloor F\rfloor\!\rfloor $), the support condition $\operatorname{supp} a\subset 3Q_{0,\mathbf{0}}$,
and Lemma \ref{CZK Taylor} with $R$ replaced by $ 2 \sqrt{n}$
[which only assumes $\mathcal K\in\operatorname{CZK}^0(E;F)$,
a condition that is always valid here],
we find that, for any $\alpha\in\mathbb{Z}_+^n$ with $|\alpha|\leq\lfloor\!\lfloor E\rfloor\!\rfloor$
and for any $x\in\mathbb{R}^n$ with $|x|>4\sqrt{n}$,
\begin{align}\label{T1mole0}
|\partial^\alpha Ta(x)|
&=\left|\int_{\mathbb R^n}\partial_x^\alpha \mathcal K(x,y)a(y)\,dy\right| \\
&=\left|\int_{\mathbb R^n}\left\{\partial^\alpha_x\mathcal K(x,y)
-\left[\operatorname{Tayl}_{\mathbf{0}}^{\lfloor\!\lfloor F-|\alpha|\rfloor\!\rfloor}
\partial_x^\alpha\mathcal K(x,\cdot)\right](y)\right\} a(y)\,dy\right|\notag\\
&\lesssim\int_{3Q_{0,\mathbf{0}}}|x|^{-n-F}\,dy
\sim|x|^{-n-F}.\notag
\end{align}

Using $T\in\operatorname{CZO}^\sigma(E,F,G,H)$, we conclude that
$T\in\operatorname{WBP}$ and $T\in\operatorname{CZK}^\sigma(E;F)$,
where $\operatorname{WBP}$
is the same as in Definition \ref{WBP}(i).
From $T\in\operatorname{CZK}^\sigma(E;F)$, we deduce that $T\in\operatorname{CZO}(E)$, where $\operatorname{CZO}(\cdot)$ is the same as in Definition \ref{WBP}(ii).
Note that \cite[Lemma 3.1.12]{tor}
says that the assumptions that $T\in\operatorname{WBP}$,
$T\in\operatorname{CZO}(\ell+\varepsilon)$
with $\ell\in\mathbb{Z}_+$ and $\varepsilon\in(0,1)$,
and $T(y^\gamma)=0$ for any $\gamma\in\mathbb{Z}_+^n$ with $|\gamma|\leq\ell$ imply that,
for any $\alpha\in\mathbb{Z}_+^n$ with $|\alpha|\leq\ell$
and for any $x_0\in\mathbb{R}^n$, $r\in(0,\infty)$, and
$\varphi\in\mathcal{D}:=C_{\mathrm{c}}^\infty$ with $\operatorname{supp}\varphi\subset B(x_0,r)$,
$$
\|\partial^\alpha T\varphi\|_{L^\infty}
\lesssim\sum_{\gamma\in\mathbb{Z}_+^n,\,|\gamma|\leq|\alpha|+1} r^{|\gamma|-|\alpha|}\|\partial^\gamma\varphi\|_{L^\infty},
$$
where the implicit positive constant depends only on $T$.
These assumptions hold with the values
$\ell=\lfloor\!\lfloor E\rfloor\!\rfloor\wedge\lfloor G\rfloor,$
$\varepsilon\in(0,E^{**}),$ $x_0=\mathbf{0},$ $r=4\sqrt{n},$ and
$\varphi=a,$ where $ E^{**}$ is the same as in \eqref{r**}.
Hence, for any $\alpha\in\mathbb{Z}_+^n $ with $|\alpha|\leq\lfloor\!\lfloor E\rfloor\!\rfloor\wedge\lfloor G\rfloor$, we obtain the bound
$$
\|\partial^\alpha Ta\|_{L^\infty}
\lesssim\max_{|\gamma|\leq|\alpha|+1}\|\partial^\gamma a\|_{L^\infty}
\leq\max_{|\gamma|\leq\lfloor\!\lfloor E\rfloor\!\rfloor\wedge\lfloor G\rfloor+1}\|\partial^\gamma a\|_{L^\infty}.
$$
Combined with the uniform boundedness of the derivatives of the atom $a$
[with $N_a\geq(\lfloor\!\lfloor E\rfloor\!\rfloor\wedge\lfloor G\rfloor)+1]$,
this further implies that $\|\partial^\alpha Ta\|_{L^\infty}\lesssim1$.
Using this, together with \eqref{T1mole0}, we obtain,
for each $\alpha\in\mathbb{Z}_+^n$ with $|\alpha|\leq\lfloor\!\lfloor E\rfloor\!\rfloor\wedge \lfloor G\rfloor$
and for each $x\in\mathbb R^n$,
\begin{equation}\label{T1mole1}
|\partial^\alpha Ta(x)|\lesssim(1+|x|)^{-n-F}.
\end{equation}
In particular, we have the molecular size and derivative estimates that,
for each $\alpha\in\mathbb{Z}_+^n$ with $|\alpha|\leq\lfloor\!\lfloor N\rfloor\!\rfloor$
and for each $x\in\mathbb{R}^n$,
\begin{equation}\label{T1mole2}
|Ta(x)|\lesssim(1+|x|)^{-K},\
|\partial^\alpha Ta(x)|\lesssim(1+|x|)^{-M}
\end{equation}
provided that
\begin{equation}\label{T1cond1}
E>\lfloor\!\lfloor N\rfloor\!\rfloor_+,\
F\geq(K\vee M)-n,\text{ and }
G\geq\lfloor\!\lfloor N\rfloor\!\rfloor_+,
\end{equation}
where the positive parts are needed to guarantee that
at least the size bound with $\alpha=\mathbf{0}$ is always included in \eqref{T1mole1}.

\textbf{Cancellation:}\quad
For the cancellation $\int_{\mathbb R^n}x^\gamma Ta(x)\,dx=0$
when $\gamma\in\mathbb{Z}_+^n$ with $|\gamma|\leq L$ to make sense,
we need to prove that $(1+|x|)^{\lfloor L\rfloor}Ta(x)$ is integrable.
By the size bound just established in \eqref{T1mole1}, this will be the case if $F>\lfloor L\rfloor$.
From $T\in\operatorname{CZK}^\sigma(E;F)$ and \eqref{CZK and CZO}, we infer that $T^*\in\operatorname{CZO}(F)$.
The assumption that $T^*\in\operatorname{CZO}(F)$ for some $F>\lfloor L\rfloor$
also ensures that $T^*(x^\gamma)$ for any $\gamma\in\mathbb{Z}_+^n$ with $|\gamma|\leq L$ is well defined
as in Definition \ref{10.10}.
Since the atom $a$ may be taken to have as many vanishing moments as we like,
it belongs to $\mathcal D_{\lfloor L\rfloor}$,
where $\mathcal D_{\lfloor L\rfloor}$ is the same as in Lemma \ref{tor2.2.12}. Thus,
letting $\{\phi_j\}_{j\in\mathbb{N}}$ be the same as in Lemma \ref{tor2.2.12},
by the Lebesgue dominated convergence theorem with dominated function
$\sup_{j\in\mathbb{N}}\|\phi_j\|_{L^\infty}|x^\gamma Ta(x)|$
and Definition \ref{10.10}, we conclude that,
for each $\gamma\in\mathbb{Z}_+^n$ with $|\gamma|\leq L$,
\begin{align*}
\int_{\mathbb R^n}x^\gamma Ta(x)\,dx
&=\lim_{j\to\infty}\int_{\mathbb R^n}\phi_j(x)x^\gamma Ta(x)\,dx
=\lim_{j\to\infty}\langle\phi_jx^\gamma,Ta\rangle\\
&=\lim_{j\to\infty}\langle T^*(\phi_j x^\gamma),a\rangle
=\langle T^*(x^\gamma),a\rangle=0
\end{align*}
provided that $ F>\lfloor L\rfloor $ and $ H\geq \lfloor L\rfloor $.

\textbf{Differences:}\quad
Finally, we need to estimate $\partial^\alpha Ta(x)-\partial^\alpha Ta(x+h)$
for any $x,h\in\mathbb R^n$ and $\alpha\in\mathbb{Z}_+^n$
with $|\alpha|=\lfloor\!\lfloor N\rfloor\!\rfloor$.
There is nothing to do if $N\leq0$, so we may in the remainder of the
present proof assume that $N>0$,
and hence that $K\in\operatorname{CZK}^1(E;F)$
[rather than the weaker $K\in\operatorname{CZK}^0(E;F)$].

Recall that we have already checked the bound \eqref{T1mole2} and hence,
for any $x,h\in\mathbb R^n$ with $|h|\geq1$, we have
\begin{align*}
\left|\partial^\alpha Ta(x)-\partial^\alpha Ta(x+h)\right|
&\lesssim(1+|x|)^{-M}+(1+|x+h|)^{-M}\\
&\lesssim\sup_{z\in\mathbb R^n,\,|z|\leq|h|}(1+|x+z|)^{-M}
\leq|h|^{N^{**}}\sup_{z\in\mathbb R^n,\,|z|\leq|h|}(1+|x+z|)^{-M}.
\end{align*}
Thus, we concentrate on $|h|<1$.
Suppose first that $\lfloor\!\lfloor N\rfloor\!\rfloor<\lfloor\!\lfloor E\rfloor\!\rfloor$.
For each $i\in\{1,\ldots,n\}$, let
$e_i:=(0,\ldots,0,1,0,\ldots,0)\in\mathbb{Z}_+^n$
where only the $i$-th component is $ 1 $.
By the already checked bound \eqref{T1mole0} applied to each $\alpha+e_i$
(with $|\alpha+e_i|=|\alpha|+1=\lfloor\!\lfloor N\rfloor\!\rfloor+1\leq\lfloor\!\lfloor E\rfloor\!\rfloor$) in place of $\alpha$
and by $n+F\geq M$, we find that,
for any $x,h\in\mathbb{R}^n$ with $|x|>2$ and $|h|<1$,
\begin{align}\label{**}
|\partial^\alpha Ta(x)-\partial^\alpha Ta(x+h)|
&=\left|\int_0^1h\nabla\partial^\alpha Ta(x+th)\,dt\right|\\
&\lesssim|h|\int_0^1\max_{i\in\{1,\ldots,n\}}|\partial^{\alpha+e_i}Ta(x+th)|\,dt\notag\\
&\lesssim|h|\int_0^1|x+th|^{-M}\,dt
\lesssim|h|^{N^{**}}|x|^{-M}.\notag
\end{align}

Since \eqref{T1cond1} implies that $\lfloor\!\lfloor N\rfloor\!\rfloor\leq\lfloor\!\lfloor E\rfloor\!\rfloor\wedge\lfloor G\rfloor$,
it follows that the only other possibility is that
$\lfloor\!\lfloor N\rfloor\!\rfloor=\lfloor\!\lfloor E\rfloor\!\rfloor$.
Then, for any $|\alpha|=\lfloor\!\lfloor N\rfloor\!\rfloor=\lfloor\!\lfloor E\rfloor\!\rfloor$,
we again write out the kernel representation,
subtract a Taylor polynomial [recalling the vanishing moments of
the atom $a$ (with $L_a\geq\lfloor\!\lfloor F-E\rfloor\!\rfloor$)],
and apply Lemma \ref{CZK Taylor2} to arrive at the following,
for any $x,h\in\mathbb{R}^n$ with $|x|>4\sqrt{n}+2$ and $|h|<1$:
\begin{align*}
&\left|\partial^\alpha Ta(x)-\partial^\alpha Ta(x+h)\right|\\
&\quad=\left|\int_{\mathbb R^n}[\partial_x^\alpha\mathcal K(x,y)
-\partial_x^\alpha\mathcal K(x+h,y)] a(y)\,dy\right|\\
&\quad=\left|\int_{\mathbb R^n}\left[\left(I-\operatorname{Tayl}_0^{\lfloor\!\lfloor F-E\rfloor\!\rfloor}\right)
(\partial_x^\alpha\mathcal K(x,\cdot)
-\partial_x^\alpha\mathcal K(x+h,\cdot))\right](y)a(y)\,dy\right|\\
&\quad\lesssim\int_{3Q_{0,\mathbf{0}}}|h|^{E^{**}}|x|^{-n-F}\,dy
\sim|h|^{E^{**}}|x|^{-n-F}.
\end{align*}
Notice that, by $|h|<1$, $|h|^{E^{**}}\leq|h|^{N^{**}}$ if and only if $E^{**}\geq N^{**}$
and, since we are considering the case that $\lfloor\!\lfloor N\rfloor\!\rfloor=\lfloor\!\lfloor E\rfloor\!\rfloor$,
it follows that the last inequality is equivalent to $E\geq N$.
As before, for any $x\in\mathbb{R}^n$ with $|x|>4\sqrt{n}$,
we obtain $|x|^{-n-F}\lesssim|x|^{-M}$ provided that $F\geq M-n$.
Altogether, we obtain the required estimate that,
for any $\alpha\in\mathbb{Z}_+^n$ with $|\alpha|=\lfloor\!\lfloor N\rfloor\!\rfloor$ and
for any $x,h\in\mathbb{R}^n$ with $|x|>4\sqrt{n}+2$ and $|h|<1$,
\begin{equation*}
\left|\partial^\alpha Ta(x)-\partial^\alpha Ta(x+h)\right|
\lesssim|h|^{N^{**}}|x|^{-M}
\end{equation*}
provided that, in addition to conditions already imposed earlier,
we have $E\geq N$, which subsumes the earlier condition that
$E>\lfloor\!\lfloor N\rfloor\!\rfloor_+$ when $N>0$.

For a general $x\in\mathbb R^n$, we apply \cite[Lemma 3.1.22]{tor}.
Assuming that $T\in\operatorname{WBP}\cap\operatorname{CZO}(\ell+\varepsilon)$
and $T(y^\eta)=0$ for each $|\eta|\leq\ell$ [which holds, by our assumptions,
with $\ell=\lfloor\!\lfloor E\rfloor\!\rfloor\wedge\lfloor G\rfloor$ and $\varepsilon\in(0, E^{**}\wedge1)$],
the said lemma provides the following representation:
for any $\alpha\in\mathbb{Z}_+^n$ with $|\alpha|\leq\ell$
and for any $x,h\in\mathbb{R}^n$ with $|h|<1$, the following identity is valid:
\begin{align}\label{tor3.1.22}
&\partial^\alpha Ta(x+h)-\partial^\alpha Ta(x)\\
&\quad=\int_{\mathbb{R}^n}\partial_x^\alpha\mathcal K(x+h,y)
\left(I-\operatorname{Tayl}_{x+h}^{|\alpha|}\right)a(y)
\chi\left(\frac{y-x}{|h|}\right)\,dy\notag\\
&\qquad-\int_{\mathbb{R}^n}\partial_x^\alpha\mathcal K(x,y)
\left(I-\operatorname{Tayl}_{x}^{|\alpha|}\right)a(y)
\chi\left(\frac{y-x}{|h|}\right)\,dy\notag\\
&\qquad+\int_{\mathbb{R}^n} [\partial_x^\alpha\mathcal K(x+h,y)
-\partial_x^\alpha \mathcal K(x,y)]
\left(I-\operatorname{Tayl}_x^{|\alpha|}\right)
a(y)\left[1-\chi\left(\frac{y-x}{|h|}\right)\right]\,dy\notag\\
&\qquad+\sum_{|\theta|\leq|\alpha|}\frac{1}{\theta!}
\left(I-\operatorname{Tayl}_x^{|\alpha|-|\theta|}\right)\partial^{\theta}a(x+h)
T_{\theta,\alpha}\left[\chi\left(\frac{\cdot-x}{|h|}\right)\right](x+h)\notag\\
&\quad=: \mathrm{I}+\mathrm{II}+\mathrm{III}+\mathrm{IV},\notag
\end{align}
where $\chi\in\mathcal{D}$ is a fixed function
with $\mathbf{1}_{B(\mathbf{0},2)}\leq\chi\leq\mathbf{1}_{B(\mathbf{0},4)}$
and $T_{\theta,\alpha}\in\mathcal L(\mathcal S,\mathcal S')$
is the operator determined by the distributional kernel
$(y-x)^\theta\partial_x^\alpha \mathcal K(x,y)$.
By \cite[Lemma 3.1.12]{tor}, under the same assumptions, this operator $T_{\theta,\alpha}$ satisfies,
for any $\phi\in\mathcal{D}$ with
$\operatorname{supp}\phi\subset B(x_0,t)$,
\begin{equation}\label{estimate}
\|T_{\theta,\alpha}\phi\|_{L^\infty}
\lesssim\sum_{\eta\in\mathbb{Z}_+^n,\,|\eta|\leq|\alpha|-|\theta|+1}
t^{|\eta|-|\alpha|+|\theta|}\|\partial^\eta\phi\|_{L^\infty}
\end{equation}
if $|\theta|\leq|\alpha|\leq\ell$, which agrees with the summation condition in $\mathrm{IV}$.

As before, we take $\ell=\lfloor\!\lfloor E\rfloor\!\rfloor\wedge\lfloor G\rfloor$ and then,
using \eqref{estimate}, we obtain,
for any $\theta,\alpha\in\mathbb{Z}_+^n$ with
$|\theta|\leq|\alpha|\leq\ell$ and for any $x,h\in\mathbb{R}^n$ with $|h|<1$,
\begin{align*}
\left\|T_{\theta,\alpha}\left[\chi\left(\frac{\cdot-x}{|h|}\right)\right] \,\right\|_{L^\infty}
\lesssim\sum_{\eta\in\mathbb{Z}_+^n,\,|\eta|\leq|\alpha|-|\theta|+1}
|h|^{|\eta|-|\alpha|+|\theta|}|h|^{-\eta}\|\partial^\eta\chi\|_{L^\infty}
\lesssim|h|^{|\theta|-|\alpha|}.
\end{align*}
From Lemma \ref{Taylor} and the uniform boundedness of the derivatives of the atom $a$
up to order $ N_a \geq (\lfloor\!\lfloor E\rfloor\!\rfloor\wedge \lfloor G\rfloor)+1 $, we deduce that
\begin{align*}
\left|\left(I-\operatorname{Tayl}_x^{|\alpha|-|\theta|}\right)\partial^{\theta}a(x+h)\right|
\lesssim|h|^{|\alpha|-|\theta|+1}\sup_{\beta\in\mathbb{Z}_+^n,\,|\beta|=|\alpha|-|\theta|+1}
\left\|\partial^{\beta+\theta}a\right\|_{L^\infty}
\leq|h|^{|\alpha|-|\theta|+1}.
\end{align*}
Altogether we obtain $|\mathrm{IV}|\lesssim|h|$.

Similarly, for any $x,y\in\mathbb{R}^n$ with $|x-y|\in(0,4)$,
$|(I-\operatorname{Tayl}_x^{|\alpha|})a(y)|
\lesssim|x-y|^{|\alpha|+1}$
and, by \eqref{CZK0} of Definition \ref{CZK},
we have $|\partial_x^\alpha\mathcal K(x,y)|\lesssim|x-y|^{-n-|\alpha|}$.
Substituting these estimates into $\mathrm{II}$ and using
$\chi\leq\mathbf{1}_{B(\mathbf{0},4)}$, we conclude that,
for any $h\in\mathbb{R}^n$ with $|h|\in(0,1)$,
\begin{align*}
|\mathrm{II}|
&\lesssim\int_{|x-y|<4|h|}|x-y|^{-n-|\alpha|}|x-y|^{|\alpha|+1}\,dy\\
&=\int_{|x-y|<4|h|}|x-y|^{-n+1}\,dy
\sim\int_0^{4|h|} t^{-n+1}t^{n-1}\,dt
\sim|h|.
\end{align*}
The estimation of $\mathrm{I}$ is essentially similar with $x+h$ in place of $x$.
(The symmetry of $x$ and $x+h$ is not complete,
but it is easy to check that this does not cause any trouble.)

Let us finally consider the term $\mathrm{III}$.
Using the uniform boundedness of the derivatives of the atom $a$
up to order $N_a\geq(\lfloor\!\lfloor E\rfloor\!\rfloor\wedge\lfloor G\rfloor)+1$
and Lemma \ref{Taylor},
we have the alternative bound that, for any $x,y\in\mathbb{R}^n$ with $x\neq y$,
$|(I-\operatorname{Tayl}_x^{|\alpha|})a(y)|
\lesssim|x-y|^{|\alpha|+1}.$
By the uniform boundedness of the derivatives of the atom $a$ again,
we obtain the alternative bound that, for any $x,y\in\mathbb{R}^n$ with $x\neq y$,
\begin{align*}
\left|\left(I-\operatorname{Tayl}_x^{|\alpha|}\right)a(y)\right|
&\leq|a(y)|+\sum_{\beta\in\mathbb{Z}_+^n,\,|\beta|\leq|\alpha|}
\frac{|(y-x)^\beta|}{\beta!}|\partial^\beta a(x)|\\
&\lesssim1+\sum_{k=0}^{|\alpha|}|y-x|^k
\sim1+|y-x|^{|\alpha|}.
\end{align*}

Recalling that we need to estimate \eqref{tor3.1.22} for $|\alpha|=\lfloor\!\lfloor N\rfloor\!\rfloor$,
we also make two alternative estimate for
$\partial_x^\alpha\mathcal K(x+h,y)-\partial_x^\alpha\mathcal K(x,y)$.

If $\lfloor\!\lfloor N\rfloor\!\rfloor<\lfloor\!\lfloor E\rfloor\!\rfloor$,
then any $\alpha+e_i$ still satisfies $|\alpha+e_i|\leq\lfloor\!\lfloor E\rfloor\!\rfloor$
and hence we can apply \eqref{CZK0} from Definition \ref{CZK} to obtain,
for any $x,y,h\in\mathbb{R}^n$ with $|h|<1$ and $|x-y|>2|h|$,
\begin{align*}
|\partial_x^\alpha \mathcal K(x+h,y)-\partial_x^\alpha\mathcal K(x,y)|
&=\left|\int_0^1h\cdot\nabla_x\partial_x^\alpha\mathcal K(x+th,y)\,dt\right|\\
&\lesssim|h|\sup_{ t\in(0,1)}|x+th-y|^{-n-|\alpha|-1}
\sim|h|\,|x-y|^{-n-|\alpha|-1},
\end{align*}
noticing that the factor $1-\chi(\frac{y-x}{|h|})$ in $\mathrm{III}$ ensures that $|x-y|>2|h|$.

Recalling that $\lfloor\!\lfloor N\rfloor\!\rfloor\leq\lfloor\!\lfloor E\rfloor\!\rfloor\wedge\lfloor G\rfloor$,
the only other possibility is that $\lfloor\!\lfloor N\rfloor\!\rfloor=\lfloor\!\lfloor E\rfloor\!\rfloor$
and, in this case, we apply \eqref{CZKx} from Definition \ref{CZK} to find that,
for any $x,y,h\in\mathbb{R}^n$ with $|h|<1$ and $|x-y|>2|h|$,
\begin{equation*}
|\partial_x^\alpha \mathcal K(x+h,y)-\partial_x^\alpha\mathcal K(x,y)|
\lesssim|h|^{E^{**}}|x-y|^{-n-E}
\end{equation*}
if $\alpha\in\mathbb{Z}_+^n$ with $|\alpha|=\lfloor\!\lfloor N\rfloor\!\rfloor=\lfloor\!\lfloor E\rfloor\!\rfloor$.
Thus, in general, we have a bound of the form that,
for any $x,y,h\in\mathbb{R}^n$ with $|h|<1$ and $|x-y|>2|h|$,
\begin{equation*}
|\partial_x^\alpha\mathcal K(x+h,y)-\partial_x^\alpha\mathcal K(x,y)|
\lesssim|h|^\delta|x-y|^{-n-|\alpha|-\delta},
\end{equation*}
where, since $|x-y|>2|h|$, one can take $\delta\in(0,1]$ if $\lfloor\!\lfloor N\rfloor\!\rfloor<\lfloor\!\lfloor E\rfloor\!\rfloor$
and take $\delta\in(0,E^{**}] $ if $\lfloor\!\lfloor N\rfloor\!\rfloor=\lfloor\!\lfloor E\rfloor\!\rfloor$.

Substituting the above obtained estimates into $\mathrm{III}$, we obtain
\begin{align*}
|\mathrm{III}|
&\lesssim\int_{2|h|<|x-y|\leq2}|h|^{\delta}|x-y|^{-n-|\alpha|-\delta} |x-y|^{|\alpha|+1}\,dy
+\int_{|x-y|>2}|h|^{\delta}|x-y|^{-n-|\alpha|-\delta}|x-y|^{|\alpha|}\,dy\\
&\lesssim|h|^{\delta}\int_{2|h|}^2 t^{-n-\delta+1}t^{n-1}\,dt
+|h|^{\delta}\int_2^\infty t^{-n-\delta}t^{n-1}\,dt\\
&\lesssim|h|^{\delta}
\begin{cases}
1&\text{if }\delta\in(0,1),\\
1+\log\frac{1}{|h|}&\text{if }\delta=1,
\end{cases}
\end{align*}
and we would like to have this bounded by $|h|^{N^{**}}$.
Notice that all other terms $\mathrm{I}$, $\mathrm{II}$, and $\mathrm{IV}$
were already estimated by $|h|\leq|h|^{N^{**}}$.

If $N^{**}\in(0,1)$ (equivalently, $N\notin\mathbb Z$),
the required bound can be achieved by choosing $\delta=N^{**}$,
which is always possible if $\lfloor\!\lfloor N\rfloor\!\rfloor<\lfloor\!\lfloor E\rfloor\!\rfloor$ and which requires $E^{**}\geq N^{**}$
or equivalently $E\geq N$ if $\lfloor\!\lfloor N\rfloor\!\rfloor=\lfloor\!\lfloor E\rfloor\!\rfloor$.
Thus, we have showed that,
for any $\alpha\in\mathbb{Z}_+^n$ with $|\alpha|=\lfloor\!\lfloor N\rfloor\!\rfloor$
and for any $x,h\in\mathbb{R}^n$ with $|h|<1$,
$|\partial^\alpha Ta(x+h)-\partial^\alpha Ta(x)|
\lesssim|h|^{N^{**}}.$
In combination with the estimate \eqref{**}, we have proved that,
for any $\alpha\in\mathbb{Z}_+^n$ with $|\alpha|=\lfloor\!\lfloor N\rfloor\!\rfloor$
and for any $x,h\in\mathbb{R}^n$ with $|h|<1$,
\begin{equation*}
|\partial^\alpha Ta(x+h)-\partial^\alpha Ta(x)|
\lesssim|h|^{N^{**}}|x|^{-M}
\lesssim|h|^{N^{**}}(1+|x|)^{-M},
\end{equation*}
provided that $E\geq N\notin\mathbb Z,$ $F\geq M-n,$ and
$G\geq\lfloor\!\lfloor N\rfloor\!\rfloor.$
In combination with the size, the derivative, and the cancellation conditions,
we conclude that $Ta$ is a $(K,L,M,N)$-molecule provided that
\begin{align}\label{NnotinZ}
\begin{cases}
E\geq N\notin\mathbb Z,\\
E>0,
\end{cases}
\begin{cases}
F\geq(K\vee M)-n,\\
F>\lfloor L\rfloor,
\end{cases}
G\geq\lfloor\!\lfloor N\rfloor\!\rfloor_+,\text{ and }
H\geq\lfloor L\rfloor.
\end{align}

Let us finally consider the case that $N\in\mathbb Z$.
In order to ensure that $Ta$ is a $(K,L,M,N)$-molecule,
it certainly suffices to ensure that it is a $(K,L,M,N+\nu)$-molecule for some $\nu\in(0,1)$.
Notice that in this case $N+\nu\notin\mathbb Z$ and the previous sufficient condition applies.
On $F$ and $H$, the conditions are unchanged from \eqref{NnotinZ};
on $E$ and $G$, we have
$E\geq N+\nu$ and $G\geq\lfloor\!\lfloor N+\nu\rfloor\!\rfloor_+=N_+=\lfloor N\rfloor_+$
and the condition that $E\geq N+\nu$, for some $\nu\in(0,1)$,
is equivalent to $E>N=\lfloor N\rfloor$.
Thus, conditions that cover both \eqref{NnotinZ}
and the discussed modification for any $N\in\mathbb Z$ can be stated as
\begin{align}\label{Ngeneral}
\begin{cases}
E\geq N,\\
E>\lfloor N\rfloor_+,
\end{cases}
\begin{cases}
F\geq(K\vee M)-n,\\
F>\lfloor L\rfloor,
\end{cases}
G\geq\lfloor N\rfloor_+,\text{ and }
H\geq\lfloor L\rfloor.
\end{align}
For any $N\notin\mathbb Z$ [equivalently, $N^{**}\in(0,1)$], we have $\lfloor\!\lfloor N\rfloor\!\rfloor=\lfloor N\rfloor$,
and the condition $E\geq N$ already implies that $E>\lfloor N\rfloor$,
so that we did not change the condition \eqref{NnotinZ}.
The last condition \eqref{Ngeneral} agrees with what we stated in the present proposition,
which completes the proof of Proposition \ref{T1EFGH}.
\end{proof}

We can now collect some useful consequences of the previous proposition.

\begin{corollary}\label{T1Js}
Let $\sigma\in\{0,1\}$, $E,F,G,H\in\mathbb{R}$,
and $T\in\operatorname{CZO}^\sigma(E,F,G,H)$.
Then $T$ maps sufficiently regular atoms to $(J,s)$-molecules provided that
\begin{align}\label{T1Js1}
\sigma&\geq\mathbf{1}_{[0,\infty)}(s),\
E>s_+,\
F>J-n+s_-,\
G\geq\lfloor s\rfloor_+,\text{ and }
H\geq\lfloor J-n-s\rfloor.
\end{align}
\end{corollary}

\begin{proof}
By Definition \ref{def Js mol} of $(J,s)$-molecules,
we need to show that $T$ maps sufficiently regular atoms
to $(K,L,M,N)$-molecules for some
\begin{equation}\label{T1EJ1}
K>J+s_-,\
L\geq J-n-s,\
M>J,\text{ and }
N>s.
\end{equation}
Using Proposition \ref{T1EFGH}, we conclude that the operator $T$
has this property provided that
\begin{align}\label{T1EJ2}
\sigma&\geq\mathbf{1}_{(0,\infty)}(N),\
\begin{cases}E\geq N,\\E>\lfloor N\rfloor_+,\end{cases}
\begin{cases}F\geq(K\vee M)-n,\\F>\lfloor L\rfloor,\end{cases}
G\geq\lfloor N\rfloor_+,\text{ and }
H\geq\lfloor L\rfloor.
\end{align}

We can choose numbers $K,L,M\in\mathbb{R}$ such that
$K\in(J+s_-,F+n),$ $L\in[J-n-s,\lceil\!\lceil J-n-s\rceil\!\rceil),$ and
$M\in(J,F+n),$
where $\lceil\!\lceil\cdot\rceil\!\rceil $ is the same as in \eqref{ceil}
because the intervals are non-empty by the assumptions \eqref{T1Js1}. We also choose
\begin{equation}\label{T1EJ4}
\begin{cases}
N:=0\in(s,E)&\text{if }s<0,\\
N\in(s,\lceil\!\lceil s\rceil\!\rceil\wedge E)&\text{if }s\geq 0,
\end{cases}
\end{equation}
where the interval for $N$ is non-empty by the assumptions \eqref{T1Js1} again
and the definition of the strict rounding-up function $\lceil\!\lceil x\rceil\!\rceil:=\min\{k\in\mathbb Z:\ k>x\}$ for any $x\in\mathbb R$.
These choices clearly satisfy \eqref{T1EJ1},
so it remains to check that they also satisfy \eqref{T1EJ2}.

First, by \eqref{T1EJ4}, we have
$\sigma\geq\mathbf{1}_{[0,\infty)}(s)=\mathbf{1}_{(0,\infty)}(N)$
and $\lfloor N\rfloor=\lfloor N\rfloor_+=\lfloor s\rfloor_+$
in each case, and hence $G\geq\lfloor s\rfloor_+=\lfloor N\rfloor_+$
satisfies the desired estimate. We also check that
$E>N\geq\lfloor N\rfloor_+$, $F>(K\vee M)-n$, and
\begin{equation*}
\lfloor L\rfloor=\lfloor J-n-s\rfloor
\leq\begin{cases}\lfloor J-n+s_-\rfloor<F,\\ H.\end{cases}
\end{equation*}
Thus, all parameters satisfy their respective conditions in \eqref{T1EJ2}.

From Proposition \ref{T1EFGH} and \eqref{T1EJ2} that we just verified,
we infer that $T$ maps sufficiently regular atoms to $(K,L,M,N)$-molecules
and hence to $(J,s)$-molecules because the parameters were chosen to satisfy \eqref{T1EJ1}.
This finishes the proof of Corollary \ref{T1Js}.
\end{proof}

\begin{corollary}\label{T1Amol}
Let $s\in\mathbb R$, $\tau\in[0,\infty)$, $p\in(0,\infty)$, $q\in(0,\infty]$,
$d\in[0,n)$, and $\widetilde J,\widetilde s$ be the same as in \eqref{tauJ2}.
Then $T\in\operatorname{CZO}^\sigma(E,F,G,H)$ maps sufficiently regular atoms
to $\dot A^{s,\tau}_{p,q}(d)$-synthesis molecules provided that \eqref{T1Amol1} is satisfied.
\end{corollary}

\begin{proof}
By Corollary \ref{T1Js}, we conclude that the assumptions \eqref{T1Amol1} imply that
$T$ maps sufficiently regular atoms to $(\widetilde J,\widetilde s)$-molecules.
From Definition \ref{def A mol}, it follows that,
for both $\widetilde J$ and $\widetilde s$ in \eqref{tauJ2},
these molecules are the same as $\dot A^{s,\tau}_{p,q}(d)$-synthesis molecules.
This finishes the proof of Corollary \ref{T1Amol}.
\end{proof}

Next, we can prove Theorem \ref{T1 BF}.

\begin{proof}[Proof of Theorem \ref{T1 BF}]
We first show \eqref{T1test}.
The present assumptions are the same as those of Corollary \ref{T1Amol},
which guarantee that $T$ maps sufficiently regular atoms to
$\dot A^{s,\tau}_{p,q}(d)$-synthesis molecules.
Then Proposition \ref{ext}\eqref{ext1} provides the existence of
an operator $\widetilde T\in\mathcal L(\dot A^{s,\tau}_{p,q}(W))$ as claimed.

Now, we prove \eqref{T1ext}.
Under the assumptions \eqref{T1Amol2}, we observe that $E,F>0$ and $G,H\geq 0$,
which mean in particular that $T\in\operatorname{WPB}\cap\operatorname{CZO}(E)$,
$T^*\in\operatorname{CZO}(F)$ for $E,F>0$, as well as $T(1)=0=T^*(1)$.
These are the assumptions of the original $T(1)$ theorem of David and Journ\'e \cite{dj84}
and hence the said theorem implies the existence of $T_2\in\mathcal L(L^2)$
that agrees with $T$ on $\mathcal S$.
By Proposition \ref{ext}\eqref{ext2},
we can construct $\widetilde T$ with the additional property as claimed.
This finishes the proof of Theorem \ref{T1 BF}.
\end{proof}

\subsection{Comparison with Earlier Results}\label{T1 compare}

As it turns out, Theorem \ref{T1 BF} improves and extends the earlier
related results of \cite{ftw88,tor} but,
due to subtle differences in the technical details of the formulation,
this might not be entirely obvious at a glance.
We dedicate this section to a case by case justification of this claim,
showing that Theorem \ref{T1 BF} reduces to earlier results in many cases,
but also provides a genuine improvement in some others.

\begin{remark}\label{T1 rem1}
Let us write down the assumptions of Theorem \ref{T1 BF} under particular sets of parameter values.
\begin{enumerate}[\rm(i)]
\item\label{ftw 3.1} When $\widetilde s<0$, assumptions \eqref{T1Amol1} become
$\sigma\geq0,$ $E>0,$ $F>\widetilde J-n+(\widetilde s)_-,$
$G\geq0,$ and $H\geq\left\lfloor\widetilde J-n-\widetilde s\right\rfloor\geq0,$
which coincide with \eqref{T1Amol2} in this case.
Since $E$ is only required to satisfy $E>0$
and the assumptions under consideration become weaker with decreasing $E$,
we may without loss of generality assume that $0<E<F\wedge 1$.
Then Remark \ref{CZK cases}\eqref{factor2} proves that the assumptions under consideration are equivalent to
$$
\begin{cases}
T\in\operatorname{CZO}(E)&\text{with }E>0,\\
T^*\in\operatorname{CZO}(F)&\text{with }F>\widetilde J-n-\widetilde s
\end{cases}\ \mathrm{and}\
\begin{cases}
T(1)=0,&\\
T^*(x^\theta)=0&\text{if }\theta\in\mathbb Z_+^n
\text{ with }|\theta|\leq\left\lfloor\widetilde J-n-\widetilde s\right\rfloor.
\end{cases}
$$

\item\label{ftw 3.7} When $\widetilde s\geq0$ and $\widetilde J=n$, 
assumptions \eqref{T1Amol1} become
\begin{equation*}
\sigma\geq 1,\
E>\widetilde s,\
F>0,\
G\geq\lfloor\widetilde s\rfloor,\text{ and }
\begin{cases}
H\geq 0&\text{if }\widetilde s=0,\\
\text{(void)}&\text{if }\widetilde s>0,
\end{cases}
\end{equation*}
which coincide with \eqref{T1Amol2} if $\widetilde s=0$,
but otherwise $H\geq 0$ needs to be additionally imposed in \eqref{T1Amol2}.
Since $F$ is only required to satisfy $F>0$
and the assumptions under consideration become weaker with decreasing $F$,
we may without loss of generality assume that $0<F<E\wedge1$.
Then Remark \ref{CZK cases}\eqref{factor1} shows that the assumptions under consideration are equivalent to
$$
\begin{cases}
T\in\operatorname{CZO}(E)&\text{with }E>\widetilde{s},\\
T^*\in\operatorname{CZO}(F)&\text{with }F>0
\end{cases}\ \mathrm{and}\
\begin{cases}
T(y^\gamma)=0&\text{if }\gamma\in\mathbb Z_+^n
\text{ with }|\gamma|\leq\lfloor\widetilde s\rfloor,\\
T^*(1)=0&\text{if }\widetilde s=0.
\end{cases}
$$

\item\label{ftw 3.13a} If $\widetilde s\geq\widetilde J-n>0$,
assumptions \eqref{T1Amol1} as such do not substantially simplify but,
taking $E=\widetilde{s}+\varepsilon$ and $F=\widetilde J-n+\varepsilon$,
we have $F-E=\widetilde J-n-\widetilde{s}\leq0$
and hence assumptions \eqref{CZKxy} in Definition \ref{CZK} become void for these parameters.
If $\varepsilon>0$ is small enough, we have
$\lfloor\!\lfloor E\rfloor\!\rfloor=\lfloor \widetilde s\rfloor,$
$E^{**}=(\widetilde s)^*+\varepsilon=:\delta,$
$\lfloor\!\lfloor F\rfloor\!\rfloor=\lfloor \widetilde J\rfloor-n,$ and
$F^{**}=(\widetilde J)^*+\varepsilon=:\rho,$
where $\lfloor\!\lfloor E\rfloor\!\rfloor $ and $ E^{**}$
are the same as, respectively, in \eqref{ceil} and \eqref{r**}.
Thus, \eqref{CZK0} and \eqref{CZKx} say that
$T\in\operatorname{CZO}(\lfloor \widetilde s\rfloor+\delta)$ and
$\delta>(\widetilde s)^*$, and \eqref{CZKy} becomes
\begin{align*}
\left|\partial_x^\alpha\partial_y^\beta K(x,y)
-\partial_x^\alpha\partial_y^\beta K(x,y+v)\right|
\lesssim|v|^{\rho}|x-y|^{-n-|\alpha|-|\beta|-\rho}\quad\text{if}\quad
\begin{cases}|\alpha|\leq\lfloor\widetilde s\rfloor,\\
|\beta|=\lfloor \widetilde J\rfloor-n-|\alpha|.\end{cases}
\end{align*}
\end{enumerate}
For $\tau=0$, in which case $\widetilde J=J$ and $\widetilde s=s$,
the assumptions in \eqref{ftw 3.1} coincide with
those in \cite[Theorem 3.1]{ftw88} or \cite[Theorem 3.2.1]{tor},
and the assumptions in \eqref{ftw 3.7} with
those in \cite[Theorem 3.7]{ftw88} or \cite[Theorem 3.2.7]{tor},
and assumptions in \eqref{ftw 3.13a} with
those in \cite[Theorem 3.13]{ftw88} or \cite[Theorem 3.2.13]{tor}
in the special case that $J-n-s\leq 0$.
Thus, we obtain an extension of their results to matrix-weighted spaces
$\dot A^s_{p,q}(W)$ under exactly the same assumptions that they have
in the unweighted case of $\dot F^s_{p,q}$.
For $\tau>0$, our result also gives a natural extension
by replacing $(J,s)$ with $(\widetilde J,\widetilde s)$.
\end{remark}

Let us finally consider the case when $\widetilde J-n>\widetilde s\geq 0$.
In this case, the $\operatorname{CZK}(E;F)$ assumption,
with parameters in \eqref{T1Amol1},
does not factorise into separate assumptions on $T$ and $T^*$ as above,
but we need to also consider the mixed smoothness assumptions in Definition \ref{CZK}.
Our formulation of these assumptions is slightly different from \cite{ftw88,tor},
which makes the comparison more delicate than the cases covered in Remark \ref{T1 rem1}.
For the sake of comparison, we first state the following result
which is the case $J-n-s>0$ of either \cite[Theorem 3.13]{ftw88} or \cite[Theorem 3.2.13]{tor}.

\begin{theorem}\label{T1 ftw3.13}
Let $s\in\mathbb R$, $p\in(0,\infty)$, $q\in(0,\infty]$,
and $J:=\frac{n}{\min\{1,p,q\}}$ satisfy
$J-n>s\geq 0.$ Let $T\in\mathcal L(\mathcal S,\mathcal S)$
with Schwartz kernel $\mathcal K\in\mathcal S'(\mathbb R^n\times\mathbb R^n)$ satisfy
$T\in\operatorname{WBP},$
$$\begin{cases}
T(y^\gamma)=0&\text{if }\gamma\in\mathbb{Z}_+^n\text{ and }|\gamma|\leq\lfloor s\rfloor,\\
T^*(x^\theta)=0&\text{if }\theta\in\mathbb{Z}_+^n\text{ and }|\theta|\leq\lfloor J-n-s\rfloor
\end{cases}
\ and\
\begin{cases}T\in\operatorname{CZO}(\lfloor s\rfloor+\delta),\\
T^*\in\operatorname{CZO}(\lfloor J-n\rfloor+\rho),\end{cases}
$$
where
\begin{equation}\label{ftw 3.15}
\delta>\max(s^*,J^*),\quad \rho>J^*,
\end{equation}
and, moreover, there exists a positive constant $C$ such that,
for any $\alpha,\beta\in\mathbb{Z}_+^n$ and $x,y,u,v\in\mathbb{R}^n$ with
$|u|,|v|<\frac12|x-y|$,
the following estimates hold\footnote{In \cite[(3.20)]{ftw88},
the exponent $\rho$ should be $\delta$ instead
because otherwise the assumption does not scale correctly with respect to dilations,
and the same assumption was stated correctly in \cite[(3.2.20)]{tor}.}
\begin{align}\label{ftw 3.19}
\left|\partial_x^\alpha\partial_y^\beta\mathcal K(x,y)
-\partial_x^\alpha\partial_y^\beta\mathcal K(x,y+v)\right|
&\leq C|v|^\rho|x-y|^{-n-|\alpha|-|\beta|-\rho}\ \text{if}\ \begin{cases}
|\alpha|\leq\lfloor s\rfloor,\\
|\beta|=\lfloor J-n\rfloor-|\alpha|
\end{cases}
\end{align}
and
\begin{align}\label{ftw 3.20}
\left|\partial_x^\alpha\partial_y^\beta \mathcal K(x,y)
-\partial_x^\alpha\partial_y^\beta \mathcal K(x+u,y)\right|
\leq C|u|^\delta|x-y|^{-n-|\alpha|-|\beta|-\delta}
&\ \text{if }\begin{cases}
|\alpha|=\lfloor s\rfloor,\\
|\beta|=\lfloor J-n\rfloor-\lfloor s\rfloor.
\end{cases}
\end{align}
Then there exists an operator $\widetilde T\in\mathcal L(\dot F^{s}_{p,q}(W))$
that agrees with $T$ on $\mathcal{S}_\infty$.
\end{theorem}

Under the same assumptions on the parameters,
our Theorem \ref{T1 BF} involves a mixed difference assumption \eqref{CZKxy}
that does not appear in Theorem \ref{T1 ftw3.13}.
With both $E=\widetilde s+\varepsilon$ and $F=\widetilde J-n+\varepsilon+\eta$
as required by \eqref{T1Amol1} for any $\widetilde s\geq0$
and with sufficiently small $\varepsilon,\eta\in(0,1)$, we have
$\lfloor\!\lfloor E\rfloor\!\rfloor=\lfloor\widetilde s\rfloor,$
$E^{**}=(\widetilde s)^*+\varepsilon,$
$\lfloor\!\lfloor F-E\rfloor\!\rfloor=\lfloor\widetilde J-n-\widetilde s\rfloor,$
and
$(F-E)^{**}=(\widetilde J-\widetilde s)^*+\eta,$
where $\lfloor\!\lfloor E\rfloor\!\rfloor $ and $ E^{**}$
are the same as, respectively, in \eqref{ceil} and \eqref{r**},
and hence \eqref{CZKxy} takes the form that, for any $x,y,u,v\in\mathbb R^n$ with
$|u|+|v|<\frac12|x-y|$,
\begin{align}\label{CZKxyJs}
&\left|\partial_x^\alpha\partial_y^\beta\mathcal K(x,y)
-\partial_x^\alpha\partial_y^\beta\mathcal K(x+u,y)
-\partial_x^\alpha\partial_y^\beta\mathcal K(x,y+v)
+\partial_x^\alpha\partial_y^\beta\mathcal K(x+u,y+v)\right|\\
&\quad\leq C|u|^{(\widetilde s)^*+\varepsilon}
|v|^{(\widetilde J-\widetilde s)^*+\eta}|x-y|^{-\widetilde J-\varepsilon-\eta}\notag
\end{align}
if $\alpha,\beta\in\mathbb Z_+^n$ satisfy $|\alpha|=\lfloor\widetilde s\rfloor$
and $|\beta|=\lfloor\widetilde J-n-\widetilde s\rfloor$.

On the other hand, we have the following conclusion.

\begin{lemma}\label{CZKmix}
Under the assumptions of Theorem \ref{T1 ftw3.13}, the estimate \eqref{CZKxyJs} holds
with $J$ and $s$, respectively, in place of $\widetilde J$ and $\widetilde s$
and with some $\varepsilon,\eta\in(0,\infty)$.
\end{lemma}

\begin{proof}
Note that, for any $a,b\in\mathbb R$, $a-b=\lfloor a\rfloor-\lfloor b\rfloor+(a^*-b^*)$,
where $a^*-b^*\in(-1,1)$, and hence
$\lfloor a-b\rfloor\in\lfloor a\rfloor-\lfloor b\rfloor+\{0,-1\}.$
Let us first suppose that $\lfloor J-n-s\rfloor=\lfloor J-n\rfloor-\lfloor s\rfloor$,
so that the assumptions on $\beta$ in \eqref{CZKxyJs} agree with those in \eqref{ftw 3.20}
and in \eqref{ftw 3.19} when $|\alpha|=\lfloor s\rfloor$.
The assumption $\lfloor J-n-s\rfloor=\lfloor J-n\rfloor-\lfloor s\rfloor$
is equivalent to $J^*\geq s^*$, which, combined with \eqref{ftw 3.15}, further implies that
in this case $\kappa:=\min\{\delta,\rho\}$
satisfies $\kappa>J^*\geq s^*$ and \eqref{ftw 3.19} and \eqref{ftw 3.20}
hold with $\kappa$ in place of both $\delta$ and $\rho$.

By the triangle inequality and \eqref{ftw 3.19} with $|\alpha|=\lfloor s\rfloor$,
we find that, for any $x,y,u,v\in\mathbb{R}^n$ with $|u|,|v|<\frac12|x-y|$,
\begin{align*}
\operatorname{LHS}\text{ of }\eqref{CZKxyJs}
&\leq\left|\partial_x^\alpha\partial_y^\beta\mathcal K(x,y)
-\partial_x^\alpha\partial_y^\beta\mathcal K(x,y+v)\right|\\
&\quad+\left|\partial_x^\alpha\partial_y^\beta\mathcal K(x+u,y)
-\partial_x^\alpha\partial_y^\beta\mathcal K(x+u,y+v)\right|\\
&\lesssim|v|^{\kappa}|x-y|^{-n-\lfloor J-n\rfloor-\kappa}+|v|^{\kappa}|x+u-y|^{-n-\lfloor J-n\rfloor-\kappa}
\sim|v|^{\kappa}|x-y|^{-n-\lfloor J-n\rfloor-\kappa}
\end{align*}
and, similarly, by \eqref{ftw 3.20}, we then have,
for any $x,y,u\in\mathbb{R}^n$ with $|u|<\frac12|x-y|$,
\begin{equation*}
\operatorname{LHS}\text{ of }\eqref{CZKxyJs}
\lesssim|u|^\kappa|x-y|^{-n-\lfloor J-n\rfloor-\kappa}.
\end{equation*}
Writing $\kappa>J^*\geq s^*$ as $\kappa=J^*+\varepsilon+\eta$,
where $\varepsilon,\eta\in(0,\infty)$,
and $J^*=s^*+(J^*-s^*)=s^*+(J-s)^*,$
and taking the minimum of the two upper bounds, we conclude that
\begin{align*}
\operatorname{LHS}\text{ of }\eqref{CZKxyJs}
\lesssim\min(|u|,|v|)^{s^*+\varepsilon+(J-s)^*+\eta}|x-y|^{-n-\lfloor J-n\rfloor -J^*-\varepsilon-\eta}
\lesssim|u|^{s^*+\varepsilon}|v|^{(J-s)^*+\eta}|x-y|^{-J-\varepsilon-\eta},
\end{align*}
which confirms \eqref{CZKxyJs} in this case.

Suppose then that $\lfloor J-n-s\rfloor=\lfloor J-n\rfloor-\lfloor s\rfloor-1$;
thus, assumption \eqref{ftw 3.20} now involves derivatives of order $|\beta|=\lfloor J-n-s\rfloor+1$.
For any $ i\in\{1,\ldots,n\}$, let
$ e_i:=(0,\ldots,0,1,0,\ldots,0)\in\mathbb{Z}_+^n$
where only the $i$-th component is $1$.
For any $|\beta|=\lfloor J-n-s\rfloor$, it then follows from
the mean value theorem and \eqref{ftw 3.20} that,
for any $x,y,u,v\in\mathbb{R}^n$ with $|u|,|v|<\frac12|x-y|$,
\begin{align*}
\operatorname{LHS}\text{ of }\eqref{CZKxyJs}
&=\left|\int_0^1v\nabla_y\partial_x^\alpha\partial_y^\beta
(\mathcal K(x,y+tv)-\mathcal K(x+u,y+tv))\,dt \right|\\
&\lesssim|v|\sum_{i=1}^n\int_0^1\left|\partial_x^\alpha
\partial_y^{\beta+e_i}(\mathcal K(x,y+tv)-\mathcal K(x+u,y+tv))\right|\,dt\\
&\lesssim|v|\int_0^1|u|^\delta|x-y-tv|^{-n-\lfloor J-n\rfloor-\delta}\,dt\\
&\sim|v|\,|u|^\delta|x-y|^{-n-\lfloor J-n\rfloor-\delta}
=|u|^\delta|v|\,|x-y|^{-\lfloor J\rfloor-\delta}.
\end{align*}
Here, we write $\delta>s^*$ as $\delta=s^*+\varepsilon$.
We now have $J^*<s^*$, and hence $(J-s)^*=1+J^*-s^*$.
For any $\eta\in(0,s^*-J^*)$, by $1-(J-s)^*-\eta=s^*-J^*-\eta>0$, we obtain,
for any $x,y,v\in\mathbb{R}^n$ with $|v|<\frac12|x-y|$,
\begin{equation*}
|v|\leq|v|^{(J-s)^*+\eta}|x-y|^{1-(J-s)^*-\eta}
=|v|^{(J-s)^*+\eta}|x-y|^{s^*-J^*-\eta}.
\end{equation*}
Substituting all these, we conclude that, for any $x,y,u,v\in\mathbb R^n$ with
$|u|+|v|<\frac12|x-y|$,
\begin{align*}
\operatorname{LHS}\text{ of }\eqref{CZKxyJs}
&\lesssim|u|^\delta|v|\,|x-y|^{-\lfloor J\rfloor-\delta}\\
&\lesssim|u|^{s^*+\varepsilon}|v|^{(J-s)^*+\eta}|x-y|^{s^*-J^*-\eta}|x-y|^{-\lfloor J\rfloor-s^*-\varepsilon}\\
&\sim|u|^{s^*+\varepsilon}|v|^{(J-s)^*+\eta}|x-y|^{-J-\varepsilon-\eta},
\end{align*}
confirming \eqref{CZKxyJs} also in this case,
which completes the proof of Lemma \ref{CZKmix}.
\end{proof}

\begin{corollary}\label{229}
When $\tau=0$, Theorem \ref{T1 BF} in this case contains \cite[Theorems 3.1, 3.7, and 3.13]{ftw88}
or equivalently \cite[Theorems 3.2.1, 3.2,7, and 3.2.13]{tor}
about the boundedness of Calder\'on--Zygmund operators on $\dot A^s_{p,q}$.
Moreover, under the same assumptions, these operators
remain bounded on $\dot A^s_{p,q}(W)$ for any $W\in A_p$.
For any $\tau\in(0,\infty)$, sufficient conditions for the boundedness of
Calder\'on--Zygmund operators on $\dot A^{s,\tau}_{p,q}(W)$
are obtained by replacing $(J,s)$ with $(\widetilde J,\widetilde s)$ in \eqref{tauJ2}.
\end{corollary}

\begin{proof}
We have already discussed the cases
$\widetilde{s}<0$ or $\widetilde{s}\geq\widetilde{J}-n\geq0$
in Remark \ref{T1 rem1}.
In the remaining case that $\widetilde{J}-n>\widetilde{s}\geq 0$,
the only difference is that the mixed difference assumption \eqref{CZKxy}
is assumed in Theorem \ref{T1 BF}, but apparently not in \cite[Theorem 3.13]{ftw88}
or \cite[Theorem 3.2.13]{tor} (restated as Theorem \ref{T1 ftw3.13}).
However, by Lemma \ref{CZKmix}, this assumption with appropriate parameters
follows from the assumptions of these theorems.
This finishes the proof of Corollary \ref{229}.
\end{proof}

\begin{remark}
In the case that $J-n>s\geq 0$, our Theorem \ref{T1 BF}
is an essential improvement of Theorem \ref{T1 ftw3.13} from \cite{ftw88,tor}.
We have already seen that the assumptions of Theorem \ref{T1 BF}
follow from those of Theorem \ref{T1 ftw3.13}.
The fact that they are strictly weaker is most easily seen
by comparing the smoothness assumption in the $x$ variable only.
Theorem \ref{T1 BF} assumes that $T\in\operatorname{CZO}(E)$ for $E>s$,
whereas Theorem \ref{T1 ftw3.13} requires that
$T\in\operatorname{CZO}(\lfloor s\rfloor+\delta)$ with $\delta>\max(s^*,J^*)$
which is a strictly stronger assumption when $J^*>s^*$.

For $m=1$ and $W\equiv 1$, the space $\dot A^{s,\tau}_{p,q}(W)$ reduces to $\dot A^{s,\tau}_{p,q}$,
and Theorem \ref{T1 BF} in this case is still new.
When $\tau=0$, we have $\dot A^{s,\tau}_{p,q}(W)=\dot A^{s}_{p,q}(W)$.
Theorem \ref{T1 BF} in this case is new for $\dot F^s_{p,q}(W)$
and improves \cite[Theorem 9.4]{ro03} and \cite[Theorem 4.2(i)]{fr04} for $\dot B^s_{p,q}(W)$
because the aforementioned theorems only consider special cases and need stronger assumptions.

Besides being a stronger (and more general) result,
a further advantage of Theorem \ref{T1 BF} is the fact that
the Calder\'on--Zygmund smoothness assumptions take a simpler version
without reference to the integer and the fractional part functions
(in contrast to the condition just discussed).
\end{remark}

\bigskip

\noindent Fan Bu

\medskip

\noindent Laboratory of Mathematics and Complex Systems (Ministry of Education of China),
School of Mathematical Sciences, Beijing Normal University, Beijing 100875, The People's Republic of China

\smallskip

\noindent{\it E-mail:} \texttt{fanbu@mail.bnu.edu.cn}

\bigskip

\noindent Tuomas Hyt\"onen

\medskip

\noindent Department of Mathematics and Statistics,
University of Helsinki, (Pietari Kalmin katu 5), P.O.
Box 68, 00014 Helsinki, Finland

\smallskip

\noindent{\it E-mail:} \texttt{tuomas.hytonen@helsinki.fi}

\smallskip

\noindent (Address as of 1 Jan 2024:) Department of Mathematics 
and Systems Analysis, Aalto University, P.O. Box 11100, FI-00076 Aalto, Finland

\noindent{\it E-mail:} \texttt{tuomas.p.hytonen@aalto.fi}

\bigskip

\noindent Dachun Yang (Corresponding author) and Wen Yuan

\medskip

\noindent Laboratory of Mathematics and Complex Systems (Ministry of Education of China),
School of Mathematical Sciences, Beijing Normal University, Beijing 100875, The People's Republic of China

\smallskip

\noindent{\it E-mails:} \texttt{dcyang@bnu.edu.cn} (D. Yang)

\noindent\phantom{{\it E-mails:}} \texttt{wenyuan@bnu.edu.cn} (W. Yuan)

\end{document}